\documentclass[10pt]{amsart}
\usepackage{amsmath,amssymb, enumerate,amsthm,amsfonts, enumitem, graphicx, comment}

\usepackage{times}
\usepackage{xypic}

\usepackage[OT2,T1]{fontenc}
\usepackage{hyperref}
\usepackage{color}
\usepackage{comment}
\DeclareSymbolFont{cyrletters}{OT2}{wncyr}{m}{n}
\DeclareMathSymbol{\Sha}{\mathalpha}{cyrletters}{"58}

\newtheorem{thm}{Theorem}[section]
\newtheorem{lem}[thm]{Lemma}
\newtheorem{prop}[thm]{Proposition}

\newtheorem{cor}[thm]{Corollary}

\newtheorem{example}{Example}
\newtheorem*{example*}{Example}

\newtheorem{lemma}[thm]{Lemma}
\newtheorem{proposition}[thm]{Proposition}
\newtheorem*{proposition*}{Proposition}
\newtheorem{dff}[thm]{Definition}

\newtheorem*{mainprop*}{Proposition \ref{finitepointsoncurves}}

\numberwithin{equation}{section}

\newcommand{\QQ}{\mathbb{Q}}

\newcommand{\dR}{\mathrm{dR}}
\newcommand{\an}{\mathrm{an}}
\newcommand{\prim}{\mathrm{prim}}
\newcommand{\GL}{{\operatorname{GL}}}

\newcommand{\SL}{{\operatorname{SL}}}
\newcommand{\Sym}{\mathrm{Sym}}

\newcommand{\Z}{{\mathbf{Z}}}
\newcommand{\F}{{\mathbf{F}}}
\newcommand{\Q}{{\mathbf{Q}}}
\newcommand{\Aff}{\mathrm{Aff}}
\newcommand{\C}{{\mathbf{C}}}

\renewcommand{\P}{{\mathbf{P}}}

\newcommand{\G}{\mathbf{G}}
\newcommand{\Spec}{{\operatorname{Spec}}}

\newcommand{\Frob}{\mathrm{Frob}}

\newcommand{\five}{8} %
\newcommand{\Gal}{{\operatorname{Gal}}}

\newcommand{\p}{{\mathfrak{p}}}
\newcommand{\LGr}{\mathrm{LGr}}

\newcommand{\Pic}{{\operatorname{Pic}}}
\newcommand{\Hom}{{\operatorname{Hom}}}
\newcommand{\inft}{\mathbf{C}}
\newcommand{\cris}{\mathrm{cris}}

\newcommand{\Res}{\mathrm{Res}}
\newcommand{\Aut}{{\operatorname{Aut}}}

\newcommand{\lmatrix}{\left(\!\!\begin{array}{cc}}
\newcommand{\rmatrix}{\end{array}\!\!\right)}
\newcommand{\lvector}{\left(\!\!\begin{array}{c}}
\newcommand{\Sp}{\mathrm{Sp}}
\newcommand{\et}{\mathrm{et}}
\newcommand{\rvector}{\end{array}\!\!\right)}
\usepackage[OT2,T1]{fontenc}
\DeclareSymbolFont{cyrletters}{OT2}{wncyr}{m}{n}
\DeclareFontFamily{OT1}{rsfs}{}
\DeclareFontEncoding{OT2}{}{} %
  
     \DeclareFontShape{OT1}{rsfs}{n}{it}{<-> rsfs10}{}
\DeclareMathAlphabet{\mathscr}{OT1}{rsfs}{n}{it}

\newcommand{\Hodgeestimatebound}{40}

\newcommand{\HPr}{H_1^\mathrm{Pr}}

\newcommand{\numcovers}{N}
\newcommand{\monmap}{\operatorname{Cov}} %
\newcommand{\monsp}{\operatorname{Mon}}

\newcommand{\Gq}{\operatorname{Aff}(q)}

\newcommand{\bigexp}{M}

\newcommand{\arith}{\mathrm{arith}}

\newcommand{\crys}{\mathrm{crys}}
\newcommand{\Jmid}{\mathsf{m}}
\renewcommand{\ss}{\mathrm{ss}}
\newcommand{\wt}{\mathrm{wt}}
\newcommand{\codim}{\mathrm{codim}}

\newif\iffriendly
\friendlytrue

\begin{document}

\title{Diophantine problems and $p$-adic period mappings}
\author{Brian Lawrence and Akshay Venkatesh}
\maketitle
\setcounter{tocdepth}{1}
\begin{abstract}
 
We give an alternative  proof of Faltings's theorem (Mordell's conjecture): a curve of genus at least two over a number field has finitely many rational points. 
Our argument utilizes the set-up of Faltings's original proof, but is in spirit closer to the methods of Chabauty and Kim: we replace the use of abelian varieties
by a more detailed analysis of the variation of $p$-adic Galois representations in a family of algebraic varieties.  The key inputs
into this analysis are the comparison theorems of $p$-adic Hodge theory, and
explicit  topological computations of monodromy.

By the same methods we show that, 
 in sufficiently large dimension and degree, the set of hypersurfaces in projective
 space, with good reduction away from a fixed set of primes, is contained in a  proper
 Zariski-closed subset of the moduli space of all hypersurfaces.   This uses in an essential way the Ax--Schanuel property for period mappings, recently established by Bakker and Tsimerman.  

 \end{abstract}
\tableofcontents

\iffriendly
\else
\fi

\section{Introduction}
\label{intro}
\subsection{} 

Let $K$ be a number field. 
This paper has two main goals.  

Firstly,  we will give a new proof of 
the finiteness of $K$-rational points on a smooth projective $K$-curve of genus $\geqslant 2$. 
The proof is closely related to Faltings's proof \cite{Faltings}, but  is based on a closer study of the variation of  $p$-adic Galois representations in a family; it makes no usage of techniques specific to abelian varieties.  
  
Secondly,  we give an application of the same methods to a higher-dimensional situation.   Consider the family of degree-$d$ hypersurfaces in $\mathbf{P}^n$
and let $F_{n,d}$ be the complement of the discriminant divisor in this family; we regard $F_{n,d}$ as a smooth $\Z$-scheme.   For $S$ a finite set of primes, points of $F_{n,d}(\Z[S^{-1}])$ correspond
to proper smooth  hypersurfaces of degree $d$ in $\mathbf{P}^n_{\Z[S^{-1}]}$. 
It is very reasonable to suppose that $F_{n,d}(\Z[S^{-1}])$ is finite modulo the action of $\GL_{n+1}(\Z[S^{-1}])$ for $d \geqslant 3$ and all $n$.  We shall show at least that, 
if $n \geqslant n_0$ and $d \geqslant d_0(n)$, then $F_{n,d}(\Z[S^{-1}])$ is contained
in a proper Zariski closed subset of $F_{n,d}$ (i.e.,  there exists
a proper   
 $\Q$-subvariety of the generic fiber $(F_{n,d})_{\Q}$ whose rational points contain $F_{n,d}(\Z[S^{-1})$). 
 To prove this higher-dimensional result, we use a very recent theorem of Bakker and Tsimerman, the Ax--Schanuel theorem for period mappings.

We can obtain a still stronger theorem along a subvariety of $F_{n,d}$ if one has control over monodromy. Namely, if $F_{n,d}^* \subset (F_{n,d})_{\Q}$
is the Zariski closure of integral points, our result actually implies that the  Zariski closure of monodromy for the universal family of hypersurfaces
must drop over each component of $F_{n,d}^*$.    
It is possible that  this imposes a stronger codimension condition on $F_{n,d}^*$ than simply ``proper'' but we do not know for sure.
   
Note that, without the result of Bakker and Tsimerman,  one can {\em still} prove that $F_{n,d}(\Z[S^{-1}])$ lies  in a proper $\Q_p$-analytic subvariety of   
$F_{n,d}(\Q_p)$,
but one cannot prove the second statement about $F_{n,d}^*$. 
  
A simple toy case to illustrate the methods is given by the $S$-unit equation, which we analyze in \S \ref{Sunit}.

\subsection{Outline of the proof}   \label{subsec:outline}Consider a smooth projective family $X \rightarrow Y$ over $K$, where $Y$ is itself a smooth $K$-variety; we suppose this extends to a family $\pi: \mathcal{X} \rightarrow \mathcal{Y}$ over the ring $\mathcal{O}$ of $S$-integers of $K$, for some finite set $S$ of places of $K$ (containing all the archimedean places). 

For $y \in Y(K)$ call $X_y$ the fiber over $y$. We want to  bound $\mathcal{Y}(\mathcal{O})$,  making
use of the fact that, if $y \in Y(K)$ extends to $\mathcal{Y}(\mathcal{O})$, then $X_y$ admits a smooth proper model over $\mathcal{O}$. 
That one can thus reduce Mordell's conjecture to finiteness results  for varieties with good reduction was observed by Parshin \cite{Parshin} and then used by Faltings in his proof of the Mordell conjecture \cite{Faltings}. 

Choosing a rational prime $p$ that is unramified in $K$ and not below any prime of $S$, write $\rho_y$ for the Galois representation
of $G_K = \Gal(\overline{K}/K)$ on the $p$-adic geometric {\'e}tale cohomology of $X_y$, i.e.\ $H^*_{\et}(X_y \times_{K} \bar{K}, \Q_p)$. 
As observed by Faltings, one deduces from Hermite--Minkowski finiteness that, %
as $y$ varies through $\mathcal{Y}(\mathcal{O})$,  there are only finitely many possibilities
for the semisimplification of  the $G_K$-representation $\rho_y$ (denoted by $\rho_y^{\mathrm{ss}}$).

We   seek to use the fact that, for $v$ a place of $K$ above $p$, one can understand the restriction $\rho_{y,v}$ of $\rho_y$ to $G_{K_v}$  via $p$-adic Hodge theory.
In the Mordell case,  when $Y$ is a projective curve, our argument proceeds by showing that {\em both} of the following statements hold
  for suitable
choice of $X$ and $v$:
\begin{quote}
(*)  The representation $\rho_y$ is semisimple for all but finitely many $y \in Y(K)$, {\em and}  the map
\begin{equation} \label{rhorho0}
 y \in Y(K) \longrightarrow \mbox{isomorphism class of $\rho_{y,v}$} \end{equation}
has finite fibers.
\end{quote}
 
Faltings proves  much stronger statements when $X$ is an abelian scheme over $Y$, using a remarkable argument with heights:
every $\rho_y$ is semisimple {\em and } $\rho_y$ determines $X_y$ up to isogeny. 
Our approach gives less, but it gives results in other cases too, such as the hypersurface
family discussed above. However, in that setting, 
  the issue of semisimplicity proves harder to control,
and what we prove instead is the following hybrid of the two statements in (*): 
the map 
 \begin{equation} \label{rhorho} y \in Y(K) \longrightarrow \mbox{restriction of $\rho_y^{\mathrm{ss}}$ to $G_{K_v}$}\end{equation}
 considered as a mapping from $Y(K)$ to isomorphism classes of $G_{K_v}$-representations, has fibers that are not Zariski dense.  (It is crucial, in the above equation, that we semisimplify $\rho_y$ as a global Galois representation and  {\em then} restrict to $G_{K_v}$.) 

For the remainder of the current \S \ref{subsec:outline},  we will explain \eqref{rhorho0} in more detail.

Our analysis uses  $p$-adic Hodge theory. However we make no use of $p$-adic Hodge theory in families:  we need only the statements over a local field. Under the correspondence 
of $p$-adic Hodge theory, the restricted representation $\rho_{y,v}$ corresponds to a filtered $\phi$-module, namely the de Rham cohomology of $X_y$  
over $K_v$  equipped with its Hodge filtration and a semilinear Frobenius map. The variation of this filtration is
described by a period mapping; in this setting, this is a $K_v$-analytic mapping
 \begin{equation} \label{res map}  \mbox{ residue disk in $Y(K_v)$}  \longrightarrow \mbox{$K_v$-points of a flag variety,}\end{equation}
Therefore, the variation of the $p$-adic representation $\rho_{y,v}$ with $y$ is controlled by \eqref{res map}. 
The basic, and very naive,  ``hope'' of the proof is that injectivity of the period map   \eqref{res map}  
should force \eqref{rhorho0} to be injective.

However,  \eqref{rhorho0} does {\em not}  follow directly from injectivity of the period map, that is to say, from  Torelli-type theorems.

Different filtrations on the underlying $\phi$-module can give filtered $\phi$-modules which are abstractly isomorphic, the isomorphism being given by a linear endomorphism commuting with $\phi$.
Hence, one needs to know 
not only that the period mapping \eqref{res map} is injective,   but that its image   has {\em finite intersection} with an orbit on the period domain   of the centralizer $\mathrm{Z}(\phi)$ of $\phi$.  In other words, we must analyze a question of ``exceptional intersections'' between the image of a period
map and an algebraic subvariety.   

To illustrate how this is done, let us restrict to the case when $Y$ is a curve. 
Assuming that we have shown that the $\mathrm{Z}(\phi)$-orbit on the ambient flag
variety is a proper subvariety, it will then be sufficient to 
 show that the image of \eqref{res map} is in fact Zariski dense. Then the intersection points between the image of \eqref{res map} and a $\mathrm{Z}(\phi)$-orbit
 amount to zeroes of a nonvanishing $K_v$-analytic function in a residue disc, and are therefore finite.
 
 To check Zariski density, the crucial point is that one can verify the same statement for the {\em complex} period map:
\begin{equation} \label{topmap} \mbox{universal cover of $Y(\C)$} \longrightarrow \mbox{$\C$-points of a flag variety}\end{equation}
 To pass between the $p$-adic and complex period maps, we use the fact that (in suitable coordinates),
 they satisfy the same differential equation coming from the Gauss Manin connection, and so have the same power series.
 This  is a simple but crucial argument, given in Lemma \ref{vCpowerseries}.
 But  -- over the complex numbers -- Zariski density can be verified by topological methods: \eqref{topmap} is now equivariant for an action of $\pi_1(Y)$,
 acting on the right according to the monodromy representation.  It is enough to  verify that the  image of $\pi_1$ under the monodromy
 representation is sufficiently large. %
In the Mordell case, we show that the monodromy
action of $\pi_1(Y)$ extends to a certain mapping class group,
and we deduce large monodromy from the same assertion for the mapping class group (where we can use Dehn twists).
This monodromy argument is related to computations of Looijenga \cite{Looijenga}, Grunewald, Larsen, Lubotzky, and Malestein \cite{GLLM}, 
and Salter and Tshishiku \cite{ST}.

 If $Y$ were not a curve, the argument above says only that the intersection of the image of \eqref{res map}
 and a $\mathrm{Z}(\phi)$-orbit  is 
 a proper 
  $K_v$-analytic subvariety of $Y(K_v)$. 
One wants to get a proper {\em Zariski-closed subvariety}
(for example, this permits one, in principle at least, to  make an inductive argument on the dimension, although we do not try to do so here.) 
We obtain this only by  appealing to a remarkable recent result of Bakker and Tsimerman,
the Ax--Schanuel theorem for period mappings: this is a very powerful and general statement about the transcendence of period mappings.

To summarize, we have outlined the strategy of the proof of \eqref{rhorho0}. 
However, we have omitted one crucial ingredient needed in this proof, and also
 a crucial ingredient needed to get from \eqref{rhorho0} to Mordell:

\begin{itemize}
\item[(a)] Showing that the centralizer $\mathrm{Z}(\phi)$ of $\phi$ is not too large, and
\item[(b)]  Controlling in some {\em a priori} way the extent to which $\rho_y$ can fail to be semisimple. 
\end{itemize}
We now discuss these issues in turn. 
 
\subsection{Problem (a): controlling the centralizer of $\phi$} \label{control_centralizer}
 As we have explained, we need a method to ensure the centralizer of the crystalline Frobenius $\phi$ acting on the cohomology of a fiber $X_y$ is not too large.
 For example, if $K_v = \Q_p$ so that $\phi$ is simply a $\Q_p$-linear map, we must  certainly rule
 out the possibility that $\phi$ is a scalar!
 
This issue, that $\phi$ might have too large a centralizer and thus (*) might fail, already occurs in the simplest possible example. When analyzing the $S$-unit equation, it is natural to take   $Y = \P^1-\{0,1,\infty\}$  and $X \rightarrow Y$
to be the Legendre family, so that $X_t $ is the curve $y^2 = x(x-1)(x-t)$.  
Unfortunately  (*)  {\em fails}: for $t \in \Z_p$,  if we write $\rho_t$ for the representation of the Galois group  $G_{\Q_p}$ on the (rational) Tate module of $X_t$,
then $\rho_t$ belongs to only finitely many isomorphism classes so long as the reduction $\bar{t} \in \F_p$
is not equal to $0$ or $1$. 

Again we proceed in two different ways:
\begin{itemize}
\item[(i)]
In general, 
 Frobenius is a {\em semilinear} operator on a vector space over an unramified extension $L_w$ of $\Q_p$; semilinearity alone gives rise to a nontrivial bound (Lemma \ref{centralizer lemma}) on the size of its centralizer,
which, in effect, becomes stronger as $[L_w:\Q_p]$ gets larger.  

In the application to Mordell, it turns out that we can always put ourselves in a situation where $[L_w: \Q_p]$
is rather large. This forces the Frobenius centralizer to be small.  We explain this at more length below. 

\item[(ii)] In the case of hypersurfaces, we do not have a way to enlarge the base field as in (i). Our procedure is less satisfactory  than in case (i), 
in that it gives much weaker results:

We are of course able to choose the prime $p$, and we choose it (via Chebotarev) so that the crystalline Frobenius at $p$ has centralizer
that is as small as possible.    To do this, we fix an auxiliary prime $\ell$, and  first use the fact (from counting points over extensions of $\F_p$) that crystalline Frobenius at $p$
has the same eigenvalues as Frobenius on $p$ acting on $\ell$-adic cohomology; thus it is enough to choose $p$ such that the latter operator has small centralizer. 
One can do this via Chebotarev, given  a lower bound on the image of the {\em global} Galois representation,
and for this we again use some $p$-adic Hodge theory (cf.\ \cite{Sen}).  Another approach, by point-counting, is outlined in Lemma \ref{end lemma}. 

\end{itemize}

Let us explain point (i) above by example. In  our analysis  of the $S$-unit equation in \S \ref{Sunit}, we replace the Legendre family instead by the family
 with fiber 
 $$X_{t} = \coprod_{z^{2^k} = t} \{ y^2=x(x-1)(x-z) \},$$
 for a suitable large integer $k$. In our situation, the corresponding map $t \mapsto [\rho_t]$ will now only have finite fibers, at least on residue disks where $\bar{t}$ is not a square -- an example  
 of the importance of enlarging $K_v$. 
 
 Said differently, we have replaced the Legendre family $X  \stackrel{\ell}{\rightarrow} \P^1-\{0,1,\infty\}$
with a family with the following composite structure: 
 $$  X' \stackrel{\ell'}{ \rightarrow} \P^1-\{0,\mu_{2^k},\infty\} \rightarrow \P^1 - \{0,1,\infty\}$$ 
 where the second map is given by $u \mapsto u^{2^k}$, and $\ell'$ is simply the restriction of the Legendre family
 over $\P^1-\{0,\mu_{2^k}, \infty\}$.  The composite defines a family over $\P^1-\{0,1,\infty\}$ with geometrically disconnected fibres,
 and this disconnectedness is, as we have just explained, to our advantage. 
 
It turns out that the families introduced by Parshin (see \cite[Proposition 9]{Parshin}),
in his reduction of Mordell's conjecture to Shafarevich's conjecture, 
automatically have a similar structure. That is to say, if $Y$ is a smooth
projective curve, Parshin's families factorize as 
$$X \rightarrow Y' \rightarrow Y,$$
where $Y' \rightarrow Y$ is finite {\'e}tale and $X \rightarrow Y'$ is a relative curve.

 There is in fact a lot of flexibility in this construction; in Parshin's original construction 
the covering $Y' \rightarrow Y$ is obtained by pulling back multiplication by $2$ on the Jacobian,
and as such each geometric fiber is a torsor under  $H^1(Y_{\bar{K}},\mu_2)$.  
We want to ensure that the Galois action on each fiber of $Y' \rightarrow Y$ has large image -- with reference
to the discussion above, this is what allows us to ensure that the auxiliary field $L_w$ is  of large degree.
 We use a variant where each fiber  admits a $G_K$-equivariant map to $H^1(Y_{\bar{K}}, \Z/q\Z)$ (for a suitable auxiliary prime $q$).
 The Weil pairing  alone implies that the Galois action on this is nontrivial, and this (although very weak) is enough to run our argument.

\subsection{Problem (b): how to handle the failure of semisimplicity} \label{handle_nonsemisimple}
Let $y \in Y(K)$. 
The  local Galois representation $\rho_y|_{G_{K_v}}$ can certainly be very far from semisimple,
and thus we cannot hope to use $p$-adic Hodge theory alone to constrain semisimplicity. 

However, the Hodge weights of a global representation are highly constrained by purity (Lemma \ref{globalsimple}).
This means, for example, that any global subrepresentation $W$ of $\rho_y$ 
corresponds, under $p$-adic Hodge theory, to a Frobenius-stable subspace $W_{\dR} \subset H^*_{\dR}(X_y \otimes_{K} K_v)$
whose Hodge filtration is numerically constrained.  Now (assuming we have arranged that the Frobenius has 
small centralizer) there are not too many choices for a Frobenius-stable subspace; on the other hand, the Hodge filtration
varies as $y$ varies $p$-adically. Thus one can at least hope to show that
such a ``bad'' $W_{\dR}$ exists only for finitely many $y \in Y(K_v)$. In this way we can hope to show that $\rho_y$
is simple for all but finitely many $y$.

The purity 
argument is also reminiscent of an argument at the torsion level in Faltings's proof (the use of Raynaud's results on \cite[p. 364]{Faltings}).

We use this argument both for Mordell's conjecture and for hypersurfaces (although for hypersurfaces
we prove a much weaker result,  just bounding from above the failure of semisimplicity). 
The linear algebra involved is fairly straightforward for curves (see {\em Claim 1} and its proof in Section \ref{highergenus})
but becomes very unwieldy in the higher-dimensional case.  To handle it in a reasonably compact way we 
use some combinatorics related to reductive groups (\S \ref{GG combinatorics}).  However this argument is not very efficient and presumably gives
results that are far from optimal.

\subsection{Effectivity; comparison with Chabauty--Kim and Faltings}
It is of interest to compare our method with that of Chabauty,  and the nonabelian generalizations thereof due to Kim \cite{KimSiegel}.
 
Let $Y$ be a projective smooth curve over $K$ with Jacobian $J$. 
Fix a finite place $v$. 
The classical method of Chabauty proceeds by 
considering $Y(K)$ as the intersection of  global points $J(K)$ on the Jacobian and local points $Y(K_v)$ on the curve, inside
$J(K_v)$. If the rank of $J(K)$ is less than the $K_v$-dimension 
of $J$ (i.e.\ the genus of the curve) it is easy to see this intersection is finite.
 
We can reinterpret this cohomologically. Let $T_p$ be the $p$-adic Tate module of $J$, where $p$ is a prime below $v$.  
There is a Kummer map $J(K) \otimes \Q_p \rightarrow H^1(G_K, T_p)$ and 
we obtain a mapping
$$Y(K)   \longrightarrow H^1(G_K, T_p) = \mathrm{Ext}^1(\mathrm{trivial}, T_p),$$
which, explicitly speaking, sends $y \in Y(K)$ to the 
extension between the trivial representation and $T_p$ realized by cohomology of the punctured curve $H^1_{\et}(Y-\{y,y_0\})$ for a suitable basepoint $y_0$.
By this discussion, and its local analogue, we get a diagram
\begin{equation}    
\xymatrix{
Y(K)   \ar[d]  \ar[r]&   J(K) \ar[d] \ar[r]^{S \qquad \qquad \qquad \qquad} & \mbox{space of global Galois representations}   \ar[d]  \\  
Y(K_v)  \ar[r] & J(K_v) \ar[r]^{S_v \qquad \qquad \qquad} &  \mbox{space of local Galois representations}. 
 }
 \end{equation}
 (Here the global and local Galois representations are extensions of $T_p$ by the trivial representation.)
 Kim generalizes this picture, replacing
 $T_p$ by deeper quotients of $\pi_1(Y)$.   The idea of $p$-adic period mappings also plays a key role in his work,  
 see \cite[p.\ 360]{KimSiegel}, \cite[p.\ 93]{KimAlbanese},  \cite[Proposition 1.4]{KimTangential}.
 The key difficulty
 to be overcome is to obtain control over the size of the space of global Galois representations
 (e.g. the rank of $J(K)$).  
 
 Our picture is very much the same: we have a map $y \rightarrow \rho_y$
 from $Y(K)$ to global Galois representations.  
 In the story just described $\rho_y$ arises from the cohomology of an {\em open} variety --
the curve $Y$ punctured at $y$ and an auxiliary point. In the situation of our paper, $\rho_y$ will arise from the cohomology
 of a {\em smooth projective} variety -- a covering of $Y$ branched only at $y$. 
 
What does this gain? Our global Galois  representations
are now pure and (presumably) semisimple.  Therefore our space of 
global Galois representations should be extremely small.   On the other hand, what we
lose is that the map $S_v$ is now no longer obviously injective.

Kim has remarked to one of us (A.V.) that it would be of interest to consider  
combining these methods in some way, in particular that one might replace the role of the pro-unipotent
completion of $\pi_1(Y)$ in Kim's analysis by a {\em relative} completion. 

We expect that our method of proof can be made algorithmic
in the same sense as the method of Chabauty.  
For example, given a curve $C$ as above, 
one would be able to ``compute'' a finite subset $S \subset C(K_v)$ which contains $C(K)$; 
``compute'' means that there is an algorithm that will compute all the elements of $S$   to a specified $p$-adic precision in a finite time.
However, the resulting method is completely impractical, as we now explain. 
 
Firstly, our argument relies on Faltings's finiteness lemma for Galois representations (Lemma \ref{finiteness})
to give a finite list of possibilities for $\rho_y^{\mathrm{ss}}$.
We expect that Faltings's proof can easily be made algorithmic;
but there may be very, very many such representations.
 
Secondly, we would need to explicitly compute the comparisons furnished by $p$-adic Hodge theory.
For a given local Galois representation $\rho_y^{\mathrm{ss}}$,
we need to calculate to some finite precision
the filtered $\phi$-module associated to it by the crystalline comparison isomorphism
of $p$-adic Hodge theory.
We expect that this should be possible, but we are not aware
of any known algorithm to achieve this.
 
To conclude let us compare our method to Faltings's original proof.  That proof gives much more than ours does: it 
gives the full Shafarevich and Tate conjectures for abelian varieties, as well as semisimplicity of the associated Galois representation.
Our proof gives none of these;  it gives nothing about the Tate conjecture, and (at least without further effort)
it does not give the Shafarevich conjecture but only its restriction to a one-dimensional subfamily of moduli of abelian varieties.
Moreover, our proof is also in some sense more elaborate, since it requires the use of tricks and  delicate computations to avoid
the various complications that we have described. Its only real advantage in the Mordell case seems to be that it is in principle algorithmic in the sense described above.
In our view, the real gain of the method is the ability to apply it to 
families of higher-dimensional varieties. Our results about hypersurfaces are quite modest, but we regard them as a proof of concept for this idea.

\subsection{Structure of the paper}

\S \ref{notn} contains notation and preliminaries.

We suggest the reader start with \S \ref{fibers} and \S \ref{Sunit} to get a sense of the argument. 

\S \ref{fibers}  sets up the general formalism and the structure of the argument.  We relate Galois representations to  a $p$-adic period map using crystalline cohomology, and we connect the $p$-adic period map to a complex period map and monodromy.  The section ends with Proposition \ref{finitesetoforbits}, a preliminary form of our main result.

\S \ref{Sunit}  gives a first application: a proof of the $S$-unit theorem, using a variant of the Legendre family. 
This is much simpler than the proof of Mordell and can be considered a ``warm-up.''  

\S\S \ref{Mordelloutline} -- \ref{KP_monodromy} give the proof of the Mordell conjecture.
\S \ref{Mordelloutline} describes the strategy of the proof: we apply a certain refined version  of Proposition \ref{finitesetoforbits},
formulated as   Proposition \ref{finitepointsoncurves}, to a specific family of varieties that we call the Kodaira--Parshin family.   \S \ref{highergenus} is the proof of Proposition \ref{finitepointsoncurves}.  In particular this is where we take advantage of ``geometrically disconnected fibers''; the argument also deals with a technical issue relating to semisimplification.   
 In \S \ref{rationalpoints} we introduce the Kodaira-Parshin family and \S \ref{KP_monodromy}
 is purely topological: it computes the monodromy of the Kodaira--Parshin family.
 
\S\S \ref{hypersurface} --  \ref{bound_frob} study
 families of varieties of higher dimension. 
\S \ref{hypersurface} introduces a recent transcendence result of Bakker and Tsimerman which is needed to study families over a higher-dimensional base.  
\S \ref{Hodge_numbers_funny} proves the main result, Proposition \ref{transprop},  which shows that fibers of good reduction lie in a Zariski-closed subset of the base.  
The argument however invokes a ``general position'' result in linear algebra, Proposition \ref{linalg}, whose proof takes up
\S  \ref{GG combinatorics}. In \S \ref{bound_frob} we suggest an alternative argument, not used in the rest of the paper, to bound the size of the Frobenius centralizer.

\subsection{Acknowledgements}
This paper owes, of course, a tremendous debt to the work of Faltings -- indeed, 
all the main tools come from his work. 
Some of the ideas originated in a learning seminar run at Stanford University on Faltings's proof \cite{Faltings}. 

The 2017 Stanford PhD thesis \cite{BL} 
of B.L.\ contained an earlier version of the arguments of  this paper. In particular,
that thesis presented a proof of the Mordell conjecture conditional on an assumption about monodromy,
and verified that assumption for a certain Kodaira--Parshin family in genus $2$.

We thank Brian Conrad for many helpful conversations and suggestions. 
A.V.\ would like to thank Benjamin Bakker, Andrew Snowden and Jacob Tsimerman for interesting discussions. 
B.L.\ would like to thank Zeb Brady, Lalit Jain, Daniel Litt,  and Johan de Jong.

We received helpful comments and feedback from several people about earlier versions of this paper. We would like to thank, in particular, Dan Abramovich, Pedro A.\ Castillejo, Raymond Cheng, Brian Conrad, Ulrich Goertz, Sergey Gorchinskiy, Kiran Kedlaya, Aaron Landesman, Siyan Daniel Li, Lucia Mocz, Bjorn Poonen, Jack Sempliner, Will Sawin, and Bogdan Zavyalov. We similarly would like to thank the anonymous referee for his or her time and effort.

We thank Brian Conrad for pointing out the proof of Lemma \ref{Artss},
and for simplifying the proof of Lemma \ref{padic BT}.
We thank Jordan Ellenberg for an interesting discussion about monodromy.
 
During much of the work on this paper, B.L.\ was supported by a Hertz fellowship and an NSF fellowship
and A.V.\ was supported by an NSF grant. During the final stages of writing
A.V. was an Infosys member at the Institute for Advanced Study. We thank all these
organizations for their support of our work. 

\newcommand{\Ad}{\mathrm{Ad}}
\newcommand{\geom}{\mathrm{geom}}
\section{Notation and preparatory results}
\label{notn}

We gather here some notation and some miscellaneous lemmas that we will use in the text. 
We suggest that the reader refer to this section only as necessary when reading the main text.

The following notation will be fixed throughout the paper.
 
\begin{itemize}
\item $K$ a number field
\item $\overline{K}$ a fixed algebraic closure of $K$
\item $G_K = \Gal(\overline{K}/K)$ the absolute Galois group
\item $S$ a finite set of finite places of $K$ containing all the archimedean places
\item $\mathcal{O}_S$ the ring of $S$-integers
\item $\mathcal{O} = \mathcal{O}_S$ when $S$ is understood
\item $p$ a (rational) prime number such that no place of $S$ lies above $p$
\item $K_w$ the completion of $K$ at a prime $w$ of $\mathcal{O}$
\item $\overline{K}_w$ a fixed algebraic closure of $K_w$
\item $\F_w$ the residue field at $w$
\item $q_w$ the cardinality of $\F_w$
\item $\overline{\F}_w$ the residue field of $\overline{K}_w$, which is an algebraic closure of $\F_w$
\item $\mathcal{O}_{(w)}$ the localization of $\mathcal{O}$ at $w$
\end{itemize}
 
By a $G_K$-set we mean a (discretely topologized) set with a continuous action of $G_K$.
 
For a variety $X$ over a field $E$ of characteristic zero,  
we denote by $H^*_{\dR}(X/E)$ the de Rham cohomology of $X \rightarrow \Spec(E)$.
If $E' \supset E$ is a field extension, we denote by $H^*_{\dR}(X/E') $ the de Rham cohomology of the base-change $X_{E'}$,
which is identified with $H_{\dR}(X/E) \otimes_{E} E'$.
 
For any scheme $S$, 
a \emph{family over $S$} is an (arbitrary) $S$-scheme $\pi: Y \rightarrow S$. 
A \emph{curve over $S$} is a family over $S$ for which 
$\pi$ is smooth and proper of relative dimension $1$ and each geometric fiber is connected.   (Note that we will also make
use of ``open'' curves, for example in \S \ref{Sunit}, but we will avoid using the word ``curve'' in that context.) 
 
Let $E / \Q_p$ be a finite unramified extension of $\Q_p$, and $\sigma$ the unique automorphism of $E$ inducing the $p$-th power map on the residue field.  By \emph{$\phi$-module (over $E$)} we will mean a pair $(V, \phi)$, with $V$ a finite-dimensional $E$-vector space and $\phi: V \rightarrow V$ a map semilinear over $\sigma$.  A \emph{filtered $\phi$-module} will be a triple $(V, \phi, F^i V)$ such that $(V, \phi)$ is a $\phi$-module and $(F^i V)_i$ is a descending filtration on $V$.  We demand that each $F^i V$ be an $E$-linear subspace of $V$ but require no compatibility with $\phi$.   Note that the filtered $\phi$-modules arising from Galois representations via $p$-adic Hodge theory satisfy
a further condition, admissibility, but we will make no use of it in this paper (see \cite[Expos\'{e} III, \S 4.4]{Asterisque} and \cite[Expos\'{e} III, \S 5.3.3]{Asterisque}). %

\subsection{Linear algebra}
 
 \begin{lemma} \label{centralizer lemma}
 Suppose that $\sigma: E \rightarrow E$ is a field automorphism of finite order $e$,  with fixed field $F$.  
Let  $V$ be an $E$-vector space of dimension $d$, and $\phi: V \rightarrow V$
 a $\sigma$-semilinear automorphism.  Define the centralizer    $\mathrm{Z}(\phi)$ of $\phi$ 
in the ring of $E$-linear endomorphisms of $V$ via
$$ \mathrm{Z}(\phi) = \{ f: V \rightarrow V \mbox{ an $E$-linear map}, \ \ f \phi = \phi f\};$$
it is an $F$-vector space. Then    $$ \dim_F \mathrm{Z}(\phi) = \dim_E \mathrm{Z}(\phi^e),$$
where $\phi^e: V \rightarrow V$ is now $E$-linear. In particular, $\dim_F \mathrm{Z}(\phi) \leqslant (\dim_E V)^2$. \end{lemma}
\proof
Let $\bar{F}$ be an algebraic closure of $F$, and let $\Sigma$ be the set of $F$-embeddings $E \hookrightarrow \bar{F}$.  
Then $\bar{V} = V \otimes_{F} \bar{F}$ 
is a $E \otimes_{F} \bar{F} \simeq \bar{F}^{\Sigma}$-module, and splitting by idempotents of $E \otimes_{F} \bar{F}$
we get a decomposition
$$ \bar{V} = \bigoplus_{\tau \in \Sigma} \bar{V}^{\tau},$$
where $\bar{V}^{\tau}$ consists of $\bar{v} \in \bar{V}$ such that  
$e \bar{v} = \tau(e) \bar{v}$ for all $e \in E$. 
(Here the multiplication $e \bar{v}$ is for the $E$-module structure,
and $\tau(e) \bar{v}$ for the $\bar{F}$-module structure, on $\bar{V}$.)
Moreover, $\phi$ extends to an $\bar{F}$-linear endomorphism $\overline{\phi}$ of $\bar{V}$; this endomorphism carries $\bar{V}^{\tau}$ to $\bar{V}^{\tau \sigma^{-1}}$. 

Fix $\tau_0 \in \Sigma$; then  projection to the $\tau_0$ factor induces an isomorphism
$$ \mathrm{Z}(\overline{\phi})  \simeq \mbox{centralizer of $\overline{\phi}^e$ on $\bar{V}^{\tau_0}$}.$$
Now $(\bar{V}^{\tau_0}, \overline{\phi}^e)$ is obtained by base extension $\tau_0: E \rightarrow \bar{F}$
from the $E$-linear map $\phi^e: V \rightarrow V$; in particular, the dimension of the centralizer
on the right is the same as $\mathrm{Z}(\phi^e)$, 
whence the result. 
\qed

\subsection{Semisimplicity}

\begin{lemma} \label{semisimple induction}
Let $H \leqslant G$ be a finite-index inclusion of groups, and 
let $\rho: H \rightarrow \GL_n(F)$ be a semisimple representation of the group $H$
over the characteristic-zero field $F$. Then the induction $\rho^G = \mathrm{Ind}_H^G \rho$ is also semisimple.
\end{lemma}
\proof   This follows readily from the fact that a representation $\rho$ of $G$ is semisimple
if and only if its restriction to a finite-index normal subgroup $G_1 \leqslant G$ is semisimple:
take $G_1$ to be the intersection of conjugates of $H$.

For ``if'' one can promote a splitting from $G_1$ to $G$ by averaging; for ``only if''
we take an irreducible $G$-representation $V$, an irreducible $G_1$-subrepresentation $W \subset V$,
and note that $G$-translates of $W$ must span $V$, exhibiting $V|_{G_1}$ as a quotient of a semisimple module.   \qed

\subsection{Global Galois representations}  \label{SS}

\begin{lemma} (Faltings) \label{finiteness}
Fix integers $w,d \geqslant 0$, and fix $K$ and $S$ as above.  
There are, up to conjugation,  only finitely many semisimple Galois representations $\rho: G_{K} \rightarrow
\GL_d(\Q_p)$ such that
\begin{itemize}
\item[(a)] $\rho$ is unramified outside $S$, and
\item[(b)] $\rho$ is pure of weight $w$, i.e.
  for every prime $\wp \notin S$ the characteristic polynomial of Frobenius at $\wp$ has   all roots algebraic, with complex absolute value $q_{\wp}^{w/2}$.
  \item[(c)] For $\wp$ as above the characteristic polynomial of Frobenius at $\wp$ has integer coefficients. 
\end{itemize} 
\end{lemma}
\proof
This is a consequence of Hermite--Minkowski finiteness; 
see the proof of \cite[Satz 5]{Faltings}, or \cite[V, Proposition 2.7]{WustholzFaltings}. 
\qed
  
We want to explain how to adapt this proof to a reductive target group.  First we recall the notion of ``semisimple''
with general reductive target, and some allied notions. 

Let $K$ be a field of characteristic zero. 
First of all, recall that if $\mathbf{G}$ is a reductive algebraic group over  $K$
and $\rho: \Gamma \rightarrow \mathbf{G}(K)$ is a representation
of the group $\Gamma$, there are natural notions of ``irreducible'' and ``semisimple''
adapted to $\mathbf{G}$, as described by  Serre \cite[3.2]{Serre}: 

\begin{quote}
the representation $\rho$ is $G$-ir, or irreducible relative to $G$, if the image $\rho(\Gamma)$ is not contained
in a proper parabolic subgroup $P \leqslant G$ defined over $K$.
\end{quote}

For example, if $\mathbf{G}$ is an orthogonal or symplectic group,
this assertion amounts to saying that there is no {\em isotropic} $\Gamma$-invariant subspace. 
Next  
\begin{quote}
the representation $\rho$ is $G$-c.r., or completely reducible relative to $G$,  if for any parabolic subgroup $P \leqslant G$ defined over $K$
containing the image $\rho(\Gamma)$, there exists  a Levi factor $L \leqslant P$,  defined over $K$, 
which also contains this image.   
\end{quote}

We will also refer to $G$-c.r. as ``semisimple'' when the target group is clear. Let $\rho: \Gamma \rightarrow \mathbf{G}(K)$ be an arbitrary representation. Let $P$ be a $K$-parabolic subgroup that contains the image of $\rho$
and which is minimal for this property. Then the projection of $\rho$ to a Levi factor $M \subset P$
is independent, up to $G$-conjugacy, of the choice of $M$; see \cite[Proposition 3.3]{Serre}.   This resulting representation is called the semisimplification of $\rho$, relative to the ambient group $\mathbf{G}$,
and will be denoted by $\rho^{\mathrm{ss}}$. 
The Zariski closure of this semisimplification is a reductive group, at least for $K$ in characteristic zero:
see \cite[Proposition 4.2]{Serre}.\footnote{In \S 4 of \cite{Serre} the assumption is stated that $K$ is algebraically closed, but this is not used in the proof of Proposition 4.2. Alternately \cite[Theorem 5.8]{GCR}
can be used to pass from $K$ to $\bar{K}$.}

Later on we will use the following observation: 
\begin{lemma} \label{Artss}
For any $\gamma \in \Gamma$, $\rho^{\mathrm{ss}}(\gamma)$ and  $\rho(\gamma)$   have the same semisimple part up to conjugacy.  
\end{lemma}
\proof Indeed, let $P$ be as above, and factorize $P = MU$ into a Levi factor $M$
and $U$ the unipotent radical of $P$.  We must prove that 
for $p=mu \in P(K)$, with $m \in M(K)$ and $u \in U(K)$,  the semisimple parts of $p$ and $m$  are conjugate within $P$.
 To prove this  take a commuting factorization
$p = p^{ss} p^{u}$, and similarly for $m$.  By functoriality, $m^{ss}$ is the image of $p^{ss}$. 
We are reduced to the case of $m$ and $p$ semisimple:
\begin{equation} \label{well known statement} \mbox{a semisimple element $p=mu$ in $P(K)$ is $P(K)$-conjugate to $m$,}\end{equation}
and clearly it is enough to be able to conjugate $p$ into $M$. 

 The element $p$ is contained in some maximal torus $T$ (\cite[10.6,11.10]{Borel}) which is contained in a Levi subgroup of $P$. However all Levi subgroups are conjugate under $U(K)$ \cite[Proposition 20.5]{Borel}
 we may therefore conjugate $p$ into $M$ as desired.
\qed 

In passing we also record: %
\begin{lemma} \label{wks2}
Suppose $P=MU$ is a parabolic subgroup of the reductive $K$-group $G$.

Let $S \leqslant P$ be a
$K$-torus, then $S$ is conjugate under $U(K)$ to its projection to $M$. 
   
In particular, $\chi: \mathbf{G}_m \rightarrow P$ be a character; then $\chi$
is conjugate, under $P(K)$, to its projection to $M$. 
\end{lemma}
\proof 
We may assume that $S$ is a maximal torus, and then the claim follows from the argument above.
 \qed

Faltings' finiteness theorem continues to apply in this context: 
  
\begin{lemma} \label{faltings_finiteness_reductive} Let $\mathbf{G} \subset \GL_n$ be a reductive group, $K$ a number field, $S$ a finite set of places.
Consider all representations
$$\rho: G_{K} \longrightarrow \mathbf{G}(\Q_{p})$$
which, when considered as representations into $\GL_n(\Q_p)$, satisfy conditions (a), (b), (c) of Lemma \ref{finiteness}
(i.e. $S$-unramified, pure of weight $w$, integral).  %
 
Then there are only finitely many possibilities for the $\mathbf{G}(\Q_p)$-conjugacy class of $\rho^{\mathrm{ss}}$. 

Indeed, there are only finitely many possibilities up to $\mathbf{G}(\Q_p)$-conjugacy
for pairs $(\mathbf{Q}, \rho: G_{\Q} \rightarrow \mathbf{L}_Q(\Q_p))$ where
$\mathbf{Q}$ is a $\Q_p$-parabolic subgroup with Levi quotient $\mathbf{L}_Q$,   the image of $\rho$ is irreducible in $\mathbf{L}_Q$,
and $\rho$ 
again satisfies the conditions of Lemma \ref{finiteness}. 
\end{lemma}
  
\proof  %
Note first that for such $\rho$, the $\mathbf{G}$-semisimplification $\rho^{\mathrm{ss}}$ is also semisimple considered as a representation
with target $\GL_n$ (since its Zariski closure is reductive, as noted above).

By Lemma \ref{finiteness} is enough to check that,
for any fixed such $\rho_0$, there are only finitely many $\mathbf{G}(\Q_p)$-orbits 
on the set of $\GL_n(\Q_p)$-conjugates of $\rho_0^{\mathrm{ss}}$ with image in $\mathbf{G}$. 
Let $\mathbf{L}$ be the Zariski closure of the image of $\rho_0^{\mathrm{ss}}$. %
It is a reductive $\Q_p$-subgroup of $\mathbf{G}$. 
Then  for $g \in \GL_n(\Q_p)$ the image of $\Ad(g) \rho_0$ belongs to $\mathbf{G}$
if, and only if, $\Ad(g) \mathbf{L} \subset \mathbf{G}$.  In other words, it is enough to verify
that the set
$$ \{g \in \mathbf{GL}_n(\Q_p): \Ad(g) \mathbf{L} \subset \mathbf{G} \}$$
consists of finitely many double cosets under $(\mathbf{G}(\Q_p), \mathbf{L}(\Q_p))$, or equivalently
finitely many $\mathbf{G}(\Q_p)$-orbits. 

We may replace $\mathbf{L}$ by its connected component, and then 
it is enough to verify this assertion at the level of Lie algebras, i.e.\ to prove the same assertion for the set
$$ \{g \in \mathrm{GL}_n(\Q_p): \Ad(g) \mathfrak{l} \subset \mathfrak{g}\}$$
 According to  Richardson's theorem
 \cite[Theorem 7.1]{Richardson}
 this forms finitely many $\mathbf{G}$ orbits   over the algebraic closure $\overline{\Q_p}$.    The result then follows from finiteness of  the Galois cohomology
$H^1(\Q_p, \mathbf{S})$ for any linear algebraic group $\mathbf{S}$ (\cite[III \S 4, Theorem 4]{SerreGalois}). 

To see the validity of the refinement, 
note that there are finitely many conjugacy classes of parabolic subgroups $\mathbf{P}$ defined over $\Q_p$,
and for each such $\mathbf{P}$
there are -- by what we just proved, applied to a Levi factor -- only finitely many $\mathbf{P}(\Q_p)$-conjugacy classes
of  (pure of weight $w$, unramified outside $S$) irreducible representations $G_K \rightarrow \mathbf{L}_P(\Q_p)$.
 \qed

\subsection{Friendly places}
For our later applications it is convenient to have available a class of
``friendly'' places of a number field $K$
at which the local behavior of homomorphisms $G_K \rightarrow \Q_p^*$
is particularly simple.   (Actually, in our applications, it would be enough to do this for $K=\Q$,
for which everything is quite straightforward, and to always use Lemma \ref{globalsimple2} with $K=\Q$.
However, it makes our arguments a little easier to write to have friendly places available for a general number field $K$).
 
 First we recall some structural theory \cite[II.3.3]{Serreladic}. Let $\mathcal{C} \subset  G_{\Q} = \Gal(\overline{\Q}/\Q)$   be the conjugacy class of complex conjugation, and let $H^+ = \langle \mathcal{C} \rangle$, the normal subgroup generated by $\mathcal{C}$;
 there is a unique nontrivial homomorphism $H^+ \rightarrow \{ \pm 1\}$ and we let $H$ be its kernel.  A subfield $K \subset \overline{\Q}$
 is totally real if and only if it is fixed by $H^+$.
 It is CM if and only if it is fixed by $H$ but not $H^+$. 
  
  For an arbitrary number field $K \subset \overline{\Q}$ 
 let $E$ and $E^+$ be, respectively, the subfields of $K$ defined by fixed fields of $G_K \cdot H$ and $G_K \cdot H^+$, respectively
 (where $G_K$ is the Galois group of $\overline{\Q}$ over $K$). 
Then $E^+$ is the largest totally real subfield of $K$, and either $E^+=E$ is totally real, 
{\em or} $E$ is CM and is the largest CM subfield of $K$.

 \iffriendly
 \begin{dff} \label{friendly def} (Friendly places). 
 Let $K$ be a number field.  
 \begin{itemize}
 \item  If $K$ has a CM subfield, then let $E$ be  its maximal CM subfield  and $E^+$ the maximal totally real
 subfield of $E$.  In this case, we say that a place $v$ of $K$ is {\em friendly} if it is unramified over $\Q$, and it lies above a place of $E^+$ that is inert in $E$.
\item  If $K$ has no CM subfield, any place  $v$ of $K$ which is unramified over $\Q$ will be understood to be friendly.
\end{itemize}
 \end{dff}
 
 Clearly, infinitely many friendly places exist; however, if $K$ has a CM subfield, they have Dirichlet density $0$.   
 \fi

Consider, now, a continuous character $\eta: \Gal(\overline{K}/K) \longrightarrow \Q_p^*$,
ramified at only finitely many places; by class field theory it corresponds to a homomorphism
$\mathbf{A}_K^*/K^* \rightarrow \Q_p^*$. In particular, its restriction to places
above $p$ gives rise to a homomorphism $\eta_p: (K \otimes \Q_p)^* \longrightarrow \Q_p^*$.
As usual, we say this is locally algebraic if it agrees, in a neighbourhood of the identity,
with the $\Q_p$-points of an algebraic homomorphism $\mathrm{Res}_{(K \otimes \Q_p)/\Q_p} \mathbf{G}_m \longrightarrow \mathbf{G}_m$
of $\Q_p$-algebraic groups, cf.\ \cite[Chapter III]{Serreladic}. 
 This condition is implied by  being Hodge--Tate at primes above $p$,
 by a theorem of Tate \cite[Chapter III, Appendix]{Serreladic}.
 Moreover, since $\eta$ is finitely ramified, it follows that $\eta_p$
is trivial on a finite-index subgroup of the units $\mathcal{O}_K^*$, 
embedded into $(K \otimes \Q_p)^*$. 

 For such $\eta$, we say that $\eta$ is pure of weight $w$ 
 when it satisfies the condition explained in  Lemma \ref{finiteness}.

\begin{lemma} \label{friendlyplace} Let $v$ be any friendly place of $K$, lying above the prime $p$ of $\Q$.  
 For any
continuous character $ \eta: \Gal(\bar{K}/K) \longrightarrow \Q_p^*$,
ramified at only finitely many places, 
pure of weight $w$, and  locally algebraic at each prime above $p$, 
one has  
$$ \eta^2|_{K_{v}^*} = \chi \cdot \mathrm{Norm}_{K_{v}/\Q_p}^{w},$$
where $\chi$ has finite order.  In particular, $w$ is even and the Hodge--Tate weight of $\eta$ at the place $v$ equals $w/2$.  
\end{lemma} 

In other words, the restriction of globally pure characters to friendly places is of a standard form. 
Note that if the coefficients are enlarged from $\Q_p^*$ to $\Q_{p^2}^*$,
the statement above is no longer true; an example is given by
the idele class character associated to a CM elliptic curve.

The proof of this result is routine. The key point  is due to Artin and Weil: an algebraic Hecke character factors through the norm map to the maximal CM subfield. 
 \proof 
Being locally algebraic, $\eta$ gives rise to an algebraic character  
of $\mathrm{Res}_{K/\Q} \G_m$, which is trivial
on a finite-index subgroup of $\mathcal{O}^*$. 
 Said differently, 
 we obtain a $\Q_p$-rational  character 
 $ \mathbf{S} \longrightarrow \G_m$ of the Serre torus $\mathbf{S}$; we will denote this also by $\eta$. 
(Note that $\eta$ is forced to be $\Q_p$-rational since it carries $\mathbf{S}(\Q_p)$ into $\Q_p^*$). 
 Here $\mathbf{S}$ is the quotient of $\mathrm{Res}_{K/\Q} \G_m$ by the Zariski closure of (a sufficiently deep finite-index subgroup of) the units. 
Because of the purity assertion, %
if $\lambda \in K^*$ is a unit at all ramified primes for $\eta$, then $\eta(\lambda)$ 
is an algebraic number all of whose conjugates have absolute value $\mathrm{N}_{K/\Q}(\lambda)^{w/2}$. 
 
  The structure of this torus was in effect computed by Weil \cite{Weil},  and in detail by Serre: 
  If $K$ admits no CM subfield, then the norm map $\mathbf{S} \rightarrow \mathbf{G}_m$ is in fact an isogeny.
  So $\eta$ is (up to finite order) the norm raised to the power $w/2$. 
 The result follows. 
  
Thus we suppose that $K$ has a CM subfield; now let $E$ be the largest CM subfield of $K$, and let $E^+$
 be the totally real subfield of $E$. Then the norm map
 $ \mathbf{S} \rightarrow \mathbf{S}_{E}$ is an isogeny; in other words, a suitable power $\eta^k$ 
 factors through the norm from $K$ to $E$.   
 Therefore
it is enough to prove the Lemma for $K=E$, replacing $v$ by the place  of $E$ below it. 
In  particular,  by definition, $v$ lies above an inert prime of $E/E^+$.

Now there is a norm map $\mathbf{S}_E \rightarrow \mathbf{G}_m$. 
Write $x \mapsto \bar{x}$ for the complex conjugation on $E$. 
The map
$x \mapsto x/\bar{x}$, from $E^*$ to $E^*$, 
is trivial on a finite-index subgroup of the units,
and its image consists entirely of elements whose norm (to $E^+$) equals $1$.  Indeed
for any $\Q$-algebra $R$ the rule  $x \mapsto x/\bar{x}$  defines a map
$(E \otimes R)^* \rightarrow (E \otimes R)^*$, corresponding to 
a unique  map of $\Q$-algebraic groups
$$ \theta:  \mathbf{S}_E \rightarrow (\Res_{E/\Q} \G_m)^1$$
where the superscript $1$ denotes the kernel of the norm to $E^+$. Together with the norm map
this gives an isogeny
$ \mathbf{S}_E \longrightarrow \G_m \times  (\Res_{E/\Q} \G_m)^1$.
Raising the character $\eta$ to a suitable power we can
suppose that it factors through the right-hand side; twisting it 
by a power of the cyclotomic character, we can arrange that 
it is trivial on the $\G_m$ factor.   

In other words,  we are reduced to checking the case
where $\eta$ factors through $\theta$. 
Now the weights of $x\mapsto \eta(x)$ and $x \mapsto \eta(\bar{x})$  
coincide, but their product is trivial;  so the weight of $\eta$ is zero. 
Also $\eta$ is trivial on $E_v^*$: consider
$$E_{v}^*  \subset (E \otimes \Q_p)^* \rightarrow \mathbf{S}(\Q_p)  \stackrel{\theta}{\rightarrow} (E \otimes \Q_p)^{1} =  \left(\prod_{w|p} E_w^*\right)^{1}.$$ 
The image of $E_{v}^*$ is contained inside 
$ \{ y \in E_{v}^*: y \bar{y} = 1\}$;  this is contained in a  $\Q_p$-anisotropic subtorus of $(\Res_{E/\Q} \G_m)^1 $.
Therefore, any $\Q_p$-rational character of $(\Res_{E/\Q} \G_m)^1$ is trivial upon pullback to $E_v^*$.
This is exactly what we wanted to prove (since, as we just saw, once $\eta$ factors through $\theta$ its weight is zero).  %
\qed

\subsection{Reducibility of global Galois representations}
\label{red_galois}
We now give some lemmas which limit the reducibility of a global  pure Galois representation. The mechanism is as follows: purity
passes to subrepresentations, and then leads to restrictions on the sub-Hodge structure. 
 
For a  decreasing filtration $F^{\bullet} V$ on a vector space $V$ (with $F^0 V =V$)  we define the \emph{weight} of the filtration to be
\begin{equation} \label{weight definition}  \mathrm{weight}_F(V) = \frac{ \sum_{p \geqslant 0}  p \dim \mathrm{gr}^p(V)  }{\dim V},\end{equation}
where $\mathrm{gr}^p(V) = F^p(V)/F^{p+1}(V)$ is the associated graded. 
\footnote{Here and in Section \ref{Hodge_numbers_funny}, the symbol $p$ is used abusively to refer to the indexing on a Hodge filtration.
We hope this will not cause confusion.}
For the other $p$-adic Hodge theory terms that appear in the following result, see \cite[\S 6]{Conrad} or \cite[Expose III]{Asterisque}.

\begin{lemma} \label{globalsimple} 
Let $K$ be a number field and $v$ a friendly place. 
Let $V$ be a Galois representation of $G_K$ on a $\Q_p$-vector space 
which is crystalline at all primes above $p$, and pure of weight $w$. 
 
Let $V_{\dR} = (V \otimes_{\Q_p} B_{\mathrm{cris}})^{G_{K_v}}$ be the filtered\footnote{Here, and in other contexts, we will write $V_{\dR}$ even
 though we are using the crystalline functor, because in our applications it will be helpful to think of it in terms of de Rham cohomology.}
$K_v$-vector space  that is associated to 
$\rho|_{K_v}$ by the $p$-adic Hodge theory functor $\underline{D}_{\mathrm{cris}}$ of \cite[Expose III]{Asterisque}. 

Then the weight of the Hodge filtration on $V_{\dR}$  equals $w/2$. 
\end{lemma}

\proof Apply Lemma \ref{friendlyplace}  to $\det(V)$. 
\qed
 
\begin{lemma} \label{globalsimple2}
Let $K$ be a number field, and $L \supset K$ a finite extension. 
Let $\rho: G_L \rightarrow \GL_n(\Q_p)$ be a representation of $G_L$
that is crystalline at  all primes above $p$, and pure of weight $w$; let $a_u(\rho)$ be the weight of the associated Hodge filtration at each such prime $u$. Then, for any friendly prime $v$ of $K$ above $p$,  $$ \sum_{u | v} [L_u:K_v] a_u(\rho) =  [L:K]  \frac{w}{2}.$$
\end{lemma}

\proof We apply Lemma \ref{globalsimple} to  $\mathrm{Ind}_{G_L}^{G_K} \rho$ and to the place $v$.
Applying the functor of $p$-adic Hodge theory to its restriction to $K_v$, we obtain
$$ (\mathrm{Ind}^K_L \rho \otimes_{\Q_p} B_{\mathrm{dR}})^{G_{K_v}} \simeq \bigoplus_{u|v} ( \rho \otimes_{\Q_p} B_{\mathrm{dR}})^{G_{L_u}} $$
(considered now as a filtered $K_v$-vector space), 
and its weight is therefore  $\frac{\sum_{u | v} [L_u:K_v] a_u(\rho)}{[L:K]}$.
\qed

\subsection{The affine group $\Aff(q)$} 
Let $q \geqslant 3$ be a prime number and let $\Sym(\mathbf{F}_q)$ be the symmetric group on the $q$ elements of $\mathbf{F}_q$.  
Let $\Aff(q) \subseteq \Sym(\mathbf{F}_q)$ be the   subgroup consisting of permutations of $\F_q$ of the form
$ x \mapsto a x+b$ 
where $a \in \F_q^*$ and $b \in \F_q$. Thus\footnote{We use $\F_q^+$ to denote the additive group $\F_q$.} $\Aff(q) \cong (\F_q)^+ \rtimes (\F_q)^*$; this group
has important applications in the theory of qualifying examinations. We shall make extensive use of it as a Galois group for certain auxiliary coverings of curves.
 
\begin{lemma} \label{com map}
For any $s \geqslant 1$ consider the map $f: \Aff(q)^{2s} \longrightarrow \F_q^+$
given by
$$f: \mathbf{g} =  (g_1, g_1', \cdots, g_s, g_s') \mapsto [g_1, g_1'] \cdot [g_2, g_2'] \cdot \cdots \cdot [g_s, g_s']$$
(here $[x,y]$ is the commutator $xyx^{-1}y^{-1}$). 
 The image of the map
\begin{equation} \label{im im} \{ \mathbf{g} \in \Aff(q)^{2s}: f(\mathbf{g}) \neq 0, \mbox{ $\mathbf{g}$ generates $\Aff(q)$} \} \rightarrow \left[ \F_q^* \right]^{2s}\end{equation}
(sending each $g_i$ to its image in the abelian quotient $\F_q^*$)
consists precisely of those $(2s)$-tuples in $\F_q^*$ whose entries generate $\F_q^*$.  The fiber above any point in the image  has the same size.  \end{lemma}

\proof 
Note that, for such a fiber to be nonempty, the 
element $\mathbf{y} = (y_1, y_1', \dots, y_s, y_s')$ of the target must have the property that the $y_i$ and $y_i'$ generate $\F_q^*$. 
In this case, any preimage $\mathbf{g} \in \Aff(q)^{2s}$ with the property that $f(\mathbf{g}) \neq 0$ necessarily generates
$\Aff(q)$. 
The fiber of $\Aff(q)^{2s}$ above $\mathbf{y}$ is
(in obvious coordinates) an affine space over $\F_q$, 
and the map $f$ is a nontrivial affine-linear map; each fiber thus has size $q^{2s-1}(q-1)$.
\qed

\subsection{Symplectic groups} 

Let $K$ be a field of characteristic zero. 
As usual if $V$ is a symplectic space over a field $K$, with nondegenerate alternating bilinear form $\langle -, - \rangle$, we write
$\Sp(V)$ for the algebraic group of  automorphisms of $V$ preserving the bilinear form.

The following statement is an algebraic version of Goursat's lemma (cf. \cite[Lemma 5.2.1]{Ribet}).
One uses the fact that the Lie algebra $\mathfrak{sp}_V$ of $\Sp(V)$ is simple,
and that all the automorphisms of $\mathfrak{sp}_V$ are inner.

 \begin{lemma}
\label{goursat}
Suppose $G$ is an algebraic subgroup of $\Sp(V)^N$, satisfying the following conditions.
\begin{itemize}
\item For $1 \leqslant i \leqslant N$, the projection $\pi_i: G \rightarrow \Sp(V)$ onto the $i$-th factor is surjective.
\item For $1 \leqslant i, j \leqslant N$, there exists $g \in G$ such that $\pi_i(g)$ and $\pi_j(g)$ are unipotent with 
fixed spaces of different dimensions.
 \end{itemize}
Then $G$ is all of $\Sp(V)^N$.
\end{lemma}
 
Any unipotent element of $\Sp(V)$ whose fixed space has codimension $1$ is of the form 
\begin{equation} \label{transv} T_v^r: x \mapsto x + r \langle v, x \rangle v\end{equation}
for some $v \in V, r \in K$.  We call $T_v^r$ a transvection with center $v$, and write   $T_v$ for $T_v^1$. 

\begin{lemma}
\label{transvection_generation}
Let $V$ be a symplectic space over $\QQ$.  Suppose $v_1, v_2 \in V$ are linearly independent and satisfy
$$\langle v_1, v_2 \rangle \neq 0.$$ The Zariski closure of the subgroup generated by $T_{v_1}, T_{v_2}$ also contains $T_v$ for every $v \in \operatorname{Span} ( v_1, v_2 )$. 
\end{lemma}

\begin{proof}
The subgroup in question preserves the splitting $V = \langle v_1, v_2 \rangle \oplus \langle v_1, v_2 \rangle^{\perp}$,
and so we reduce to the case that $V$ is $2$-dimensional. %
The statement then amounts to the fact that $\SL(2)$ is generated, as an algebraic group, by upper and lower triangular matrices. 
\end{proof}

\begin{lemma}
\label{transvection_graph}
Let $V$ be a symplectic space over $\QQ$.  Let $S$ be a set of vectors $v \in V$.  Make a graph whose set of vertices are $S$, having an edge between $v_1$ and $v_2$ if and only if $\langle v_1, v_2 \rangle \neq 0$.  If this graph is connected, then the Zariski closure of the group generated by the transvections $T_v$, for $v \in S$, contains $T_w$ for any $w$ in the span of $S$.
\end{lemma}

\begin{proof}
We can assume $S$ is finite, and then use induction on $\left| S \right|$, using Lemma \ref{transvection_generation} for the inductive step.

In detail: Suppose $S = S_0 \cup \{v\}$, with the graph on $S_0$ connected.
By inductive hypothesis we obtain all transvections centered at vectors in  $W := \mathrm{span}(S_0)$.  
 It is enough to verify that the Zariski closure in question contains the transvection $T_x$ for each vector $x$ of the form $w +v \ (w \in W)$;
 this is so  when  $\langle w, v \rangle \neq 0$ by the prior Lemma. 
The condition $\langle w, v \rangle \neq 0$ defines a    Zariski-dense subset of $W$ 
and so we also get the remaining transvections $T_x$ when $\langle w, v \rangle = 0$ in the Zariski closure of them. 
\end{proof}

\section{Fibers with good reduction in a family} \label{fibers}

In this section we give a general criterion (Proposition \ref{finitesetoforbits})
which controls, in a given family of smooth proper varieties, the collection of fibers
that have good reduction outside a fixed set of primes.  The Proposition
simply translates (using $p$-adic Hodge theory) the finiteness statement of   Lemma \ref{finiteness}
into a restriction on the image of the period map. 
\subsection{Basic notation} \label{bas note}
We use notation $K, \mathcal{O},  \mathcal{O}_{(w)}, S, G_K, \F_w$ as in \S \ref{notn}. 

Let $Y$ be a smooth $K$-variety, and $\pi: X \rightarrow Y$ a proper smooth morphism. 

Suppose that this admits a good model over $\mathcal{O}$,
i.e.\ it extends to a proper smooth morphism $\pi: \mathcal{X} \rightarrow \mathcal{Y}$ 
of smooth $\mathcal{O}$-schemes. Suppose, moreover, that all the cohomology sheaves
$\mathbf{R}^q \pi_* \Omega^p_{\mathcal{X}/\mathcal{Y}}$
are sheaves of locally free $\mathcal{O}_Y$-modules, and that the same is true of
the relative de Rham cohomology $\mathscr{H}^q = \mathbf{R}^q \pi_* \Omega^{\bullet}_{\mathcal{X}/\mathcal{Y}}$. 
There is no harm in these assumptions, because the sheaves in question are coherent $\mathcal{O}_Y$-modules
which are free over the generic point of $\mathcal{O}$ \cite[Theorem 5.5]{DeligneDegeneration}; so the assumptions can always be achieved by possibly enlarging the set $S$ of primes.

The generic fiber of $\mathscr{H}^q$ is equipped with the Gauss--Manin connection (by  \cite[Theorem 1]{KatzOda})
and,  again by enlarging $S$ if necessary,  we may suppose that this extends to a morphism
\begin{equation} \label{GMcondef} \mathscr{H}^q \rightarrow \mathscr{H}^q \otimes \Omega^1_{\mathcal{Y}/\mathcal{O}}.\end{equation}

For any  $y \in Y(K)$, we shall denote by 
 $X_y = \pi^{-1}(y)$  the fiber of $\pi$ above $y$; it is a smooth proper variety over $K$.  
 Our goal in this section is to bound $\mathcal{Y}(\mathcal{O})$. We will do this by studying 
  the $p$-adic properties of the Galois representation   attached to $X_y$,  for $y \in \mathcal{Y}(\mathcal{O}) \hookrightarrow Y(K)$.  
  Fixing a degree $q \geqslant 0$, 
we  denote by $\rho_y$ the  representation of the Galois group $G_K$ on the  {\'e}tale cohomology group of $(X_y)_{\bar{K}}$: 
\begin{equation} \label{grep} \rho_y : G_K \rightarrow \Aut \  H^q_{\mathrm{et}}(X_y \times_{K} \bar{K}, \Q_p).\end{equation}

Fix an archimedean place $\iota: K \hookrightarrow \C$, 
and fix a finite place $v: K \hookrightarrow  K_v$ satisfying:

\begin{itemize}
\item if $p$ is the rational prime below $v$, then $p > 2$, and 
\item $K_v$ is unramified over $\Q_p$, and 
\item no prime above $p$ lies in $S$. 
\end{itemize}
 
  Fix $y_0 \in \mathcal{Y}(\mathcal{O})$. In what follows, we will analyze the set
\begin{equation} \label{Udef} U := \{y \in \mathcal{Y}(\mathcal{O}): y \equiv y_0 \mbox{ modulo $v$}. \}\end{equation}
and give criteria for the finiteness of $U$ in terms of the associated period map.   Clearly if $U$ is finite for each choice of $y_0$, then $\mathcal{Y}(\mathcal{O})$ is finite too. 

Finally, we put
$$X_0 = \pi^{-1}(y_0)$$ to be the fiber above $y_0$. 

\subsection{The cohomology at the basepoint $y_0$}  
  
For any $K$-variety $Z$, we shall denote by $Z_{\C}$   its base change to $\C$
via $\iota$, and by $Z_{K_v}$ its base change to $K_v$ via $v$. 

Let \begin{equation} \label{Vdef} V =    H^q_{\dR}(X_0/K).\end{equation} 
Let $d = \dim_K V$. We will also denote by $V_v$ and $V_{\inft}$
the $K_v$- and $\C$-vector spaces obtained by $\otimes_{K} K_v$ or $\otimes_{(K,\iota)} \C$. 
Then $V_{\inft}$ is naturally identified with
the de Rham cohomology of the variety $X_{0,\C}$, which is also (by the comparison theorem)
identified with the singular cohomology of $X_{0,\C}$ with complex coefficients:
$$V_{\C} \simeq H^q_{\mathrm{sing}}(X_{0,\C}, \C).$$
 
In particular, 
monodromy defines a representation 
$ \mu: \pi_1(Y_{\C}(\C), y_0) \longrightarrow \GL(V_{\inft}),$
whose Zariski closure we denote by $\Gamma$:
\begin{equation} \label{Gdef} \Gamma = \mbox{ Zariski closure of $\mathrm{image}(\mu)$},\end{equation}
an algebraic subgroup of $\GL(V_{\inft})$.   Note that both $V_{\inft}$ and $\Gamma$
depend on the choice of archimedean place $\iota$, although this dependence is suppressed in our notation.

\subsection{The Gauss--Manin connection}  \label{GMsubsec}

The connection \eqref{GMcondef} allows us to identify the cohomology of nearby fibers.
This is true both for the $K_v$ and $\C$ topologies. However, as we now discuss, both identifications
can be described as the evaluation of a single power series with $K$ coefficients,
which is convergent both for the $K_v$ and $\C$ topology.

Specifically, if we fix a local basis $\{v_1, \dots, v_r\}$  for $\mathscr{H}^q$
in a neighborhood of   some point of the scheme $\mathcal{Y}$, 
and write $\nabla v_i = \sum_j A_{ij} v_j$, where  $A_{ij} $ are sections of $ \Omega^1_{\mathcal{Y}}$,
then a  local section $\sum f_i v_i$ is flat exactly when it solves the equation 
\begin{equation} \label{flat f} d(f_i) = - \sum_{j} A_{ji} f_j.\end{equation}

In particular,  if $y_0 \in \mathcal{Y}(\mathcal{O})$ and the place $v$ is as before,
let $\overline{y_0} \in \mathcal{Y}(\F_v)$ be the reduction,
and choose a system of parameters  $p, z_1, \dots, z_m \in \mathcal{O}_{\mathcal{Y}, \overline{y_0}}$ for the local ring of $\mathcal{Y}$ at $\overline{y_0}$;
we may do this so that $(z_1, \dots, z_m)$ generate the kernel of the morphism $\mathcal{O}_{\mathcal{Y},\overline{y_0}} \rightarrow \mathcal{O}_{(v)}$ 
corresponding to $y_0$. 
The completed local ring $\widehat{\mathcal{O}}_{\mathcal{Y}, \overline{y_0}}$ at $\overline{y_0}$ is therefore identified with $\mathcal{O}_v[[z_1, \dots, z_m]]$,
and the image of $\mathcal{O}_{\mathcal{Y},\overline{y_0}}$ in it is contained in $\mathcal{O}_{(v)}[[z_1, \dots, z_m]]$. 
 
Fix a basis  $\{\bar{v}_1, \dots, \bar{v}_r\}$ for $\mathscr{H}^q$ at $\overline{y_0}$,
which we assume to be compatible with the Hodge filtration, i.e.\
each step of the Hodge filtration $F^i \mathscr{H}^q$ at $\overline{y_0}$ is spanned by a subset of $\{\bar{v}_i\}$. 
Then by lifting we obtain a similar basis $\{v_1, \dots, v_r\}$
for $\mathscr{H}^q$ over the  local ring $\mathcal{O}_{\mathcal{Y}, \overline{y_0}}$ of $\mathcal{Y}$ at $\overline{y_0}$.
With respect to such a basis $v_i$, the coefficients $A_{ij}$ of \eqref{flat f} are    
of the form $A_{ij} =\sum_{k=1}^{m} a_{ij,k} dz_k$, where $a_{ij,k}  \in \mathcal{O}_{\mathcal{Y}, \overline{y_0}}$.
In particular, the coefficients of $a_{ij,k}$, considered as formal power series in the $z_i$,   lie in $\mathcal{O}_{(v)}$.

We may write down a formal solution to \eqref{flat f}, 
where the $f_i$ are given by formal power series  
in $K[[z_1, \dots, z_m]]$. By direct computation we see that these are $v$-adically absolutely convergent for  $|z_i|_v  < |p|_v^{1/(p-1)}$ (where $p$ is the residue characteristic of $\mathcal{O}_v$)
and $\iota$-adically absolutely convergent for sufficiently small $|z_i|_{\C}$. %

By assumption, we have $p>2$, and $v$ is unramified above $p$.  Thus we obtain an identification  
\begin{equation} \label{vGM}  \mathrm{GM}:  H^q_{\dR} (\mathcal{X}_{y_0} / K_v)  \stackrel{\sim}{\rightarrow} H^q_{\dR}(\mathcal{X}_y / K_v) \end{equation}
whenever $y \in \mathcal{Y}(\mathcal{O}_v)$ satisfies $y \equiv y_0$ modulo $v$, and
\begin{equation}\label{topGM} \mathrm{GM}: H^q_{\dR}(X_{y_0,\C} / \C) \stackrel{\sim}{\rightarrow} H^q_{\dR}(X_{y, \C} / \C),\end{equation}
when $y \in Y_{\C}(\C)$ is sufficiently close to $y_0$.  
 In the coordinates of the basis $v_i$ fixed above,  $\mathrm{GM}$ is given by an $r \times r$ matrix with entries
$$A_{ij}(z_1, \dots, z_m) \in  \mathcal{O}_{(v)}[[z_1, \dots, z_m]],$$ 
convergent in the regions noted above.
 
  The fiber over the $\mathcal{O}$-point $y_0$ of $\mathcal{Y}$ gives a smooth proper $\mathcal{O}$-model $\mathcal{X}_0$  for $X_0$. 
For $y \in Y(\mathcal{O}_v)$ with $y \equiv y_0$ modulo $v$,  we have a commutative diagram
     \begin{equation}   \label{GMconnection}
 \xymatrix{
V_v = H^q_{\dR}(X_{y_0}/K_v)   \ar[dd]^{\mathrm{GM}} \ar[rd]^{\sim} &    \\  
&  H^q_{\mathrm{cris}}(\overline{\mathcal{X}_0}) \otimes_{\mathcal{O}_v} K_v. \\
H^q_{\dR}(X_{y}/K_v) \ar[ru]^{\sim} &       \\ 
 }
 \end{equation}
 where $\mathrm{GM}$ denotes the map induced by the Gauss--Manin connection, 
 $H^q_{\cris}$ is the crystalline cohomology of $\overline{\mathcal{X}_0}$ (as a reference for crystalline cohomology, see \cite{Berthelot, BOgus}), 
the  diagonal arrows are the canonical identification  \cite[Corollary 7.4]{BOgus} of crystalline cohomology with the de Rham cohomology of a lift,
and the commutativity of the diagram can be deduced from the results of  \cite[Chapter V]{Berthelot} (see Proposition 3.6.4 and prior discussion). 

This crystalline cohomology  $V_v=H^q_{\dR}(X_0/K)$ is equipped %
with a Frobenius operator
$$\phi_v: V_v \longrightarrow V_v,$$ 
which is semilinear with respect to the Frobenius on the unramified extension $K_v/\Q_p$.  By the isomorphisms of \eqref{GMconnection}, this $\phi_v$ acts on $H^q_{\dR}(X_{y}/K_v)$ and $H^q_{\dR}(X_{y_0}/K_v)$ as well, \emph{in a manner compatible with the map $\mathrm{GM}$}. 

 \subsection{The period mappings in a neighbourhood of $y$} \label{periodmaps}
  
 Now $V=H^q_{\dR}(X_0/K)$ is equipped with a Hodge filtration:
\begin{equation} \label{Hodge} V=F^0 V \supset F^1 V \supset \dots \end{equation}
 Let $\mathcal{H}$ be the $K$-variety parameterizing flags in $V$
with the same dimensional data as \eqref{Hodge}, and let $h_0 \in \mathcal{H}(K)$ be the point corresponding to the Hodge filtration on $V$.

Base changing by means of $v$ and $\iota$, we get  a $K_v$-variety $\mathcal{H}_v$  
and a $\C$-variety $\mathcal{H}_{\inft}$.
We denote by $h_0^{\iota} \in \mathcal{H}_{\inft}(\C)$ the image of $h_0$. 
 
Let $\Omega_{\inft}$ be a contractible analytic neighbourhood of $y_0 \in Y_{\C}^\an$. 
The Gauss-Manin connection defines an isomorphism $H_{\dR}(X_t/ \C) \simeq H_{\dR}(X_0/ \C)$ for each $t \in \Omega_{\inft}$. 
In particular, the Hodge structure on the cohomology of $X_t$ defines a point of $\mathcal{H}_{\C}(\C)$; this gives rise to the complex period map
$$\Phi_{\C}: \Omega_{\inft} \longrightarrow \mathcal{H}_{\C}(\C).$$
Indeed, $\Phi_{\C}$ extends to a map from the universal cover of $Y_{\C}^\an$ to $\mathcal{H}_{\C}(\C)$
and this map is equivariant for the monodromy action of $\pi_1(Y_{\C}^\an, y_0)$ on $\mathcal{H}_{\C}(\C)$. We conclude that the image of the period map can be bounded below by monodromy. \begin{lem}
Suppose given a family $X \rightarrow Y$, and take notation as above; in particular, $\Gamma$ is the Zariski closure of monodromy, and $h_0^{\iota} = \Phi_{\C}(y_0)$. 
Then we have the containment
\begin{equation} \label{GZ} \Gamma \cdot  h_0^{\iota} \subset  \mbox{the Zariski closure of $\Phi_{\C}(\Omega_{\inft})$ inside $\mathcal{H}_{\inft}$}.\end{equation}
\end{lem}
\begin{proof}
The preimage  $\Phi_{\C}^{-1} Z$ of any algebraic subvariety $Z \subset \mathcal{H}_{\C}$, with $Z \supset \Phi_{\C}(\Omega_{\inft})$, 
is a complex-analytic subvariety of $\widetilde{Y_{\C}^\an}$ containing $\Omega$ and thus all of $\widetilde{Y_{\C}^{\an}}$;
therefore
$$\pi_1(Y_{\C}, y_0) \cdot h_0^{\iota} \subset Z$$ 
and then $Z$ contains the Zariski closure of the right hand side, which is $\Gamma \cdot h_0^{\iota}$. 
\end{proof}
 
We need a $v$-adic analogue. 
Again, if $y \in \mathcal{Y}(\mathcal{O}_v)$ satisfies $y \equiv y_0$ modulo $v$,  the Gauss--Manin connection
\eqref{GMconnection} allows one to identify the Hodge filtration on $H^q_{\dR}(X_{y}/K_v)$
with a filtration on $V_v$, and thus with a point of $\mathcal{H}(K_v)$.  This gives rise to a $K_v$-analytic function
$$ \Phi_v:  \Omega_v \longrightarrow \mathcal{H}(K_v), \mbox{ where }\Omega_v = \{ y \in \mathcal{Y}(\mathcal{O}_v): y \equiv y_0 \mbox{ modulo $v$} \}.$$
  
The following simple Lemma plays a crucial role. It allows us to analyze the Zariski closure of the $p$-adic period map
in terms of the Zariski closure of the complex period map; for the latter we can use monodromy. 
  
\begin{lemma} \label{vCpowerseries}
Suppose given power series $B_0, \dots, B_N \in K[[z_1, \dots, z_m]]$ such that all $B_i$
are absolutely convergent, with no common zero,  both in the $v$-adic and complex disks %
$$U_v= \{\underline{z} : |z_i|_v < \epsilon\}  \mbox{ and } U_{\C} = \{\underline{z} : |z_i|_{\C} < \epsilon\}.$$
Write
$$\underline{B}_v: U_v \rightarrow \mathbf{P}^N_{K_v}$$
$$\underline{B}_{\C}: U_{\C} \rightarrow \mathbf{P}^N_{\C}$$
for the corresponding maps.

Then there exists a  $K$-subscheme $\mathcal{Z} \subset \mathbf{P}^N$
whose base extension to $K_v$ (respectively $\C$) gives the Zariski closure
of $\underline{B}_{\C}(U_{\C}) \subset \mathbf{P}^N_{\C}$
(respectively $\underline{B}_v(U_v) \subset \mathbf{P}^N_{K_v}$).  In particular, these Zariski closures have the same dimension. 
\end{lemma}
\proof

We take $I$ the ideal of $\mathcal{Z}$ to be that generated by all homogeneous polynomials $Q \in K[x_0, \dots, x_N]$
such that $Q(B_0, \dots, B_N)$ is identically zero. 

To verify the claim (for $K_v$; the proof for $\C$ is identical) we just need to verify
that if a homogeneous polynomial $Q_v \in K_v[x_0, \dots, x_N]$ vanishes on $\underline{B}_v(U_v)$ then 
$Q_v$ lies in the $K_v$-span of $I$. 
But if $Q_v$ vanishes on  $\underline{B}_v(U_v)$
then $Q_v(B_0, \dots, B_N) \equiv 0$ in $K_v[[z_1, \dots, z_m]]$. The identical vanishing of $Q_v(B_0, \dots, B_N)$
is an infinite system of linear equations on the coefficients of $Q_v$, with coefficients in $K$. Any $K_v$-solution of such a linear system
is, of course, a $K_v$-linear combination of $K$-solutions. 
\qed 
  
By embedding $\mathcal{H}$ into a projective space $\mathbf{P}^N$, and applying the prior two Lemmas, we deduce: 
 
\begin{lemma} \label{Zarclos}
The dimension of the Zariski closure (in the $K_v$-variety $\mathcal{H}_{K_v}$) of 
$\Phi_v(\Omega_v)$
is at least the (complex) dimension of $\Gamma \cdot h_0^{\iota}$. 

In particular, if $\mathcal{H}_v^{\mathrm{bad}} \subset \mathcal{H}_v$ is a Zariski-closed subset of dimension less than $\dim_{\C}(\Gamma \cdot h_0^{\iota})$,   
then $\Phi_v^{-1}(\mathcal{H}_v^{\mathrm{bad}})$ is contained in a proper $K_v$-analytic subset of $\Omega_v$, by which we mean a subset cut out by $v$-adic power series converging absolutely on $\Omega_v$.   \end{lemma}
 
One can do better than this using the results of Bakker and Tsimerman, replacing ``proper $K_v$-analytic'' by ``Zariski-closed.'' See \S \ref{hypersurface}. 
We do not need this improvement for the applications to Mordell.

\subsection{Hodge structures} \label{HSsec}
We use $p$-adic Hodge theory to relate Galois representations to crystalline cohomology.  A good reference is \cite{Conrad} or \cite{Asterisque}.
  
For each $y\in U$ the representation $\rho_y$ (see \eqref{grep}) is crystalline upon restriction to $K_v$,  because of the existence of the model $\mathcal{X}_y$
for $X_y$.  By $p$-adic Hodge theory,
there  is \cite[Proposition 9.1.9]{Conrad} a fully faithful embedding of categories:
\begin{equation} \label{CCR} \mbox{crystalline representations of $\Gal_{K_v}$ on $\Q_p$ vector spaces} \hookrightarrow \mathcal{FL},\end{equation}
where the  objects of $\mathcal{FL}$ are triples $(W, \phi, F)$ of a $K_v$-vector space $W$,
a Frobenius-semilinear automorphism $\phi: W \rightarrow W$, and a descending filtration $F$ of $W$. 
The morphisms in the category $\mathcal{FL}$ are morphisms of $K_v$-vector spaces
that respect $\phi$ and filtrations \cite[Expose III, \S 4.3]{Asterisque}.

By the crystalline comparison theorem of Faltings \cite{Faltings_pHT}, the embedding \eqref{CCR} carries $\rho_y$ to the triple  $ \left( H_{\dR}^q(X_y/ K_v), \phi_v, \mbox{ Hodge filtration for $X_y$} \right)$.  But \eqref{GMconnection} induces an isomorphism in $\mathcal{FL}$: 
$$ \left( H_{\dR}^q(X_y/ K_v), \phi_v, \mbox{ Hodge filtration for $X_y$} \right) \simeq (V_v, \phi_v, \Phi_v(y)),$$
 
As a sample result of what we can now show, we give the following. 
We will use the method of proof  again and again, so it seems useful to present it in the current simple context.
 
\begin{proposition}
\label{finitesetoforbits}
Notation as above: in particular $X \rightarrow Y$ is a smooth proper family over $K$, 
$V$  is  the degree $q$ de Rham cohomology of a given fiber $X_0$ above $y_0 \in Y(K)$, 
$\mathcal{H}$ a space of flags in $V$,   $$\Phi_v:  \{\mbox{$y \in \mathcal{Y}(\mathcal{O}_v): y \equiv y_0$}\} \longrightarrow \mathcal{H}(K_v)$$
is the $v$-adic period mapping, $\Gamma \subset \GL(V_{\C})$ is the Zariski closure of the monodromy group, and $h_0 = \Phi(y_0)$ is the image of $y_0$ under the period mapping. 

Suppose that 
\begin{equation} \label{dim_ineq} \dim_{K_v}\left( \mathrm{Z}(\phi_v^{[K_v:\Q_p]}) \right)   < \dim_{\C} \ \Gamma \cdot h_0^{\iota}\end{equation}
where the left-hand side $\mathrm{Z}(\dots)$ denotes the centralizer, in $\Aut_{K_v}(V_v)$, of the  $K_v$-linear operator  $\phi_v^{[K_v:\Q_p]}$. 
    
Then the set 
\begin{equation} \label{Uss} \{y \in Y(\mathcal{O}): y \equiv y_0 \mbox{ modulo $v$}, \rho_y\mbox{ semisimple} \} \end{equation}
 is contained in a proper $K_v$-analytic subvariety of the residue disk of $Y(K_v)$ at $y_0$.   
 \end{proposition}

\proof 
For any $y$ as in \eqref{Uss} the Galois representation $\rho_y$ belongs to a finite set of isomorphism classes (Lemma \ref{finiteness}). 
By our previous discussion the triple  $(V_v, \phi_v, \Phi_v(y))$ also belongs to a finite set of isomorphism classes (now in the category $\mathcal{FL}$). Choosing representatives 
$(V_v, \phi_v, h_i)$ for these isomorphism classes,  we must have
$$ \Phi_v(y) \in  \bigcup_{i} \mathrm{Z}(\phi_v) \cdot h_i,$$
where $\mathrm{Z}(\phi_v)$ is the subgroup of elements in $\GL_{K_v}(V_v)$
which commute with $\phi_v$. 

Now certainly $\mathrm{Z}(\phi_v) \subset \mathrm{Z}(\phi_v^{[K_v:\Q_p]})$, and the right-hand side
is  now the $K_v$-points of a  $K_v$-algebraic subgroup of $\GL_{K_v}(V_v)$.   Therefore, any $y$ as in \eqref{Uss} is contained in the preimage, under $\Phi_v$, of a proper Zariski-closed
subset of $\mathcal{H}_v$ with dimension the left hand side of \eqref{dim_ineq}. This is obviously a $K_v$-analytic subvariety as asserted. It is proper because  of Lemma \ref{Zarclos}. 
\qed
 
In conclusion we note that we really have bounded $\mathcal{Y}(\mathcal{O})$ rather than  the set of $y \in Y(K)$ for which the 
abstract Galois representation $\rho_y$ has good reduction outside $S$. To bound the
latter set, we would have to deal with the possibility that such $y$ would be nonintegral at $S$; this would require a more detailed analysis
``at infinity'' and we have not attempted it. 

\section{The $S$-unit equation} \label{Sunit}
As a first application,  and a warm-up to the more complicated case of curves of higher genus, we  will show finiteness of the set of solutions to the $S$-unit equation.
This argument is not logically necessary for the later proofs but we hope it will serve as a useful introduction to them. 

\begin{thm}
The set
$$ U= \{ t \in \mathcal{O}_S^*:  1-t \in \mathcal{O}_S^*\}$$
is finite.
\end{thm}
 
\subsection{Reductions}
We begin with some elementary reductions.

We may freely enlarge both $S$ and $K$. Thus, we may suppose that 
$S$ contains all primes above $2$ and that $K$ contains the $8$th roots of unity.
Let $m$ be the largest power of $2$ 
dividing the order of the group of roots of unity in $K$.  By assumption $m \geqslant 8$. 

First of all, it suffices to prove finiteness of the set 
$$ U_1 = \{ t \in \mathcal{O}_S^*:  1-t \in \mathcal{O}_S^*, t \notin (K^*)^2\},$$
because $U \subset  U_1 \cup U_1^2 \cup U_1^4 \cup \dots \cup U_1^m.$
To see this, we take $t \in U$ and try to repeatedly extract its square root; observe that such a square root, if in $K$,  also belongs to $U$.
 If we cannot extract an $m$th root of $t$, we are done; otherwise, write $t=t_1^{m}$
and adjust $t_1$ by a primitive $m$th root of unity to ensure that $t_1$ is nonsquare. %

Suppose that $t \in U_1$. Since $t$ is a nonsquare and $\mu_m \subset K$ the order
of $t$ in the group $(K^*)/(K^*)^{m}$ is exactly $m$.  
Otherwise there is some proper divisor $k > 1$ of $m$, and an element $a \in K^*$, 
such that $t^k = a^m$, i.e.\ $t \in a^{m/k} \mu_k$, contradicting the fact that $t$ is nonsquare. 
 
Fixing  $t^{1/m}$  an $m$th root of $t$ in $\overline{K}$, the field $K(t^{1/m})$ is Galois over $K$, and Kummer theory guarantees that its Galois group  is $\Z/m\Z$. 
There are (Hermite--Minkowski) only finitely many possibilities for $K(t^{1/m})$.
Enumerate them; call them $L_1, \dots, L_r$, say. 
Each $L_i$ is a cyclic degree-$m$ extension of $K$, and it is sufficient to prove finiteness of the set 
\begin{equation} \label{brr} U_{1,L} = \{ t \in U_1,  K(t^{1/m}) \simeq L\}.\end{equation}
for a fixed field $L \in \{L_1, \dots, L_r\}$; here we understand $K(t^{1/m})= K[x]/(x^m-t)$.

Fix an $L$ as above.  $L$ is cyclic of degree $m$ over $K$. Choose
a prime $v$ of $K$ such that: 
\begin{itemize}
\item[(i)]  the class of Frobenius 
at $v$ generates $\mathrm{Gal}(L/K)$; 
\item[(ii)]  the prime $p$ of $\Q$ below $v$ is unramified in $K$. 
\item[(iii)] no prime of $S$ lies above $p$.
\end{itemize}
In particular, $v$ is inert in $L/K$; thus, if $t \in U_{1,L}$
then $t$ is not a square in $K_v$, for otherwise $L \otimes_{K} K_v \simeq K_v[x]/(x^m-t)$
would not be not a field. 
 
In summary, it is enough to prove the following lemma. 
\begin{lemma}
\label{Sunit_reduced}
Suppose $K$ contains the $8$th roots of unity, and $S$ contains all primes above $2$. 
Fix a cyclic field extension $L/K$,
a place $v \not \in S$ as above, and a basepoint $t_0 \in \mathcal{O}_S$.   %
Let $U_{1, L}$ be as above.  Then the set
\begin{equation} \label{br} \{t \in U_{1,L}: t \equiv t_0 \mbox{ modulo $v$}\}\end{equation}
is finite.
\end{lemma}

The proof of this Lemma will occupy the rest of the section.  Throughout the proof, $p$ is the prime of $\Q$ below $v$, and ``Tate module'' always refers
to $p$-adic Tate module.

\subsection{A variant of the Legendre family}
\label{Legendre_section}
As discussed in the Introduction, we apply Proposition
\ref{finitesetoforbits} not to the Legendre family, but to a modification of it: 
Let   $ \mathcal{Y} =\P^1_{\mathcal{O}}-\{0,1, \infty \} $
(where $0,1, \infty$ denote the corresponding sections over $\operatorname{Spec} \mathcal{O}$)
and let $\mathcal{Y}' = \P^1_{\mathcal{O}} - \{0, \mu_m,\ \infty \}$; let $\pi: \mathcal{Y}' \rightarrow \mathcal{Y}$
be the map $u \mapsto u^m$. 

Let $\mathcal{X} \rightarrow \mathcal{Y}'$ be the Legendre family,
so that its fiber over $t$ is the curve $y^2=x(x-1)(x-t)$;
and consider the composite
$$ \mathcal{X} \longrightarrow \mathcal{Y}' \stackrel{\pi}{\longrightarrow} \mathcal{Y}.$$ 
We will apply our prior results to the family $\mathcal{X} \rightarrow \mathcal{Y}$;  
also, as before, we denote by $X$ and $Y$ the fibers of $\mathcal{X}$ and $\mathcal{Y}$
over $\Spec(K)$.  Thus the geometric fiber $X_t$  of $X \rightarrow Y$ over $t \in Y(K)$ is the disjoint union of the  curves
$y^2 =x(x-1)(x-t^{1/m})$ over all $m$th roots of $t$. 
 
\subsection{Proof of Finiteness}  \label{proof of finiteness Sunit}

Assume for the moment the following two Lemmas; they will be proved in \S \ref{proofs_of_two_lemmas}.

\begin{lemma}[Big monodromy]  \label{monolemma}
Consider the family of curves over $\C-\{0,1\}$
whose fiber over $t \in \C$ is the union of the elliptic curves $E_z: y^2=x(x-1)(x-z)$, over all $m^{\text{th}}$ roots $z^m=t$. 
Then the action of monodromy  
\begin{equation} \label{oinker} \pi_1(\C-\{0,1\}, t_0) \longrightarrow  \Aut \left( \bigoplus_{z^m = t_0} H^1_B(E_z, \Q)\right) \end{equation}
has Zariski closure  containing $\prod_{z} \SL(H^1_B(E_z, \Q))$.
\end{lemma} 

\begin{lemma}[Generic simplicity]
\label{sunit_nonsimple}
Let $L$ be a number field and $p$  a rational prime, larger than $2$, and unramified in $L$.  There are only finitely many $z \in L$
such that $z, 1-z$ are both $p$-units, but 
for which the Galois representation of $G_L$ on the Tate module $T_p(E_z) = H^1_{\et}(E_{z, \bar{L}}, \Q_p)$ of
the elliptic curve 
 $$E_z: y^2=x(x-1)(x-z),$$  fails to be simple.
 \end{lemma}
Of course much stronger results than Lemma \ref{sunit_nonsimple} are known. 
The point here is that we prove this in a ``soft'' fashion, using the Torelli theorem as a substitute for more sophisticated arguments;
although we use the specific feature of Hodge weights $0$ and $1$, the argument is robust enough to generalize 
(although with a little added complexity, see e.g.\ Lemma \ref{generalsimple2}).

\begin{proof}[Proof of Lemma \ref{Sunit_reduced} assuming Lemmas \ref{monolemma} and \ref{sunit_nonsimple}]
This argument is similar to the proof of Proposition \ref{finitesetoforbits},
with added complication coming from the interaction of the fields $K$ and $L$.
Recall that we have fixed $t_0 \in U_{1,L}$ and we must verify the finiteness of the set
of $t \in U_{1,L}$ with $t \equiv t_0$ modulo $v$. 

By Lemmas \ref{sunit_nonsimple} and \ref{finiteness}, it is enough to verify the finiteness
of the subset of such $t$ where the pair
$(K(t^{1/m}), \rho_t | G_{K(t^{1/m})})$
lies in a fixed isomorphism class; in particular
$(K_v(t^{1/m}), \rho_t|G_{K_v(t^{1/m})})$
lies in a fixed isomorphism class. 

Under the correspondence of $p$-adic Hodge theory,
$\rho_t$ restricted to $K_v(t^{1/m})$ corresponds to the
filtered $\phi$-module
\begin{equation} \label{dat}  \left( H^1_{\dR}(X_{t, K_v} / K_v) \mbox{ as $K_v(t^{1/m})$-module}, \mbox{Frobenius, filtration}\right),\end{equation}
where we equip $H^1_{\dR}(X_{t, K_v} / K_v)$  
with the structure of $2$-dimensional vector space over $K_v(t^{1/m})$
that arises from the scheme structure of $X_t$ over $K(t^{1/m})$.

Let us clarify this vector space structure over $K_v(t^{1/m})$, which is crucial to our argument.
Although {\em a priori} a $K$-scheme,  the factorization $X \rightarrow Y' \rightarrow Y$
induces on $X_t$ the structure of $K(t^{1/m})$-scheme, i.e.\
arising from the morphism $X_t \rightarrow (Y')_t \simeq \Spec K(t^{1/m})$. 
Now the de Rham cohomology of $X_t$ is the same
whether we consider it as a $K(t^{1/m})$-variety or as a $K$-variety.
If we consider it as $K$-variety, we can   recover its structure of 
 $K(t^{1/m})$-vector space by means of the natural map
$$K(t^{1/m}) = H^0_{\dR}(Y'_t/K) \rightarrow H^0_{\dR}(X_t/K).$$
 The same picture works with $K$ replaced by $K_v$ everywhere. 
 
(Similarly, there are two natural interpretations  for ``Frobenius'' in \eqref{dat},
but they are equivalent: 
As just explained, we can consider the space $H^1_{\dR}$
as the de Rham cohomology of either a $K_v(t^{1/m})$-scheme, or of the associated
$K_v$-scheme obtained simply by restricting the scalars.  Both of these schemes have evident integral models, over
$\mathcal{O}_v[x]/(x^m-t)$ and $\mathcal{O}_v$ respectively. Accordingly,
the de Rham cohomologies can be identified with the crystalline cohomologies of the special fibers;
these crystalline cohomologies are identified, in a fashion that respects the semilinear Frobenius endomorphisms.) 

The Gauss--Manin connection for the family $X \rightarrow Y$ induces
\begin{equation} \label{above} H^1_{\dR}(X_{t, K_v} / K_v) \simeq H^1_{\dR}(X_{t_0, K_v} K_v)\end{equation}
which, by compatibility of Gauss--Manin connection with the cup product, is compatible with their module structures over the corresponding $H^0$s.
The corresponding identification of $H^0$s induces the standard identification   $K_v(t^{1/m}) \simeq K_v(t_0^{1/m})$ and therefore the isomorphism \eqref{above} is compatible with structures of $K_v(t^{1/m}) \simeq K_v(t_0^{1/m})$-modules.
 
Therefore, under the identification of \eqref{above}, the $F^1$-step of the filtration on $H^1_{\dR}(X_{t, K_v} / K_v)$ is identified with a $K_v(t_0^{1/m})$-line inside $H^1_{\dR}(X_{t_0, K_v} / K_v)$.
Call this line $\Phi(t)$.
The variation of this line gives a $K_v$-analytic period mapping 
\begin{equation} \label{permap}
\xymatrix{
\Phi: \{t \in K_v,  t \equiv t_0 \mbox{ modulo $v$} \} \ar[r] &   \mbox{$K_v(t^{1/m})$-lines in $H^1_{\dR}(X_{t_0, K_v} / K_v)$}   \ar[d]  \ar[r]^{\qquad \qquad  \simeq} &  \mathbf{P}^1_{K_v(t_0^{1/m})}  \ar[d] \\
& \mbox{$K_v$-subspaces in $H^1_{\dR}(X_{t_0, K_v} / K_v)$} \ar[r]^{\qquad \qquad \simeq} & \mathrm{Gr}(2m,m)_{K_v}.
}
\end{equation}
(The period mapping  for the family $X \rightarrow Y$ {\em a priori} takes values in the bottom row, but we have just seen that it factors through the top row.
See \S \ref{GMsubsec} for a more detailed discussion
of the radius of convergence; in particular it defines a rigid analytic function on a domain
containing $\{ t \in K_v, t \equiv t_0 \mbox{ modulo $v$}\}$ i.e.\ the $K_v$-points in a residue disk.)
 
Therefore (applying the Gauss--Manin connection to identify \eqref{dat} with similar data over $t_0$) the  isomorphism class of the  quadruple
 $$  \left(  K_v(t_0^{1/m}), \mbox{$H^1_{\dR}(X_{t_0, K_v} / K_v)$ as $K_v(t_0^{1/m})$-module}, \Phi(t), \mbox{Frob}_v \right) $$
is determined from (\ref{dat}) and therefore the triple
$$ \left( H^1_{\dR}(X_{t_0, K_v} / K_v) \mbox{ as $K_v(t_0^{1/m})$-module}, \Phi(t), \mbox{Frob}_v \right)$$
lies in a  finite set of isomorphism classes for filtered $\phi$-modules over $K_v(t_0^{1/m})$ 
(coming from the finitely many automorphisms of $K_v(t_0^{1/m})$ over $K_v$). 
Therefore, $\Phi(t)$ lies in a finite collection of orbits for  
$$ Z = \mbox{centralizer of $\Frob_v$ in $K_v(t_0^{1/m})$-linear automorphisms of $H^1_{\dR}(X_{t_0, K_v} / K_v)$}.$$
Now we can apply Lemma \ref{centralizer lemma} to 
the field extension $K_v(t_0^{1/m})/K_v$ and the $K_v$-linear automorphism
$\Frob_v^{[K_v:\Q_p]}$ of $H^1_{\dR}(X_{t_0, K_v} / K_v)$.   This gives us that
$$\dim_{K_v} Z \leqslant  (\dim_{K_v(t_0^{1/m})} H^1_{\dR} )^2 =4.$$ 

Our analysis thus far has shown 
that the set of $t \in U_{1,L}$ such that $t \equiv t_0$ modulo $v$
is contained in
$$ \Phi^{-1}\left( \mathcal{Z}\right),$$
where $\Phi$ is the period map as in \eqref{permap} and $\mathcal{Z} \subset \mathrm{Gr}_{K_v}(2m,m)$ has dimension at most $4$. 
By Lemma \ref{Zarclos}, this set is finite so long as we verify an assertion about the complex period map, namely,  that the dimension
of the orbit of the algebraic monodromy group over $\C$ is strictly greater than $4$.  As in Lemma \ref{Zarclos},
we fix an embedding $K \hookrightarrow \C$ throughout the following discussion.

As mentioned, the vector space $V = H^1_{\dR}(X_{t_0}/K)$ has the natural  structure
of a $2$-dimensional vector space over $K(t_0^{1/m})$. 
The splitting
of $X_{t_0, \C}$ into geometric components induces a splitting
\begin{equation} \label{VCsplit}  V_{\C} = \bigoplus_{i=1}^m  V_i,\end{equation}
where each $V_i$ is a $2$-dimensional complex vector space;
moreover the Hodge filtration on $H^1_{\dR}(X_{t_0}/K) \otimes \C$ also splits along this decomposition. 
Lemma \ref{monolemma}
shows that the algebraic monodromy group $\Gamma$ contains $\prod_{i=1}^m \SL(V_i)$. 
The pertinent flag variety  $\mathcal{H}  \simeq \mathrm{Gr}(V, m)$ 
is the  variety of $m$-dimensional subspaces in $V$; the splitting
\eqref{VCsplit} induces a natural inclusion $\prod_{i=1}^m \P V_i \hookrightarrow \mathcal{H}_{\C}$. 
Therefore the  
the orbit $\Gamma h_0^{\iota}$ is all of $\prod_{i=1}^m \P V_i$ and, in particular, has dimension $m \geqslant 8$. 
Lemma \ref{Zarclos} now gives the desired finiteness.
 
In conclusion, assuming Lemmas \ref{monolemma} and \ref{sunit_nonsimple}, we have shown that the set described in  \eqref{brr} is finite. 
\end{proof}

\subsection{Big Monodromy and Generic Simplicity} \label{proofs_of_two_lemmas}

In this section we prove Lemmas \ref{monolemma} and \ref{sunit_nonsimple}.

\proof[Proof of Lemma \ref{monolemma}]  Write $\Gamma$ for the Zariski closure in question.  It preserves the splitting of \eqref{oinker}, 
although not the individual summands.  Then:
\begin{itemize}
\item[-]  $\Gamma$ transitively permutes the factors on the right-hand side of \eqref{oinker},
by considering the action of local monodromy near $t=0$; 

\item[-] $\Gamma \cap \SL(2)^m$ projects to $\SL(2)$ in each factor:
 indeed, this projection contains a finite-index subgroup of the algebraic monodromy group of the Legendre family. 
\item[-] 
$\Gamma$ contains an element of the form
$$(1, 1, \dots, 1, u, 1, \dots, 1)$$
where $u \in \SL(2)$ is a nontrivial unipotent element, 
as we see by considering the action of local monodromy near $t=1$. 
\end{itemize}
We now apply a slight variant of Lemma \ref{goursat} to conclude that  $\Gamma \supset  \SL(2)^m$. 
\qed

\proof[Proof of Lemma \ref{sunit_nonsimple}]
Fix $z_0 \in L$ with the quoted $p$-integrality properties; in particular, $E_{z_0}$ 
has good reduction at all primes of $L$ above $p$. 

It is enough to show the same finiteness when we restrict to the set
$$ V_L = \{z \in L:  z \equiv z_0 \mbox{ modulo $v$, for all $v|p$}\}.$$
If $T_p(E_z)$ is reducible there exists a one-dimensional subrepresentation $W_z \subset T_p(E_z)$.
By Lemma  \ref{globalsimple2} (applied with $K = \Q$)
there is a place  $w$ of $L$ above $p$ such that $F^1 (W_z^{\dR}) = W_z^{\dR}$;
here $W_z^{\dR}$ is the filtered $L_w$-vector space associated to $W_z$ by $p$-adic Hodge
theory over the $p$-adic field $L_w$.

Because the Newton and Hodge polygons of $W_z^{\dR}$ have the same endpoint, the slope of semilinear Frobenus acting on $W_z^{\dR}$
is equal to $1$; by the same reasoning for $H^1_{\dR}(E_z/L_w)$,  the sum of slopes for the semilinear Frobenius acting on $H^1_{\dR}(E_z/L_w)$ is $1$,
so it has another slope equal to $0$. 

In particular, 
the $L_w$-linear Frobenius $\Frob_w^{[L_w:\Q_p]}$ has distinct eigenvalues.

Also, the $L_w$-line $W_z^{\dR}$ must coincide with  $F^1 H^1_{\dR}(E_z/L_w)$,  
so that the latter space
is the slope-1 eigenline for the semilinear Frobenius $\Frob_w$.
 
As in the discussion around \eqref{GMconnection}, 
Gauss--Manin induces an identification 
\begin{equation} \label{gmc2} H^1_{\dR}(E_{z_0}/L_w) \simeq H^1_{\dR}(E_z/L_w)\end{equation}
of $L_w$-vector spaces with semilinear Frobenius action. 
But the position of the Hodge line $F^1 H^1_{\dR}(E_z/L_w)$  
varies $w$-adic analytically  inside the disk  $V_L$ -- here 
we  use \eqref{gmc2} to identify this line to a line inside the fixed space $H^1_{\dR}(E_{z_0}/L_w)$ -- and  the associated $w$-adic analytic function is nonconstant (by the  -- trivial -- Torelli theorem for elliptic curves). 
It follows there are at most finitely many $z \in V_L$ for which $F^1 H^1_{\dR}(E_z/L_w)$
is the slope-1 Frobenius eigenline.
Taking the union over possible $w$ we still see that the exceptional set is finite. 
\qed

\newcommand{\Xo}{\overline{X}_{\bar{o}}} 
\newcommand{\size}{\mathrm{size}_v}

\section{Outline of the argument for Mordell's conjecture} \label{Mordelloutline}

The proof of the Mordell conjecture is substantially harder than the $S$-unit equation. 
To try to assist the reader, we summarize the proof here,
and then elaborate on the ingredients over the next three sections.
  
First of all, we will make crucial use of the type of structure that occurred in \S \ref{Legendre_section}, to which we give a name:
    
\begin{dff} \label{aff family}
An \emph{abelian-by-finite family} over $Y$ is
a sequence of morphisms
$$X \longrightarrow Y' \stackrel{\pi} \longrightarrow Y$$
where $\pi$ is finite {\'e}tale, and $X \rightarrow Y'$ is (equipped with the structure of) a polarized abelian scheme.

A \emph{good model} for such a family,
over an $S$-integer ring $\mathcal{O} \subset K$,  is a family $\mathcal{X} \rightarrow \mathcal{Y}' \rightarrow \mathcal{Y}$
of smooth, proper $\mathcal{O}$-schemes,
satisfying the same conditions  and also 
 the assumptions at the start of \S \ref{bas note},    and  recovering $X \rightarrow Y' \rightarrow Y$ on base change to $K$. 
\end{dff}
Of course the polarization on $X \rightarrow Y'$ is an additional structure but for brevity we do not explicitly include it in the notation. 
 
 For any such abelian-by-finite family $X \rightarrow Y' \rightarrow Y$
take a  complex point $y_0 \in Y(\C)$ and 
consider the action of the topological fundamental group
$\pi_1(Y(\C),  y_0)$ on  
$$ H^1_B(X_{y_0}, \Q) \simeq \bigoplus_{\pi(\tilde{y}) = y_0} H^1_B(X_{\tilde{y}}, \Q),$$
where the sum is taken over $\tilde{y} \in Y'(\C)$ lying over $y_0$. 
We say that the family has {\em full monodromy}
if the Zariski closure of $\pi_1(Y, y_0)$, in its action on the right-hand side, 
contains the product of symplectic groups:  
\begin{equation}
\label{fullmonodromy}
\overline{  \left( \mbox{image of $\pi_1(Y(\C), y_0)$} \right)} \supset  \prod_{\pi(\tilde{y})=y_0} \mathrm{Sp}\left( H^1_{B}(X_{\tilde{y}}, \Q), \omega \right),
\end{equation}
where the symplectic group is with reference to the form $\omega$ defined by the polarization.  

The key reason to use abelian-by-finite families is that we can guarantee that the Galois orbits
on any fiber of $Y' \rightarrow Y$, above a $K$-rational point of $Y$, are ``large.''  In fact, what we need
(see discussion in Introduction) is that most points in the fiber above $y_0 \in Y(K)$
cannot be defined over ``small'' extensions of $K_v$. 
To quantify the notions of large and small we introduce the following quantity:  
 
 \begin{dff}
Let $E$ be a $G_K$-set and $v$ a place of $K$ such that the $G_K$-action on $E$ is unramified at $v$. 
Let
\begin{equation} \label{sizedef} \size(E) =  \frac{ \mbox{number of elements of $E$ that belong to $\mathrm{Frob}_v$-orbits of size $< \five$}}{\mbox{number of elements of $E$}}\end{equation}
If $E$ is a zero-dimensional $K$-scheme, we will write $\size(E)$ instead of $\size(E(\bar{K}))$. 
\end{dff}
Note that if $E \rightarrow E'$ is a morphism of $G_K$-sets, and all fibers have the same cardinality, then 
\begin{equation} \label{shrink} \size(E)  \leqslant \size(E').\end{equation} 

The next result is, in essence, a variant of Proposition \ref{finitesetoforbits},
but it requires some careful indexing.  It will be proved in \S \ref{highergenus}.

\begin{proposition}
\label{finitepointsoncurves}
 Let $Y$ be a curve over $K$ of genus $g \geqslant 2$. 
 
 Let $X \rightarrow Y' \stackrel{\pi}{\rightarrow} Y$ be an abelian-by-finite family over $Y$, 
 with full monodromy  (see Definition \ref{aff family} and subsequent discussion).
Let $d$ be the relative dimension of $X \rightarrow Y'$. 
 Suppose that $X \rightarrow Y' \stackrel{\pi}{\rightarrow} Y$ admits a good model over the ring $\mathcal{O}$ of $S$-integers of $K$.
Let $v \notin S$ be a friendly place of $K$ (Definition \ref{friendly def}). 
 
  Let $\size$ be as in \eqref{sizedef}. 
 Then the set
 $$ Y(K)^* :=   \left \{y \in Y(K): \size(\pi^{-1} (y))   < \frac{1}{d+1}\right \}$$
 is finite. 
\end{proposition}

In \S \ref{rationalpoints}, we introduce a specific abelian-by-finite family $X_q \rightarrow Y_q' \stackrel{\pi}{\rightarrow} Y$
for each prime $q \geqslant 3$, referred to as the ``Kodaira--Parshin family for the group $\Aff(q)$.''
Roughly, $Y_q'$ is a Hurwitz space for $\Aff(q)$ and $X_q$ is the Prym of the universal curve.
It has the following properties: \label{KP prop}
\begin{itemize}  
\item[(i)] It has full monodromy (Theorem \ref{monodromy theorem}).  
\item[(ii)]
The relative dimension $d_q$ of $X_q \rightarrow Y_q'$ is given by $d_q = (q-1)(g-\frac{1}{2})$.
\item[(iii)] For each $y_0 \in Y(K)$  there is a $G_K$-equivariant identification
of $\pi^{-1}(y_0)$ with the conjugacy classes of  surjections $\pi_1^{\mathrm{geom}}(Y-y_0, *) \twoheadrightarrow \Aff(q)$
that are nontrivial on a loop around $y_0$.   
\end{itemize}
 
Note that we can identify $\pi_1^{\mathrm{geom}}$
with the profinite completion of a free group on $2g$ generators $x_1, x_1', \dots, x_g, x_g'$ %
in such a way that the loop around $y_0$ corresponds to the conjugacy class of $[x_1, x_1'] [x_2, x_2'] \dots [x_g, x_g']$. 
Therefore, the  set of surjections $\pi_1^{\mathrm{geom}}(Y-y_0, *) \twoheadrightarrow \Aff(q)$
nontrivial on a loop around $y_0$ 
is identified with the left-hand side of \eqref{im im}. 
  
There is probably nothing very special about the use of $\Aff(q)$, but it is simple enough that we can compute everything explicitly.  
  
Assuming these things we can prove:  
\begin{thm} \label{main theorem}
Let $Y$ be a curve over the number field $K$ with genus $g \geqslant 2$. Then $Y(K)$ is finite.
\end{thm}

\proof 

We apply Proposition \ref{finitepointsoncurves} to the Kodaira--Parshin family with parameter $q$.  What we will 
show is that we may choose $q$ and the place $v$ in such a way that $v$ is friendly and
\begin{equation} \label{empty} \size(\pi^{-1} (y)) < \frac{1}{d_q+1} \ \  \mbox{ for all  $y \in Y(K)$.} \end{equation}
The key point is to use the mapping \eqref{surj} below and the Weil pairing to give an upper bound on $\size(\pi^{-1}(y))$. 

We choose $q$ 
with the following properties:  
\begin{itemize}
\item[(i)] $q-1$ is not divisible by $4$ or by any odd primes less than $\five [K:\Q]$.
\item[(ii)]  The Galois closure $K'$ of $K$ is linearly disjoint from $\Q(\zeta_{q-1})$ over $\Q$. 
\item[(iii)] $ \frac{\five \cdot 2^{g+1}}{(q-1)^{g}} < \frac{1}{ (g-1/2) (q-1) +1}$.
\end{itemize}
This is possible by Dirichlet's theorem: 
we choose $q$ such that $q$ is not congruent to $1$ mod $\ell$
for any prime $\ell$ that either
divides the discriminant of $K$, or that is less than $\five [K:\Q]$, and also $q$ is not congruent to $1$ mod $4$. 
Then linear disjointness follows: for ramification reasons  $K' \cap \Q(\zeta_{q-1}) = \Q$. 
Such a $q$ can be chosen arbitrarily large; in particular it can be chosen to satisfy the third condition. 

Now  form the Kodaira--Parshin family $X=X_q \longrightarrow Y'_{q} \longrightarrow Y$
for the 
group $\Aff(q)$ and choose a set $S$ such that it has a good model over the ring
of $S$-integers.

Next we show that there exists a place $v \notin S$ of $K$ such that: 
\begin{itemize}
\item[(i)] $v$ is friendly (in the sense of Definition \ref{friendly def}) %
\item[(ii)] $(q_v, q-1)=1$ (recall that $q_v$ was the cardinality of the residue field at $v$)
\item[(iii)] For any odd prime factor $r$ of $q-1$,
 the class of $q_v$ in $(\Z/r)^*$ has order at least $\five$. 
 \end{itemize}
Note that the latter two conditions depend only on the residue class of $q_v$ modulo $q-1$.
We will produce $v$ by the Chebotarev density theorem, applied to $\Gal(K'(\zeta_{q-1})/\Q)$.
By hypothesis, $K'$ and $\Q(\zeta_{q-1})$ are linearly disjoint over $\Q$, so the map
$$ \Gal(K'(\zeta_{q-1})/\Q) \longrightarrow \Gal(K' / \Q) \times  \Gal (\Q (\zeta_{q-1}) / \Q).$$ 
is an isomorphism.

If $K$ has no CM subfield, choose $\sigma \in \Gal(K'/\Q)$ arbitrarily. Otherwise 
let $E$ be the maximal CM subfield of $K$, and let $E^+$ the maximal totally real subfield;
choose some $\sigma \in \Gal(K' / E^+) \subseteq \Gal(K'/\Q)$ inducing the nontrivial automorphism of $E$ over $E^+$.

By the Chinese Remainder Theorem, we can choose a residue class $a \in (\Z/(q-1))^*$ whose reduction modulo $r$ is a primitive root for $(\Z/r)^*$ for every prime factor $r$ of $(q-1)$.

By Chebotarev density, there is
a place $\wp$ of $K'(\zeta_{q-1})$ such that the Frobenius $\mathrm{Frob}_{\wp}$ 
is the element $(\sigma, a)$ of $\Gal(K'(\zeta_{q-1})/\Q) \simeq \Gal(K' / \Q) \times (\Z / (q-1))^*$.
Let $p$ be the prime of $\Q$ below $\wp$; thus $p \equiv a$ modulo $q-1$.  The place $v$ of $K$ below $\wp$ has residue field of size $q_v  = p^i$,
with $i \leqslant [K: \Q]$;
therefore, if $r$ is an odd prime factor of $(q-1)$, the order of $\mbox{$q_v$ mod $r$}$
is at least $\left \lceil \frac{r-1}{[K:\Q]} \right \rceil \geqslant \five$.  For the last inequality we used property (i) of $q$. 

If $K$ admits a CM subfield then the place of $E^+$ below $\wp$ is inert in $E$, by choice of $\sigma$. 
This shows that there indeed exists $v$ as desired. 

Now consider the Kodaira--Parshin family $X=X_q \longrightarrow Y'_{q} \longrightarrow Y$
for the group $\Aff(q)$ 
and write $d_q$   for the relative dimension of $X \rightarrow Y$. 
For any $y \in Y(K)$
property (iii) of Kodaira--Parshin covers (page \pageref{KP prop}), and the surjection $\Aff(q) \twoheadrightarrow \F_q^* \simeq \Z/(q-1)$,  gives rise to  a map
of $G_K$-sets
\begin{equation}
\label{surj}
\pi^{-1}(y)  \longrightarrow  \underbrace{ H^1_{\et}(Y_{\bar{K}},  \Z/(q-1)).}_{M}
\end{equation}
Let $\Upsilon \subseteq M$ be the image of the map. In explicit coordinates,
the map \eqref{surj} has been studied in Lemma \ref{com map} (see also remark after (iii) on page \pageref{KP prop}). 
Therefore, by Lemma \ref{com map}, all fibers of the map have the same size.
Therefore, in view of \eqref{shrink}, it is enough to show that
$\size(\Upsilon)  <  \frac{1}{d_q+1}.$
 
Now $M$ has the structure of a $(2g)$-dimensional free module over $\Z/(q-1)$.
On choosing an identification of $M$ with $(\Z / (q-1))^{2g}$, the set $\Upsilon$
consists of those elements $(y_1, y_1', \dots, y_g, y_g')$ such that
the elements $y_1, y_1', \ldots, y_g, y_g'$ generate $(\Z / (q-1))$.
(This is shown in the proof of Lemma \ref{com map}.)

$M$ is also equipped with a Galois-equivariant Weil pairing 
$$\langle -, - \rangle: M  \times M \rightarrow \mu_{q-1}^{\vee} := \Hom(\mu_{q-1}, \Z/(q-1)\Z).$$
The Weil pairing is perfect, i.e.\ the corresponding map 
$M \rightarrow \Hom(M, \mu_{q-1}^{\vee})$
is an isomorphism.
The Frobenius at $v$ induces, in particular, an automorphism $T: M \rightarrow M$ that satisfies
$$ \langle T v_1, T v_2 \rangle = q_v^{-1} \langle v_1, v_2 \rangle.$$ 

We want to bound the number of elements of $M$ belonging to $T$-orbits of size less than $\five$. %
These elements are contained in the union of the submodules $\ker(T^i-1)$  for $1 \leqslant  i  \leqslant \five$. 
If $m_1, m_2 \in \ker(T^i-1)$ then $(q_v^{-i} - 1) \langle m_1, m_2 \rangle = 0$.  For every odd prime factor $r$ of $q-1$ we know that
$q_v^i$ is not congruent to $1$ modulo $r$; therefore
$(q_v^i-1)$ is relatively prime to $r$. 
Thus $2\langle m_1, m_2 \rangle = 0$
for any $m_1, m_2 \in \ker(T^i-1)$.

Now if $A$ is a finite abelian group endowed with a nondegenerate pairing $A \times A \rightarrow \Q/\Z$  
 then any subgroup $B \subset A$ such that $\langle B, B \rangle = 0$ has order at most $\sqrt{A}$.
Applying this to $2M$ we find 
$$\left | 2\ker(T^i-1) \right | \leqslant \left( \frac{q-1}{2} \right)^g \implies \left| \ker(T^i-1) \right | \leqslant 2^g (q-1)^g.$$
Hence, the number of elements of $M$ contained in the union of the submodules $\ker(T^i-1)$ for $1 \leqslant  i  \leqslant \five$ is at most $\five \cdot 2^{g} (q-1)^g$.

It remains to give an upper bound for the ``$\size$'' of $\Upsilon$, the image of \eqref{surj}. The number
of generating $(2g)$-tuples in $\Z/N$ equals
$\# (\Z/N)^* \times \mathbf{P}^{2g-1}(\Z/N)$, which equals 
 $ N^{2g} \cdot \prod_{p|N} (1-p^{-2g}) \geqslant \frac{1}{2} N^{2g}.$
So $\Upsilon$ has at least $\frac{1}{2} (q-1)^{2g}$ elements, of which at most $\five \cdot 2^{g} (q-1)^g$ belong to
Frobenius orbits of size $\five$ or smaller.
It follows that
$$\size(\pi^{-1} (y))  \stackrel{\eqref{shrink}}{\leqslant} \size(\Upsilon)  \leqslant \frac{\five \cdot 2^{g} (q-1)^g}{\frac{1}{2} (q-1)^{2g}  } =   \frac{\five \cdot 2^{g+1}}{(q-1)^{g}} < \frac{1}{\underbrace{(g-1/2) (q-1)}_{d_q} +1},$$
the last inequality by property (iii) of the prime $q$. This 
concludes the proof of \eqref{empty}.  
\qed

\newcommand{\bad}{W}
\section{Rational points on the base of an abelian-by-finite family}
\label{highergenus}

In this section we
prove Proposition \ref{finitepointsoncurves}, which
is in essence
a variant of Proposition \ref{finitesetoforbits}, and which we rewrite for the reader's convenience.       
\begin{mainprop*}
 Let $Y$ be a curve over $K$ of genus $g \geqslant 2$. 
 
 Let $X \rightarrow Y' \stackrel{\pi}{\rightarrow} Y$ be an abelian-by-finite family over $Y$, 
 with full monodromy  (see Definition \ref{aff family} and subsequent discussion).
Let $d$ be the relative dimension of $X \rightarrow Y'$. 
 Suppose that $X \rightarrow Y' \stackrel{\pi}{\rightarrow} Y$ admits a good model over the ring $\mathcal{O}$ of $S$-integers of $K$.
Let $v \notin S$ be a friendly place of $K$ (Definition \ref{friendly def}). 
 
  Let $\size$ be as in \eqref{sizedef}. 
 Then the set
 $$ Y(K)^* :=   \left \{y \in Y(K): \size(\pi^{-1} (y))   < \frac{1}{d+1}\right \}$$
 is finite. 
\end{mainprop*}

Here's what happens in the proof. There are two central lemmas, Lemmas \ref{claim1} and \ref{claim2}.
\begin{itemize}
\item 
The assumption that $\size(\pi^{-1} (y)) < \frac{1}{d+1}$ guarantees that most points in the fiber $\pi^{-1} (y)$
are defined over fields of large degree over $\Q_p$.
As discussed in \S \ref{control_centralizer}, 
we will use the fact that an extension $K_v$ of $\Q_p$ is of large degree 
to bound the centralizer of Frobenius for a variety defined over $K_v$. 

Some care is required with indexing since we only have \emph{most} points;
in particular, we need to identify the fibers over $p$-adically nearby points $y$.
The discussion of indexing occupies the first part of the proof;
the bound on the Frobenius centralizer is in the proof of Lemma \ref{claim2}.
\item Lemma \ref{claim1} handles 
 the possible failure of semisimplicity (see  discussion in \S \ref{handle_nonsemisimple}). 
 As in Lemma \ref{sunit_nonsimple}, we use constraints on Hodge weights coming from global representations
(Lemma \ref{globalsimple}) to show that only finitely many fibers can give rise to non-semisimple
Galois representations.   This requires
a general position argument in linear algebra (Lemma \ref{finaux}). 
\end{itemize}
 
\proof  
Through the proof, we denote by $p$ the prime of $\Q$ below $v$; ``Tate module'' always means ``$p$-adic Tate module,'' 
and ``{\'e}tale cohomology'' means geometric {\'e}tale cohomology taken with $\Q_p$ coefficients.
 
Recall  also that we have fixed an algebraic closure $\overline{K}$ with Galois group $G_K$. 
Fix 
an extension of $v$ to that field; the completion of $\overline{K}$ gives an algebraic closure $\overline{K_v}$
of $K_v$.   In particular, if $L \subset \overline{K}$ is unramified at $v$, we obtain a Frobenius element
$\Frob_v \in \mathrm{Gal}(L/K)$. 

Fix $y_0 \in Y(K)^*$.    It is sufficient to show that there are only finitely many points of $Y(K)^*$   that lie in the residue disk $$ \Omega_v = \{y \in Y(K_v): y  \equiv y_0 \mbox{ modulo } v\},$$
which we are regarding as a $K_v$-analytic manifold.  

For each $y \in Y(K)$, 
let $E_y$ be the ring of regular functions on the zero-dimensional scheme $\pi^{-1}(y)$;
this is an {\'e}tale $K$-algebra and
$\Hom(E_y, \bar{K})$ is identified with the $G_K$-set
$\pi^{-1}(y)_{\bar{K}}$ 
of preimages of $y$ under $\pi$.
By our assumptions, the $G_K$-set $\pi^{-1}(y)_{\bar{K}}$ is unramified at $v$. Write $E_0$ for $E_{y_0}$.

The fiber $X_y$ of $X \rightarrow Y$ above $y \in Y(K)$ 
is {\em a priori} a $K$-scheme, but the factorization $X \rightarrow Y' \rightarrow Y$
gives it the structure of an $E_y$-scheme;
in particular its de Rham cohomology $H^1_{\dR}(X_y/K)$
has the structure of a free $E_y$-module.  Moreover,  the polarization on $X$
induces an {\em $E_y$-bilinear} symplectic pairing
$$H^1_{\dR}(X_y/K) \times H^1_{\dR}(X_y/K) \longrightarrow E_y.$$

Write  $E_{0,v} = E_0 \otimes_{K} K_v$, and $V_v := H^1_{\dR}(X_{y_0}/K_v)$. Then $V_v$ is a free $E_{0,v}$-module equipped 
with an ($E_{0,v}$-bilinear) symplectic form.   Denote by $\mathcal{H}_v \subset \mathcal{G}_v$  the $K_v$-schemes  defined by Weil restriction: %
$$ \mathcal{G}_v = \mathrm{Res}^{E_{0,v}}_{K_v} \  \mathrm{Gr}(V_v,g)$$
$$ \mathcal{H}_v =  \mathrm{Res}^{E_{0,v}}_{K_v} \  \LGr(V_v,\omega).$$
Here $\mathrm{Res}^{E_{0, v}}_{K_v}$ denotes Weil restriction of scalars, $\mathrm{Gr}(V_v,g)$ classifies free $E_v$-submodules of rank $g$ inside $V_v$, and $\LGr$ classifies free rank-$g$ submodules on which  the symplectic pairing is trivial.

Then the period map at $y_0$ gives a $K_v$-analytic function
$$\Phi_v:  \Omega_v \longrightarrow \mathcal{H}_v$$
(see \S \ref{GMsubsec} for a more detailed discussion 
of the radius of convergence; in particular it defines a rigid analytic function on a domain
containing $\Omega_v$, i.e.\ the $K_v$-points in a residue disk). 

{\em A priori}, this period mapping is valued in a suitable Lagrangian Grassmannian
of $K_v$-linear subspaces inside $V_v$, but, just as in the discussion of \S \ref{proof of finiteness Sunit}, each of these Lagrangian subspaces are actually $E_{0,v}$-stable, so that the period mapping
 actually takes values inside $\mathcal{H}_v$. 
Lemma \ref{Zarclos}, and the assumption of full monodromy, 
imply  that $\Phi_v(\Omega_v)$ is Zariski-dense in $\mathcal{H}_v$.  

To proceed further, as we discussed in the proof sketch, we need to carefully index the points above $y$. 
Firstly, $E_y$ decomposes as a product of fields:
$$ E_y = \prod_{y'} K(y')$$
where the product is over points 
$y'$ of the scheme $Y'$ lying above $y$.
For any such $y'$, the fiber $X_{y'}$ of $X \rightarrow Y'$ above $y'$ is a $d$-dimensional abelian variety over the field $K(y')$;
write $\rho_{y'}$ for the corresponding $2d$-dimensional $p$-adic Galois representation of the absolute Galois group of $K(y')$. 

The base change $E_y \otimes_K K_v$ splits as a product of fields
\begin{equation} \label{first splitting} E_{y} \otimes_{K} K_v =  \prod_{y', w} K(y')_w\end{equation}
indexed by pairs $(y', w)$,
where $y'$ is a closed point of $\pi^{-1}(y)$ as above,  and $w$ is a place of $K(y')$ over $v$.
In this situation we will say, for short, that
$(y', w)$ is above $(y,v)$.   

Write $X_{y',w}$ 
for the base change of $X_{y}$ along $E_y \rightarrow K(y')_w$,
and $\rho_{y',w}$ for the $G_{K(y')_w}$-representation
on its {\'e}tale cohomology.
The de Rham cohomology $V_v = H^1_\dR(X_{y} / K_v)$ over $K_v$
splits as a product 
\begin{equation} \label{second splitting} V_v = \prod_{y',w} V_{y',w},  \ \ V_{y',w} = H^1_{\dR}(X_{y',w}/K(y')_w)\end{equation}
in a fashion that is compatible with the $E_{y} \otimes_K K_v$-module structure and \eqref{first splitting}.
The dimension of each $V_{y',w}$ over $K(y')_w$ is the same, namely, $2d$. 

Crystalline cohomology of the reduction modulo $v$ (or, phrased differently, the Gauss--Manin connection for $Y' \rightarrow Y$)
gives an isomorphism
\begin{equation} \label{hh1} E_y \otimes_{K} K_v \stackrel{\sim}{\longrightarrow} E_{0,v} = E_0 \otimes_{K} K_v\end{equation}
whenever $y$ belongs to the residue disk $\Omega$ of $y_0$.  
In particular this induces a bijection
\begin{equation} \label{hh} \ \mbox{$(y', w)$ above $(y,v)$} \stackrel{\sim}{ \longleftrightarrow}
\mbox{$(y_0',w_0)$ above $(y_0, v)$}  
\end{equation}
 since both sides are identified with the spectrum of the common algebra of \eqref{hh1}. 
 Moreover, 
the identification \eqref{hh1} is compatible with  the Gauss-Manin isomorphism
\begin{equation} \label{GM1} H^1_{\dR}(X_y/K_v)  \stackrel{\mathrm{GM}}{\longrightarrow} H^1_{\dR}(X_{y_0}/K_v).\end{equation}

If $(y', w)$ corresponds to $(y_0', w_0)$ under this identification, then  \eqref{hh1}
and \eqref{GM1} induce  
\begin{equation} \label{bij iso} K(y')_w \simeq K(y_0')_{w_0}, \ \  H^1_{\dR}(X_{y', w}/K(y')_w) \simeq H^1_{\dR}(X_{y_0', w_0}/K(y_0')_w).\end{equation}

Also, \eqref{second splitting} induces the splitting of the variety $\mathcal{H}_v$  as a product
 $$\mathcal{H}_v = \prod_{(y_0', w)}   \mathcal{H}_{(y_0', w)},$$
where  the product is taken over $(y_0', w)$ above $(y_0, v)$, and where 
$$ \mathcal{H}_{(y_0', w)} = \mathrm{Res}^{K(y_0')_w}_{K_v} \  \LGr(V_{y_0', w},\omega).$$ 
We have a similar decomposition $\mathcal{G}_v = \prod_{(y_0', w)} \mathcal{G}_{(y_0', w)}.$

If $y \in \Omega_v$, and if $(y',w)$ above $(y, v)$ corresponds to $(y_0', w_0)$ under \eqref{hh}, then 
\begin{equation} \label{Hodgepodge} \mbox{projection to $\mathcal{H}_{(y_0', w_0)}$ of $\Phi_v(y)$} =  F^1  H^1_{\dR}(X_{y',w}).\end{equation}
where we identify $F^1 H^1_{\dR}(X_{y', w})$ with a Lagrangian in the $K(y_0')_{w_0}$-vector space $V_{y_0', w_0}$ using the Gauss--Manin connection \eqref{bij iso}. 
This result \eqref{Hodgepodge} comes down to the fact, already noted, that \eqref{hh} and \eqref{GM1} are compatible. 
  
We will establish the following two lemmas. 

\begin{lemma}[Generic Simplicity] \label{claim1}
There is a finite subset $F \subset \Omega_v \cap Y(K)^*$ such that, for $y \in \left( \Omega_v \cap Y(K)^* \right) - F$, 
there exists $(y',w)$ above $(y,v)$ such that:
\begin{itemize}
\item[(i)] $[K(y')_w:K_v] \geqslant \five$ 
\item[(ii)] $\rho_{y'}$ is   simple as a $G_{K(y')}$-representation.
\end{itemize}
\end{lemma} 
 
Observe that  for $y \in \left(\Omega_v \cap Y(K)^* \right) - F$, 
and $y'$ above $y$, there are only finitely many possibilities
for the isomorphism class of the field $K(y')$. 
Thus, by Lemma \ref{finiteness}, there are only finitely many possibilities
for the isomorphism class of the pair $(K(y)', \rho_{y'})$,
and so also only finitely many possibilities for the isomorphism class of any pair $(K(y')_w, \rho_{y'}|_{K(y)'_w})$
arising from $(y', w)$ as in Lemma \ref{claim1}.  The proof of Proposition 
\ref{finitepointsoncurves} will then follow from Lemma \ref{claim1} above and Lemma \ref{claim2} below. \qed

\begin{lemma}[Galois representations really do vary in our family] \label{claim2} 
Fix a finite field extension $K'_v$ of $K_v$, with $[K'_v: K_v] \geqslant \five$, 
and a Galois representation $\rho'$ of the absolute Galois group of $K'_v$.  

There are only finitely many $y \in \Omega_v \cap Y(K)$ for which 
there exist $(y',w)$ satisfying conditions (i) and (ii) of Lemma \ref{claim1}
and moreover the pair
$$ (K(y')_w, \rho_{y', w}) \mbox{ is isomorphic to }(K'_v,\rho')$$
i.e.\ there is an isomorphism $K(y)'_w \rightarrow K'_v$
carrying the isomorphism class of $\rho'$ to that of $\rho_{y',w}$. 
\end{lemma}
 
To prove Lemmas \ref{claim1} and \ref{claim2} we shall analyze the period mapping more carefully. 

\proof[Proof of Lemma \ref{claim2}]    
Under the correspondence of $p$-adic Hodge theory, $\rho_{y',w}$ corresponds to 
the $K(y')_w$-vector space $H^1_{\dR}(X_{y'}/K(y')_w)$, together with its natural semilinear Frobenius operator $\phi$, and 
the (two-step) filtration defined by $F^1 H_{\dR}(X_{y'}/K(y')_w)$.   

Suppose that $(y', w)$ corresponds to $(y_0', w_0)$ under \eqref{hh}. 
Using the isomorphism \eqref{bij iso} the triple just described corresponds to 
$$ (H^1_{\dR}(X_{y_0'}/K(y_0')_{w_0}), \mbox{$\phi_v$ = semilinear Frobenius}, \ \mbox{projection of $\Phi_v(y)$ to $\mathcal{H}_{y_0', w_0}$}). $$
It is enough to show that the set of $y$, for which this triple belongs to a fixed isomorphism class, is finite. 

Belonging to a fixed isomorphism class means that the projection of  $\Phi_v(y)$ to $\mathcal{H}_{y_0', w_0}$ lies inside a single orbit for the action of the Frobenius centralizer $Z(\phi_v)$ on $\mathcal{G}_{y_0', w_0}$,
and so also a single orbit of $Z(\phi_v^{[K_v:\Q_p]})$ on $\mathcal{G}_{y_0', w_0}$.  (In both cases, these centralizers are   taken inside $K(y_0')_{w_0}$-linear automorphisms of $V_{y_0', w}$.)

Apply Lemma \ref{centralizer lemma} to the field extension $K(y_0')_w/K_v$ %
to see that this Frobenius centralizer has $K_v$-dimension at most 
$(\dim_{K(y_0')_{w_0}} V_{y_0', w_0})^2=4d^2$.

As noted earlier,  the period map $\Phi_v$ has Zariski-dense image (in the $K_v$-variety $\mathcal{H}_v$; therefore this remains true when projected to $\mathcal{H}_{y_0', w_0}$).
Since $\dim_{K_v} \mathcal{H}_{y_0', w_0} = [K(y')_{w}: K_v] \cdot \frac{d(d+1)}{2}  \geqslant 4 d(d+1) >4d^2$,
Lemma \ref{Zarclos} completes 
the proof of Lemma \ref{claim2}.
\qed

\proof[Proof of Lemma \ref{claim1}]   
Let us call $y \in Y(K)^* \cap \Omega_v$ ``bad''
when, for every  $(y',w)$ above $(y,v)$ such that
$[K(y')_w:K_v] \geqslant \five$,  the representation $\rho_{y'}$ fails to be simple. 
We must show there are only finitely many bad $y \in Y(K)^* \cap \Omega_v$. 
  
\begin{quote} {\bf Sublemma:}  If $y \in Y(K)^* \cap \Omega_v$ is bad, there exists:
\begin{itemize}
\item[-]  $(y', w)$ above $(y,v)$, with  $[K(y')_w:K_v] \geqslant \five$;    \item[-] a  nonzero proper Frobenius-stable subspace  $\bad^{\dR}_{y',w}$ of  $H^1_{\dR}(X_{y'}/K(y')_w)$ 
such that  $\dim F^1 \bad^{\dR}_{y',w} \geqslant \dim(\bad^{\dR}_{y',w})/2$. (Here, and in the discussion below, dimensions are dimensions over $K(y')_w$.)
\end{itemize} 
\end{quote} 

\begin{quote}
{\em Proof of sublemma:}  Take a bad $y \in Y(K)^* \cap \Omega_v$. 
For each $y'$ above $y$  let $\bad_{y'}$ be a nonzero 
subrepresentation of $\rho_{y'}$ of minimal positive dimension.   (It is therefore possible that $\bad_{y'}$ is all of $\rho_{y'}$). 
For each place $w$ of $K(y')$ we define
$\bad^{\dR}_{y',w}$ by applying $p$-adic Hodge theory to $\bad_{y'} \leqslant \rho_{y'}$;
thus $\bad^{\dR}_{y', w}$ is a $\phi$-stable submodule of $H^1_{\dR}(X_{y'}/K(y')_w)$. 
   
Note that 
\begin{equation} \label{st}  \dim \bad^{\dR}_{y',w} \leqslant  d \mbox{ whenever $[K(y')_w:K_v] \geqslant \five$.}\end{equation}
Indeed because $y'$ is bad, the supposition $[K(y')_w:K_v] \geqslant \five$ forces $\rho_{y'}$ to be non-simple;
because it preserves (up to similitude) a bilinear form, we have 
  $\dim \bad_{y'} \leqslant \frac{1}{2} \dim \rho_{y'}$,
thus \eqref{st}. 
 
Now assume that, for each $(y',w)$ 
above $(y,v)$,  satisfying $[K(y')_w: K_v] \geqslant \five$, 
we have $$ \dim F^1 \bad^{\dR}_{y',w}  < \frac{1}{2} \dim \bad^{\dR}_{y',w}.$$
We will derive a contradiction, which will conclude the proof. 

By Lemma \ref{globalsimple2}, applied to $\bad_{y'}$ as a Galois representation of $K(y')$, we have
\begin{equation} \label{above1} \sum_{w|v}   [K(y')_w:K_v]  \frac{ \dim F^1 \bad^{\dR}_{y',w}}{\dim \bad^{\dR}_{y',w}}   =  \frac{1}{2} [K(y'):K]\end{equation}
for any $y'$ a closed point of $\pi^{-1}(y)$.    Sum over $y'$ above $y$;  using
\eqref{st} we get 
  {\small   \begin{equation} \label{above2}  \sum_{[K(y')_w:K_v] \geqslant \five}    [K(y')_w:K_v]   \left(\frac{1}{2} - \frac{1}{2 d}\right)  
   +   \sum_{[K(y')_w:K_v] < \five}    [K(y')_w:K_v] \geqslant  \frac{1}{2}  \sum_{(y',w)} [K(y')_w:K_v] \end{equation}}
Here 
all summations are over $(y', w)$ above $(y,v)$. 
Therefore, 
     {\small   \begin{equation} \label{above3} 
      \sum_{[K(y')_w:K_v] < \five}   \frac{1}{2}   [K(y')_w:K_v] \geqslant   \frac{1}{2d}\sum_{[K(y')_w:K_v] \geqslant \five}    [K(y')_w:K_v]     .\end{equation}}
      Let $e_1, \dots, e_k$ be the cycle structure of $\Frob_v$ acting on the $\bar{K}$ points of $\pi^{-1}(y)$. The inequality above means that
      $$  \frac{1}{2} \sum_{i: e_i  < \five} e_i \geqslant \frac{1}{2d} \sum_{i: e_i \geqslant \five} e_i,$$
      which is to say that $\size(\pi^{-1} y) \geqslant \frac{1}{d+1}$. This contradicts the assumption that $y \in Y(K)^*$. 
   \qed
 \end{quote}      

We now return to the proof of Lemma \ref{claim1}.  Fix any $(y_0', w)$ above $(y_0, v)$
with $[K(y_0')_w:K_v] \geqslant \five$.  Such a $(y_0', w)$ exists because of the assumption that
$y_0 \in Y(K)^*$.
In view of the {\em Sublemma} and \eqref{Hodgepodge}, it is enough to show that there are only finitely many $y \in Y(K) \cap \Omega_v$ 
such  the  projection of $\Phi_v(y)$ to $\mathcal{H}_{(y_0',w)}$ lies in  the  subvariety
$$\mathcal{H}_{(y_0', w)}^{\mathrm{bad}}   \subset \mathcal{H}_{(y_0', w)}$$
defined as the Lagrangian, $K(y_0')_w$-subspaces 
$F \subset V_{y_0',w}$ (recall \eqref{second splitting} for definition) for which there exists a Frobenius-stable   subspace $W \subset V_{y_0',w}$,
satisfying 
\begin{equation} \label{crit} \dim(F \cap W) \geqslant \frac{1}{2} \dim(W), \end{equation}

By the lemmas that follow, $\mathcal{H}_{(y_0', w)}^{\mathrm{bad}}$ is contained in a proper closed  $K_v$-subvariety of $\mathcal{H}_{(y_0', w)}$; %
we conclude as in the proof of Lemma \ref{claim2}.   \qed 

\begin{lemma} \label{generalsimple2}
Suppose $L_w$ is a finite unramified extension of $K_v$ of degree  $r \geqslant \five$.  
Let $(V, \omega)$ be a symplectic $L_w$-vector space, with $\dim_{L_w} V = 2d$;
let $\phi: V \rightarrow V$ be semilinear for the Frobenius automorphism of $L_w/K_v$ 
and bijective. 

Then there is a   Zariski-open $$\mathcal{A} \subseteq  \mathrm{Res}^{L_w}_{K_v} \  \LGr(V,\omega)$$
(where  $\LGr(V, \omega)$ is the Lagrangian Grassmannian, and $\mathrm{Res}^{L_w}_{K_v}$ denotes Weil restriction
of scalars from $L_w$ to $K_v$)
with the following property:

If $F \subset V$ is a Lagrangian $L_w$-subspace,  corresponding to a point of $\mathcal{A}(K_v)$,  there is no $\phi$-invariant $L_w$-subspace $W$ of $V$ satisfying
\eqref{crit}. \end{lemma}

\proof
Just as in Lemma \ref{centralizer lemma},
$V \otimes_{K_v} \overline{K_v}$
splits into $2d$-dimensional spaces $V_1, \dots, V_r$  indexed by embeddings $L_w \hookrightarrow \overline{K_v}$;
we can order them so that $\phi$ induces isomorphisms $V_i \simeq V_{i+1}$ for $1 \leq i \leq r-1$,
and  thus can identify them all with $V_1$ (we do not use the ``cyclic'' isomorphism $V_r \simeq V_1$). 

The base extension $W \otimes_{K_v} \overline{K_v}$ of any $\phi$-invariant $L_w$-subspace
yields a subspace $\bigoplus W_i \leqslant \bigoplus V_i$, where each $W_i$ corresponds to $W_1$ under the above identifications. 
Similarly, the base extension of a  Lagrangian $L_w$-subspace  $F \leqslant V$ 
gives an subspace $\bigoplus F_i \leqslant \bigoplus V_i$, where each $F_i$ is Lagrangian. 
If \eqref{crit} is satisfied, then $\dim(F_i \cap W_i) \geqslant \frac{1}{2} \dim(W_i)$ for each $1 \leqslant i \leqslant  r$.

The next, and final, Lemma shows that the set of $(F_1 \dots, F_r)$
for which such a $W$ exists is a proper, Zariski-closed subset. Thus 
 there is a Zariski-open set inside
in $$\left( \mathrm{Res}^{L_w}_{K_v} \  \LGr(V,\omega) \right) \times_{K_v} \overline{K_v}$$
such that,  if $F$ belongs to this Zariski-open, it has
the property quoted in the statement. Taking the intersection of Galois conjugates
of this set, we get the desired Zariski-open inside  $\mathrm{Res}^{L_w}_{K_v} \  \LGr(V,\omega)$. 
 \qed
 
\begin{lemma}  \label{finaux} Let $(V, \omega)$ be a symplectic vector space over a field of characteristic zero  
 with $\dim(V) =2d$; write $\mathrm{LGr}(V, \omega)$
for the Grassmannian of Lagrangian subspaces. 
  Let $\mathrm{E}$ be the set of $r$-tuples of Lagrangian subspaces
   $$ (F_1, \dots, F_r) \in \LGr(V, \omega)^r$$
for which there exists a  proper nonzero subspace $W \subset V$ such that $\dim(F_j \cap W) \geqslant \frac{1}{2} \dim(W)$ 
for every $j$.
If $r \geqslant \five$  then $\mathrm{E}$ is contained in a proper, Zariski-closed subset of $\LGr(V, \omega)^r$.  
\end{lemma}

\proof 
In fact our argument will show that $r \geqslant 5$ is enough.

First we argue that $\mathrm{E}$ is Zariski-closed.  Consider the product $\mathrm{Gr}(V) \times \mathrm{LGr}(V, \omega)^r$ parametrizing tuples $(W, F_1, F_2, \ldots, F_r)$ such that each $F_i$ is Lagrangian.  For each $i$, the dimension $\operatorname{dim} F_i \cap W$ is (Zariski) upper semicontinuous; so the set $\tilde{\mathrm{E}}$ of tuples satisfying the conditions described is closed.  Now $\mathrm{E}$ is the image of the closed set $\tilde{\mathrm{E}}$ under a proper map, so it is itself closed.

Since $\mathrm{E}$ is closed it's enough to produce a single tuple $(F_1, \ldots, F_r)$ not in $E$.
  
Take $e_1, \ldots, e_d, e_1', \ldots, e_d'$ a standard symplectic basis for $V$, so $\langle e_i, e_i' \rangle = 1$, and $\langle e_i', e_i  \rangle= -1$, and all other pairings between basis vectors are zero.  Let 
\begin{eqnarray*}
F_1 & = & \operatorname{span} (e_1, e_2, \ldots, e_d)\\
F_2 & = & \operatorname{span} (e_1', e_2', \ldots, e_d')\\
F_3 & = & \operatorname{span} (e_1 + e_1', e_2 + e_2', \ldots, e_d + e_d')\\
F_4 & = & \operatorname{span} (e_1 + 2 e_1', e_2 + 4 e_2', \ldots, e_d + 2d e_d').
\end{eqnarray*}
Now each of these four spaces is maximal isotropic, and any two of them have trivial intersection.  

Write $\pi_{12}: V \rightarrow F_1$ for the projection along the decomposition $V = F_1 \oplus F_2$,
and similarly define $\pi_{21}: V \rightarrow F_2$. 
Both $\pi_{12}$ and $\pi_{21}$ are isomorphisms when restricted to either $F_3$ or $F_4$. 
Write 
$\Phi_{12;3}: F_1 \rightarrow F_2$ for the isomorphism
$$ F_1 \stackrel{\pi_{12}^{-1}}{\longleftarrow} F_3 \stackrel{\pi_{21}}{\longrightarrow} F_2.$$
In explicit coordinates $\Phi_{12;3}$ takes $e_i$ to $e_i'$, and the similar map $\Phi_{12;4}$ takes $e_i$ to $2i e_i'$.

We claim that only finitely many $W$ can satisfy the condition  stated in the Lemma with respect to $F_1, F_2, F_3, F_4$. 
Suppose given such a $W$.  Since $W \cap F_1$ and $W \cap F_2$ have trivial intersection with each other, and they each have dimension at least $\frac{1}{2} \dim(W)$, we have  a direct sum decomposition
\begin{equation} \label{split}
W  =   \left(W \cap F_1\right) \oplus  \left(W \cap F_2\right)
\end{equation}
and an equality $\dim(W \cap F_1) = \dim(W \cap F_2) = \frac{1}{2} \dim{W}$.
Similarly, we find that $\dim(W \cap F_3) = \dim(W \cap F_4) = \frac{1}{2} \dim{W}$.

Next $\pi_{12}$ gives an isomorphism $F_3 \rightarrow F_1$; comparing dimensions, we see the restriction
$$ \pi_{12}: W \cap F_3  \stackrel{\sim}{\longrightarrow} W \cap F_1$$
is an isomorphism as well. 
Similarly $\pi_{21}: W \cap F_3 \stackrel{\sim}{\longrightarrow} W \cap F_2$. 
  
In particular, $\Phi_{12,3}$ carries $W \cap F_1$ isomorphically to $W \cap F_2$.
The same reasoning applies to $\Phi_{12,4}$. Therefore, $W \cap F_1$
is stable under
$\Phi_{12,4}^{-1} \Phi_{12,3}$, which shows that 
$W \cap F_1 \subseteq F_1$ is stable under the map $e_i \mapsto 2 i e_i$.  
  
There are then finitely many possibilities for $W \cap F_1$;
then there are also finitely many possibilities for $W \cap F_2 = \Phi_{12,3} (W \cap F_1)$
and then by \eqref{split} finitely many possibilities for $W$; call them 
$W_1, \ldots, W_N$.

Now, for  each $W_i$, the condition that $\dim(F_5 \cap W_i) \geqslant \frac{1}{2} \dim(W_i)$ cuts out a proper Zariski-closed subset of the Lagrangian Grassmannian parametrizing $F_5$; 
thus we may choose $F_5$ so that no $W$ satisfies the dimension bound.
\qed

\section{The Kodaira--Parshin family}

\label{rationalpoints}

The argument that we have given for Mordell's conjecture in 
Section  \ref{Mordelloutline}
made use of a specific abelian-by-finite family, the Kodaira--Parshin family. In this section
we explain how to construct this family, making use (in effect)  of an algebraic version of the theory of Hurwitz spaces. 
We need this theory only in characteristic zero. 
  
\subsection{Hurwitz spaces for curves} \label{Hurwitz1}
     \begin{prop} \label{hurwitz} 
  Let $Y$ be a curve of genus at least $2$ over a number field $K$, and let $G$  
  be a center-free finite group. 
  Then there is a $K$-curve $Y'$ equipped with an {\'e}tale map $\pi: Y' \rightarrow Y$, 
  and a relative curve $Z \rightarrow Y'$, with the following properties:
  \begin{itemize}
  \item[(i)] ``$Y'$ parameterizes $G$-covers of $Y$ branched at a single point'': For $y \in Y(\bar{K})$,
there is a bijection between $\pi^{-1}(y)$ and the set of $G$-conjugacy classes of surjections $\pi_1^{\mathrm{geom}}(Y- y,*)  \twoheadrightarrow G$
nontrivial on a loop around $y$.   
Moreover, if $y \in Y(K)$, this identification is $G_K$-equivariant. 
  \item[(ii)]   ``$Z$ gives the universal $G$-cover of $Y$ branched at a single point'':   
  There is a morphism $Z \rightarrow  Y' \times Y$
  of relative curves over $Y'$ (here, we are regarding $Y' \times Y$ as 
  the trivial family of curves over $Y'$, with fiber $Y$ everywhere). 
  
Moreover $G$ acts on $Z$ covering the trivial action on $Y' \times Y$. 
This action makes $Z \rightarrow Y' \times Y$ into a $G$-covering away from the graph of $\pi$.  
If we take the fiber of this morphism of relative curves above   $y' \in Y'(\bar{K})$, the resulting
   map $Z_{y'} \rightarrow Y$  of curves is ramified exactly at $\pi(y')$.  The induced homomorphism
   $$\pi_1^{\mathrm{geom}}(Y- y',y_0) \rightarrow \Aut_G(Z_{(y', y_0)}) \cong G$$  
is exactly (in the conjugacy class of) the surjection from (i) classified by $y'$. 
\end{itemize}
\end{prop}

There are several references on this matter that address much more general settings (e.g.  \cite[\S 3.22]{Mochizuki}) but
since none of them give the precise statement we need, we will simply  outline a direct proof,  descending from the complex analytic analogue, in  \S \ref{Hurwitzproof}. 
  
Now we apply this to the group $G =\Aff(q)$: 

\begin{dff} \label{KPF}
  Let $Y$ be a curve of genus at least $2$ over a number field $K$, and let $q$ be a prime number. 
  The    \emph{Kodaira-Parshin  curve family} over $Y$ with parameter $q$ will be
  the sequence of morphisms  
\begin{equation} \label{KP curve} Z_q \longrightarrow  Y'_{q} \longrightarrow Y,\end{equation} 
obtained from Proposition \ref{hurwitz}  applied to the group $G = \Aff(q)$.
\end{dff}

We now want to form an associated abelian-by-finite family to the Kodaira--Parshin curve family.

\subsection{Prym varieties}  \label{Prym1} %
We first describe the situation fiberwise: 

Given a morphism $C_1 \rightarrow C_2$ of curves over an algebraically closed field,
the associated \emph{Prym variety} is the cokernel of the induced map
$\mathrm{Pic}^0(C_2) \rightarrow \mathrm{Pic}^0(C_1)$ on Jacobians. 

Now suppose that that the covering $C_1 \rightarrow C_2$ is 
Galois, with Galois group $\Aff(q)$, and ramified over exactly one point of $C_2$.   The degree of this covering is $q(q-1)$.
Rather than take its Prym directly, however, 
we prefer to use a reduced version. Namely, 
we can form a smaller degree-$q$ covering $C_1' \rightarrow C_2$ using the permutation action of $\Aff(q)$ on $\Z/q\Z$,
and we are interested in the Prym variety of this associated covering: %
\begin{equation} \label{rpdef}  \mathrm{coker}(\Pic^0(C_2) \rightarrow \Pic^0(C_1'))\end{equation}
We emphasize again that this is {\em not} the Prym variety of $C_1 \rightarrow C_2$ but a ``reduced'' version of it
where the role of $C_1$ has been replaced by $C_1'$.

We can reformulate this in terms of $C_1$, rather than the associated curve $C_1'$. 
Both $\mathrm{Pic}^0(C_2)$ and $\mathrm{Pic}^0(C_1')$
map to $\mathrm{Pic}^0(C_1)$, with finite kernel. 
The image of $\mathrm{Pic}^0(C_2)$ in $\mathrm{Pic}^0(C_1)$
is now the connected component of the $\Aff(q)$-invariants; similarly   the image of $\mathrm{Pic}^0(C_1')$ in $\mathrm{Pic}^0(C_1)$
is the connected component of the invariants by the subgroup $H_{q} = (\Z/q\Z)^*$,
which  is a point stabilizer in the permutation  action of $\Aff(q)$ on $\Z/q\Z$. 
 In summary, then, the Prym variety of $C_1' \rightarrow C_2$ is  isogenous to 
cokernel of the  map %
$$  \mbox{connected component of  $\Pic^0(C_1)^{G_q}$} \rightarrow  \mbox{connected component of $\Pic^0(C_1)^{H_q}$}.$$
 This is an abelian variety of dimension $(2g-1) \cdot \frac{q-1}{2}$, isogenous to \eqref{rpdef}. 

We may alternately describe this as follows: Form the idempotent
$$  e := \frac{1}{\# H_{q}} \sum_{h \in H_q} h - \frac{1}{\# \Aff(q)} \sum_{g \in \Aff(q)} g \in \Q[\Aff(q)]$$
and let $e' = 1-e$ be the complementary idempotent.  Then $e'' := \# \Aff(q) \cdot e' \in \Z[\Aff(q)]$
acts on $\Pic^0(C_1)$,  and moreover the connected component of its kernel is isogeneous to the Prym variety described above:
\begin{equation} \label{alternate def} \mbox{connected component of $ \Pic^0(C_1)[e'']$} \stackrel{\mathrm{isog.}}{\longrightarrow} \left(\mbox{Prym for $C_1' \rightarrow C_2$}\right).\end{equation}

Equation \ref{alternate def} gives a way to access the `reduced'' Prym variety (at least up to isogeny) that we can conveniently apply in our relative situation: 
In the situation described in Definition \ref{KPF},  $Z_q \rightarrow Y'_{q}$ is a relative curve over $Y'_{q}$ and it admits a $\Aff(q)$-action, where $\Aff(q)$
acts trivially on the base.  The relative Picard scheme of this curve is an abelian scheme over $Y'_{q}$
equipped with a symmetric and fiberwise ample line bundle. Thus we may form
 $$ X_q =  \mbox{ relative identity component of } \Pic^0_{Z_q \rightarrow Y'_{q}} [ e''],$$
where $[e'']$ means the kernel of $e''$, and for the notion of ``relative identity component,'' see \cite[Proposition 15.6.4]{EGAIV}.
This $X_q$ is an abelian scheme over $Y'_q$, equipped with a symmetric and fiberwise ample line bundle;
its fiber over any $y \in Y'_q(\bar{K})$ coincides with the construction on the left hand side of \eqref{alternate def}; in particular this fiber is isogenous to
the  reduced Prym variety of the associated $\Aff(q)$-covering  $Z_{y} \rightarrow Y$. 

\begin{dff} \label{KPfam} Notation as in the prior definition. 
The {\em Kodaira--Parshin family of Jacobians} over $Y$,  associated to the group $\Aff(q)$, is the 
sequence of morphisms 
 $$X_{q} \longrightarrow Y'_{q} \rightarrow Y,$$
where $X_q$ is, as defined above,  the
reduced relative Prym of  $Z \rightarrow Y_q' \times Y$, considered as a morphism of relative cuves over $Y_q'$. 
  \end{dff}
This is an abelian-by-finite family, in the sense of Definition \ref{aff family}.

\subsection{Proof of Proposition \ref{hurwitz}} \label{Hurwitzproof}
We give the proof of Proposition \ref{hurwitz}. As we have mentioned this is largely for lack of a good reference which states
precisely what we need; certainly much more general statements about Hurwitz schemes exist in the literature. 

We start by supposing that $Y$ is a proper smooth curve over $\C$;
while we work over $\C$ we identify $Y$ with its complex points.

For $y \in Y$ set \label{Sydefpage}
$S(y)$ to be the set of conjugacy classes of surjective homomorphisms from
$\pi_1(Y-y, *) \twoheadrightarrow G$, with the property that a loop around $y$ has nontrivial image.
Equivalently, $S(y)$ is the finite set of isomorphism classes of connected coverings
of  $Y$ with Galois group $G$, branched  precisely at $y$.  

For $y$ near $y^*$ 
there is a natural identification $S(y) \cong S(y^*)$ since we can topologically identify $(Y, y)$ and $(Y, y^*)$. 
Thus the set $\coprod_{y \in Y(\C)} S(y)$ has the structure of a Riemann surface
$Y'$ equipped with a covering map $e: Y' \rightarrow Y$. 
Explicitly, for each $y' \in Y'$, we have $y' \in S(e(y'))$, or in words:
$y'$ classifies a connected $G$-covering of $Y$  branched at $y = e(y')$. 

Moreover, the coverings indexed by the elements of $S(y)$ fit together to a morphism
$$ f: Z \rightarrow Y' \times Y$$
of smooth complex manifolds; here $G$ acts on $Z$, covering the trivial action on $Y' \times Y$. 
More explicitly:
\begin{itemize}
\item   $f$ is a covering map and a $G$-torsor
when restricted to the complement of the analytic divisor
$$ \Delta := \mbox{graph of $e$} \subset Y' \times Y$$ 
\item the pullback of the above morphism along ${y'} \times Y \hookrightarrow Y' \times Y$
(for $y' \in Y'$) is isomorphic to the covering of $Y$
classified by $y'$.   
\end{itemize}
Near the preimage of $\Delta$ on $Z$
the map looks in local coordinates like $(z,w) \mapsto (z, w^n)$ for suitable $n$. 

Now everything can be algebraized, i.e.\ $Z$ and $Y'$ have unique
structures of complex algebraic variety compatible with their analytic structures, and the 
$G$-action on $Z$ as well as the morphisms $Z \rightarrow Y' \times Y$
and $Y' \rightarrow Y$ are
algebraic.  This is clear for $Y'$; also the 
the structure sheaf of $Z$ defines a coherent analytic sheaf on $Y' \times Y$
which can be made algebraic by  GAGA  (\cite[Theorem 3]{GAGA}); similarly
the algebra structure on this coherent analytic sheaf comes from an algebra structure on the
algebraic sheaf   \cite[Theorem 2]{GAGA}.   

We now switch to using the letters $Z, Y, \dots$ for the complex algebraic varieties, rather than the associated analytic spaces. 
So we have defined a sequence of complex algebraic varieties
\begin{equation} \label{ZZ} Z \stackrel{f}{\longrightarrow} Y'\times Y   \stackrel{e \times \mathrm{id}}{\longrightarrow}   Y \times Y\end{equation}
where $f$ is {\'e}tale away from the graph of $e$, and $e$ is {\'e}tale; the composite
$Z \rightarrow Y^2$ is therefore
  {\'e}tale away from  the diagonal $\Delta$.  (Note that it is equivalent to check {\'e}tale in the algebraic
  and analytic settings, see \cite[XII, \S 3]{SGA1}). 
  
Now suppose that $Y$ is actually defined over a subfield $K \subset \C$;
we denote by $Y_K$ the corresponding $K$-scheme (similarly $(Y^2)_K$, etc.);
we want now to descend everything in sight to $K$.   
 
\begin{lemma} \label{kp_claim}
Write $Z^{\circ}$
for the preimage of $Y^2-{\Delta}$ in $Z$.
\begin{itemize}
\item[(1)]  The  {\'e}tale  cover $F: Z^{\circ} \rightarrow Y^2-\Delta$  
can  be uniquely extended to a cover
 $F_K: Z_K^{\circ} \rightarrow (Y^2-\Delta)_K$. (In both cases, these {\'e}tale covers are understood to be {\em equipped with $G$-action}.)
 \item[(2)] Let $(y_1, y_0) \in Y(\bar{K})^2$, with $y_1 \neq y_0$. The geometric fiber
 $$F_K^{-1}(y_1, y_0)/G$$ is identified with the set  $S(y_0)$, as defined above,
 now using {\'e}tale $\pi_1^{\geom}(Y-y_0, y_1)$. 
 If $(y_1, y_0) \in Y(K)^2$ this identification is equivariant for $G_K$. 
 \item[(3)]  The quotient $Z_K^{\circ}/G$ (which is {\'e}tale over $(Y^2-\Delta)_K$)
extends uniquely to an {\'e}tale cover of $Y_K^2$.  This cover is isomorphic to one of the form
 $Y'_K \times Y_K \rightarrow Y_K^2$ for an {\'e}tale cover $Y'_K \rightarrow Y_K$, 
 such that $Y'$ is the base change of $Y'_K$ to $\C$. 
 \end{itemize}
\end{lemma}

Assume Lemma \ref{kp_claim} (the proof, which will be given in a moment, will involve only the theory of {\'e}tale $\pi_1$ and group theory).   It produces a sequence $Z_K^{\circ} \rightarrow Y'_K \times Y_K \rightarrow Y_K^2$; 
 we need  to extend $Z_K^{\circ}$ to a $K$-structure on all of $Z$, and extend the first map accordingly.

Let $Z_K \rightarrow Y^2_K$ be the normalization of $Y_K^2$ inside the fraction field of 
$Z_K^{\circ}$.  Then $Z_K$ is normal, and finite over $Y_K^2$. 
The base extension $Z_K \otimes_{K} \C$
is therefore also normal (the extension of a normal scheme
along a field extension in characteristic zero is normal -- see \cite[Tag 037Z]{StacksProject} or \cite[Cor. 6.14.2]{EGAIVb}), 
and it is finite over $Y^2$. Consequently, $Z_K \otimes_{K} \C$
coincides with the normalization of $Y^2$ in the function field of $Z^{\circ}$.
This latter normalization is identified with $Z$, for $Z$ is also
normal and finite over $Y^2$. 

The morphism $Z_K^{\circ} \rightarrow  Z_K^{\circ}/G \rightarrow Y'_K \times Y_K$ now extends to $Z_K \rightarrow Y'_K \times Y_K$,
and the other desired properties can be verified since they are true over $\C$. 

\proof[Proof of Lemma \ref{kp_claim}]
We do this by means of  the theory of the {\'e}tale fundamental group. 
We first formulate the basic point in purely group theoretic terms.

Let $\Gamma, G$ be groups, with $G$  finite center-free,
and $c$ a conjugacy class of morphisms  in $\Hom(\Lambda,\Gamma)$ for some other group $\Lambda$; 
when we apply this, $\Gamma$ will be a $\pi_1$ of a punctured curve,
$\Lambda$ will be the profinite completion of an infinite cyclic group, and $c$ will come 
from monodromy around the puncture.  
Consider the set $S = S(\Gamma, c, G)$ of all surjective homomorphisms
$\varphi: \Gamma \rightarrow G$, with the property that they are nontrivial when pulled back by $c$. 
There are natural commuting actions of $\Gamma$ and $G$ on $S$:
$$ \gamma \cdot \varphi =  \varphi \circ \mathrm{Ad}(\gamma)^{-1} \ (\gamma \in \Gamma), \ \ \varphi \cdot h =     \mathrm{Ad}(h^{-1}) \circ \varphi \ \ (h \in G).$$
where we've written $\Ad(x)$ for the automorphism $g \mapsto x g x^{-1}$. 
 
This  $\Gamma$-action  extends  {\em uniquely}    to an action (commuting with $G$) of any overgroup $\widetilde{\Gamma} \supset \Gamma$ in which
$\Gamma$ is normal and whose  conjugation action preserves $c$. Indeed the extension is described by exactly the same formula;
uniqueness comes from the fact 
that the stabilizer of $\varphi \in S(\Gamma, c, G)$ inside $\Gamma \times G^{\mathrm{op}}$ is given by
$$ \{(\gamma \in \Gamma, h \in G): h^{-1} = \varphi(\gamma)\},$$ 
and so $\varphi$ is determined by its stabilizer in $\Gamma \times  G^{\mathrm{op}}$.  (We used that $\varphi$ is surjective and that $G$ is center-free.)
    
We apply this as follows. 
As above, fix two points $y_0 \neq y_1 \in Y(\C)$;
we will use $\mathbf{y} = (y_1, y_0) $ as  a  geometric basepoint for $Y \times Y$. 
Consider the sequence of pointed schemes:
\begin{equation} \label{pss}  \underbrace{ (Y-\{y_0\}, y_1)}_{\Gamma := \pi_1} \stackrel{p \mapsto (p,y_0)}{\longrightarrow}  \underbrace{ (Y^2 - \Delta,  \mathbf{y}) }_{\tilde{\Gamma}^{\geom} := \pi_1}\stackrel{(y, y') \mapsto y'}{\longrightarrow} (Y, y_0)\end{equation}  
and let $\Gamma, \tilde{\Gamma}^{\geom}$ be defined as the geometric {\'e}tale $\pi_1$ of the first and second spaces, at the specified basepoints. 
Now the  long exact sequence  for homotopy groups of a fibration
gives rise to an exact sequence of  {\em  topological} fundamental groups; in the setting at hand this is short exact because the $\pi_2$ of $Y-\{y_0\}$ vanishes.   
The corresponding sequence of 
geometric {\'e}tale fundamental groups is obtained by profinite completion; it  remains exact 
by  the results of \cite{ProfiniteGroups}. Therefore the first map above  identifies $\Gamma$ with a normal subgroup of $\tilde{\Gamma}^{\geom}$.
It follows easily that, if we write 
$$ \tilde{\Gamma} = \pi_1((Y^2-\Delta)_K,\mathbf{y}),$$
(arithmetic fundamental group) then the map $\Gamma \rightarrow \tilde{\Gamma}^{\geom}$
identifies $\Gamma$ to a normal subgroup of $\tilde{\Gamma}$. 

Now let $S=S(\Gamma, c, G)$ be as above,
where $c$ is the conjugacy class of maps $\widehat{\Z} \rightarrow \pi_1(Y-\{y_0\}, y_1)=\Gamma$
arising from the monodromy around $y_0$.  
The commuting $\Gamma \times G$ actions on $S$ define a cover of $Y-\{y_0\}$, equipped
with an action of $G$ by automorphisms, whose fiber at $y_1$ is identified with $S$.
This cover may be described as follows: it is the disjoint union of all the
connected $G$-covers of $Y$ branched precisely   at $y_0$. 
In other words, it is the restriction of $Z \rightarrow  Y^2-\Delta$
to the fiber $\{y_0\} \times (Y-\{y_0\})$. From the uniqueness just described, the 
extension of this $\Gamma \times G$ action on $S$ to an action
of $\widetilde{\Gamma}^{\geom} \times G$ corresponds to 
the cover $Z \rightarrow Y^2-\Delta$. Therefore, the  (further) unique extension of 
the $\Gamma \times G$-action on $S$ to $\widetilde{\Gamma} \times G$
gives the  statement  (1) in the Claim. 

Statement (2) of the  {\em Claim} (and the $G_K$-equivariance if $y_1, y_0$ are $K$-rational) follows for the specific $(y_1, y_0)$ chosen above; however,
since we showed that the $K$-structure on $Z^{\circ}$ is unique, it must also be true for any choice of $(y_1, y_0)$. 

For statement (3) we notice that the action of $\Gamma$ on $S(\Gamma, c, G)/G$ is in fact trivial. Therefore
the resulting action of $\widetilde{\Gamma}$ factors through the quotient  $\pi_1(Y_K, y_0)$ %
arising from the last map of \eqref{pss}. This amounts to the third assertion. 
   \qed

\newcommand{\MCG}{\mathrm{MCG}}
\newcommand{\good}{singly ramified}
\section{The monodromy of Kodaira--Parshin families}
\label{KP_monodromy}
\subsection{Introduction, Notation, Statement of Main Theorem}

In this section we consider surfaces in the classical topological category: by a ``surface''
we mean the complement of finitely many interior points inside a connected, orientable, compact two-dimensional manifold with boundary.
Thus a surface can have both boundary and punctures. 
 Throughout this section, the letters $Y$ and $Z$ will denote such a surface,
and we will use $y_0$ to denote a base point on $Y$. 
For such a surface $Y$, $\MCG(Y)$ denotes the mapping class group of $Y$.
To emphasize, $Y$ could have ``punctures'' or boundary. 
The book of Farb and Margalit \cite{FM} is a reference on this material that contains all the results we will use.
When we discuss homology or cohomology, the coefficients are always assumed to be the rational numbers $\Q$ unless stated otherwise.

We first reformulate the statement to be proven.

\subsection{Covers and their homology}
 
Let $Y$ be a surface (possibly with punctures or boundary).   An $\Aff(q)$-cover of $Y$ is, by definition, 
a connected surface $Z$ together with a degree $q$ covering map
$$ \pi: Z \longrightarrow Y$$
whose monodromy representation  
on a general fiber is equivalent to the action of $\Aff(q)$ on $\mathbf{F}_q$
(i.e.\ we can label points in the fiber by $\mathbf{F}_q$ in such a way that the monodromy
representation has image $\Aff(q)$). 
We will often abuse notation and refer to this cover simply as $Z$, i.e., regard the map $\pi$ as implicit. 

After choice of basepoint $y_0 \in Y$, such a cover determines 
an $\Aff(q)$-conjugacy class\footnote{A priori, the map is defined up to conjugation by the normalizer of $\Aff(q)$ in $\Sym(\mathbf{F}_q)$.
This normalizer is equal to $\Aff(q)$.}
of maps
\begin{equation} \label{pi1rep} \pi_1(Y, y_0) \twoheadrightarrow \Aff(q).\end{equation}
We define two $\Gq$-covers $(Z_1, \pi_1)$ and $(Z_2, \pi_2)$ to be {\em  isomorphic}
when there is a homeomorphism $Z_1 \simeq Z_2$ commuting with the projections to $Y$;
equivalently, when the associated conjugacy classes of $\pi_1$-representations \eqref{pi1rep}
coincide.

If we have fixed a $\Gq$-cover $Z \rightarrow Y$, we denote by $\monmap
: \pi_1 \rightarrow \Aff(q)$ any homomorphism
in the conjugacy class of \eqref{pi1rep}. For $\eta \in \pi_1$ we can unambiguously talk about the  \emph{cycle decomposition} of $\monmap(\eta)$ in $\Sym(\mathbf{F}_q)$, which we regard as a partition of the positive integer $q$; this cycle decomposition is conjugation-invariant.  
 
Given any covering map $\pi: Z \rightarrow Y$,
the pullback and pushforward on homology define a splitting
$$H_1(Z, \Q) = \pi^* H_1(Y, \Q) \oplus \underbrace{ H_1^{\mathrm{Pr}}(Z, Y; \Q)}_{\mathrm{ker } \left ( \pi_*: H_1(Z) \rightarrow H_1(Y) \right ).}$$
Henceforth we will drop the coefficients $\Q$ from the notation.  The symbol $\mathrm{Pr}$ stands for primitive; alternatively, $\HPr(Z, Y)$ is the homology of a Prym variety.

Now $H_1(Z, \Q)$ and $H_1(Y, \Q)$ are both equipped with skew-symmetric pairings, the intersection pairings.
 The map $\pi^*$ scales the pairing by the degree $q$ of the covering of $Z \rightarrow Y$.  
 If the pairing on $H_1(Z, \Q)$ is nondegenerate, we may identify the primitive homology with
 the orthogonal complement to  $\pi^* H_1(Y, \Q)$
in $H_1(Z, \Q)$, and in particular this primitive homology inherits a skew-symmetric pairing. 
In our case, $Z$ and $Y$ will both be compact surfaces punctured at a single point, 
 and therefore the intersection pairings on $H_1(Z, \Q)$ and $H_1(Y, \Q)$   are perfect.

\subsubsection{The mapping class group and its action on homology; the map $\monsp$}

Clearly the diffeomorphism group of $Y$ acts on the  finite set of isomorphism classes of $\Gq$-covers of $Y$,
and this action factors through the mapping class group $\MCG(Y)$. 
In algebraic terms,  this action is induced from the map $\MCG(Y) \longrightarrow \mathrm{Out}(\pi_1(Y, y_0))$. 
 
Let $\MCG(Y)_Z$ denote the stabilizer of $(Z, \pi)$ for this action. Since $\Aff(q)$ has trivial centralizer in $\Sym(\mathbf{F}_q)$, 
such elements lift uniquely to mapping classes on $Z$, i.e.\ there is a homomorphism
$$\MCG(Y)_Z \longrightarrow \MCG(Z).$$
Namely, fixing a representative $\alpha: Y \rightarrow Y$,
there is a unique $f: Z \rightarrow Z$ that renders the diagram
     \begin{equation}   \label{wah}
 \xymatrix{
 Z  \ar[d]^{\pi}\ar[r]^f & Z \ar[d]^{\pi} \\
 Y  \ar[r]^{\alpha} & Y
 }
 \end{equation}
commutative.  Sending the mapping class of $\alpha$ to the mapping class of $f$ defines the desired homeomorphism.

This construction gives rise to actions of $\MCG(Y)_Z$ on $H_1(Z)$ and $\HPr(Z, Y)$.  This latter action is the \emph{monodromy map}
$$\monsp :  \MCG(Y)_Z \rightarrow \Sp ( \HPr ( Z, Y )).$$

\subsubsection{The main theorem} \label{mtsetup}

Fix a surface $Y$ of genus $g \geqslant 2$, a point $y \in Y$, a prime $q \geqslant 3$; as before, $\Gq$ denotes the group of affine-linear transformations of $\mathbf{F}_q$.

We consider $\Gq$-covers $Z^{\circ}$ of $Y - \{y\}$ such that the monodromy around $y$ is nontrivial (hence a $q$-cycle); the compactification of such a cover is a surface $Z$ of genus $gq  - \frac{q-1}{2}$.  We call such $Z$ \emph{\good} $\Gq$-covers of $Y$.  The notation hides the dependence on the point $y$, which will remain fixed.   There are, up to isomorphism, only finitely many such $Z$; choose a representative for each isomorphism class and call them $Z_1, Z_2, \ldots, Z_{\numcovers}$, and let $\monmap_1, \monmap_2, \ldots, \monmap_{\numcovers}: \pi_1(Y - \{y\}) \rightarrow \Gq$ be 
representatives for the associated monodromy mappings. 
  
 Let $\MCG(Y - \{y\})_0$ denote the intersection of the groups $\MCG(Y - \{y\})_{Z_i}$.  The individual monodromy maps attached to the covers $Z_i$ combine to give a map
\begin{equation} \label{KPmon} \monsp: \MCG(Y - \{y\})_0 \rightarrow \prod_{i=1}^{\numcovers} \Sp( \HPr(Z_i, Y) ).\end{equation}
 The mapping class group of a punctured surface fits in the Birman exact sequence  \cite[Theorem 4.6]{FM} %
\begin{equation} \label{birman} 0 \rightarrow \pi_1(Y, y) \rightarrow \MCG(Y - \{y\}) \rightarrow \MCG(Y) \rightarrow 0. \end{equation}
Let $\pi_1(Y, y)_0$ denote the inverse image of $\MCG(Y - \{y\})_0$ in $\pi_1(Y, y)$; the inclusion $\pi_1(Y, y)_0 \subseteq \pi_1(Y, y)$ is of finite index.

The restriction of \eqref{KPmon} to the subgroup $\pi_1(Y, y)_0$
describes the monodromy of a Kodaira-Parshin family, as in Definition \ref{KPfam}.
We review this connection in more detail in \S \ref{sssTMT}. 
The following statement is equivalent to the large monodromy property of Kodaira--Parshin families,
stated without proof as point (i) before Theorem \ref{main theorem}.

 \begin{thm} \label{monodromy theorem}
Let notation be as above; in particular, $Z_1, \dots, Z_{\numcovers}$ are a set of representatives
for isomorphism classes of singly ramified $\Gq$-covers of $Y$. Then the map
\begin{equation} \label{TBD} \monsp: \pi_1(Y, y)_0 \rightarrow \prod_{i=1}^{\numcovers} \Sp( \HPr(Z_i, Y) )\end{equation}
has Zariski-dense image.
\end{thm}

We briefly outline the proof.
We give in \S \ref{GStd} a ``normal form'' for each $\Gq$-cover.
 Using the sequence \eqref{birman},
we reduce to showing a similar assertion with $\pi_1(Y, y)_0$ replaced by $\MCG(Y-\{y\})_0$.
This allows us to use Dehn twists. Using our normal form for $\Gq$-covers,
and constructing a suitable system of curves to Dehn-twist around, 
we can see that the monodromy surjects onto each factor $\Sp(\HPr(Z_i, Y))$. 
A version of Goursat's lemma completes the proof.

When considering the general problem (replacing $\Aff(q)$ or cyclic covers by $G$-covers) the primitive homology must
be further decomposed according to the representation theory of $G$. 
Looijenga \cite{Looijenga} has proven a similar result for cyclic covers of surfaces without monodromy;  in fact, Looijenga determines the exact image of $\monsp$ in this situation. 
See also \cite[Theorem 1.6]{GLLM} for an analogous result for unramified covers of a closed surface,
and \cite{ST} for covers whose covering group is the Heisenberg group.

\subsubsection{Application to Theorem \ref{main theorem}} \label{sssTMT}  For clarity we now write out why Theorem \ref{monodromy theorem}, in the form stated, above, implies 
what is used in Theorem \ref{main theorem}: namely, the Kodaira-Parshin family $X_q \rightarrow Y_q' \stackrel{\pi}{\rightarrow} Y$ for the group $\Aff(q), q \geq 3$, 
has full monodromy.  

In stages:
\begin{itemize}
\item We begin, as in \S  \ref{Hurwitz1} with a family
$Z_q \rightarrow Y_q' \rightarrow Y_q$, with $Z_q \rightarrow Y_q'$ a relative curve.
(The Kodaira-Parshin family was constructed by applying a Prym construction to this, as explicated
in Definition \ref{KPfam}). 

 \item  Fix $y \in Y_q(\C)$.  The fiber of $Y_q' \rightarrow Y_q$  above $y \in Y_q(\C)$ is identified with
 the isomorphism classes of singly ramified $\Aff(q)$-covers of $Y(\C)$, branched at $y$.
  This follows from  property (i) of Proposition \ref{hurwitz}.

\item   Fix $y' \in Y_q'(\C)$ above $y \in Y_q(\C)$.    %
Let $Z_{q,y'}$ be the fiber of $Z_q \rightarrow Y_q'$ over $y'$.

By construction,
$Z_{q,y'} \rightarrow Y(\C)$
is a singly branched $\Aff(q)$-cover,
and we can form the degree $q$ cover associated to the action of $\Aff(q)$ on $\Z/q\Z$,
i.e. $Z_{q,y'}(\C) \times_{\Aff(q)} \Z/q\Z$. 
By our definitions above, we have
$$Z_{q,y'}(\C) \times_{\Aff(q)} \Z/q\Z \simeq Z_i$$
(for some unique $i$ in $\{1, 2, \dots, N\}$) 
as Riemann surfaces over $Y(\C)$. 
 
\item 
The construction of Kodaira--Parshin 
families then gives rise, as in
\eqref{alternate def}, to an isogeny
$$\mbox{fiber  $X_{q,y'}$ of $X_q$ above $y'$} \longrightarrow  \mathrm{Prym}(Z_i \rightarrow Y(\C))   $$  
 This isogeny induces an  isomorphism of the rational homology groups:
$$H_1^{\Pr}(Z_i, Y; \Q) \simeq \mbox{first homology of $X_{q,y'}(\C)$ with rational coefficients.}$$
\item 
This identification is compatible with monodromy, and so 
Theorem \ref{monodromy theorem} translates to definition \eqref{fullmonodromy}
of full monodromy. 
\end{itemize}

\subsection{Dehn twists and liftable curves} 
\label{dehn_homology}

We say that $e$ is a \emph{simple closed curve} in a surface $Y$ if it is the image of a smooth embedding $S^1 \rightarrow Y$; a simple closed curve has no self-intersection.  For us, a simple closed curve will always come with an orientation, namely, the orientation induced from a fixed orientation on $S^1$.  If $y \in Y$ is a point, then we say $\eta \in \pi_1(Y, y)$ is \emph{represented by a simple closed curve} if there is a loop $e$ in $Y$, based at $y$ and representing the class $\eta \in \pi_1(Y, y)$, which is a simple closed curve.  We may say (somewhat imprecisely) that $\eta$ ``is'' a simple closed curve.

If $e$ is a simple closed curve in $Y - \{y\}$, the Dehn twist $D_e$ about $e$ acts on $H_1(Y)$ by the transvection $T_e$;
indeed, we can regard $D_e$ as an element of $\MCG(Y - \{y\})$. 
We want
to study how this lifts to an $\Gq$-cover   $Z \rightarrow Y$:    Let $n_e$ be the order of the image of $e$ in $\Gq$.  Then $D_e^{n_e}$ lifts to an automorphism of the cover $Z$, as we now describe.  Suppose the image of $e$ under $\pi_1(Y, y_0) \rightarrow \Gq \rightarrow \mathrm{Sym}(\mathbf{F}_q)$ has cycle structure $(d_1, \ldots, d_k)$. The preimage of $e$ under $Z \rightarrow Y$ is a disjoint union of circles $e_1, \ldots, e_k$, with the circles $e_i$ in natural bijection with the cycles in the permutation.  Then $D_e^{n_e}$ lifts to the product of commuting Dehn twists
$$\prod_i D_{e_i}^{n_e / d_i}$$
on $Z$.

In our cases, the only possibilities for cycle structure are as follows:
\begin{itemize}
\item[-]
If $e$ maps to an element of $\Gq$ that is not in $\F_q^+$, i.e.\ has nontrivial image $a \in \F_q^*$,
then $(d_1, \dots, d_k) =   (1, \mathrm{ord}_q(a), \mathrm{ord}_q(a), \ldots, \mathrm{ord}_q(a))$. 
\item[-]
If $e$ maps to a nonzero element of $\F_q^+$ then $(d_1, \dots, d_k) = (q)$. 
\item[-]
If $e$ maps to the identity element of $\Gq$, then $(d_1, \dots, d_k) = (1, \dots, 1)$. 
\end{itemize}

Now we note that:
\begin{lemma}
Let $e$ be a simple closed curve in $Y - \{y\}$.
Then the classes of the preimages $[e_1], \dots, [e_k]$ in the homology of $Z$ are linearly independent; projected to $\HPr(Z, Y)$, their span has dimension $k - 1$. 
\end{lemma}
\proof
$Y$ admits the structure of a CW complex with one $2$-cell such that $e$ belongs to the $1$-skeleton.
 The inclusion of this $1$-skeleton into $Y-y$ is a homotopy equivalence.
Correspondingly the inclusion of the preimage  (in $Z$) of this $1$-skeleton into the preimage (in $Z$) of $Y-y$ is
also a homotopy equivalence.  Note also that the inclusion
of $Y-y$ into $Y$ induces an isomorphism on $H_1$, with a similar
statement for $Z - \pi^{-1}(y)$. 

These remarks allow us to reduce the Lemma to corresponding assertions for a covering of a finite graph,  which are clear.
\qed

The action of $\monsp(D_e^{n_e})$ on $\HPr(Z, Y)$ is a unipotent transformation $u$ such that the image of $u-1$ is exactly the span of the classes of the circles $e_i$.  
By the Lemma just proven, this has dimension $k-1$; correspondingly  the fixed space $\monsp(D_e^{n_e})$ on $\HPr(Z, Y)$  has codimension $k-1$.

We record the following consequence:

\begin{lemma}
\label{dehn_index}
Suppose $Z \rightarrow Y$ is an $\Gq$-cover.  Let $e$ be a simple closed curve in $Y$, and take $\bigexp$ such that $D_e^{\bigexp} \in \MCG(Y)_Z$.  
Then the rank of $\monsp(D_e^{\bigexp})-\mathrm{Id}$ acting on $\HPr(Z,Y)$
determines the conjugacy class of $\monmap(e)$ in the symmetric group $\Sym(\mathbf{F}_q)$. 
\end{lemma}

A particularly important case is when $e$ is a simple closed curve in $Y$ such that $\monmap(e)$ maps to a generator for $\mathbf{F}_q^*$ under the natural map $\Gq \rightarrow \mathbf{F}_q^*$. 
We call such a $e$ a {\em liftable curve}
(for the $\Gq$-cover $Z \rightarrow Y$).  Its preimage in $Z$ splits into a union of simple closed curves $e^+$
of degree $1$ over $e$, and $e^-$ of degree $q-1$ over $e$. 
For liftable $e$ we write
$$ \widetilde{e} := \mbox{ projection of the class of $e^+$ to primitive homology.} $$
According to our discussion above, 
$D_e$ induces a transvection on $\HPr(Z, Y)$, with center $\widetilde{e}$. 

Write $\cdot$ for the intersection pairing on homology. Given liftable curves $A$ and $B$, we have
\begin{equation} \label{phf} \widetilde{A} \cdot \widetilde{B}  =   (A^+ \cdot B^+) - \frac{1}{q} A \cdot B.\end{equation}
Indeed, identifying primitive homology with the kernel of the pushforward,
we have $q\widetilde{A} = q A^+ - \pi^* A$, and so the intersection pairing of $q \widetilde{A}$ with $q \widetilde{B}$ is
$$ (q A^+ - \pi^* A)  \cdot(q B^+ - \pi^* B) = q^2 (A^+ \cdot B^+) -  2q (A \cdot B) + q (A \cdot B),$$
as desired.

\subsection{A normal form for an $\Aff_q$-cover} \label{GStd}

Again, take $Z$ a \good\ $\Gq$-cover of $Y$. We will describe the cover $Z \rightarrow Y$ in a normal form by cutting $Y$ carefully,
using  essentially the fact that $\Aff(q)$ is solvable. 
The end result is roughly that the covering $Z\rightarrow Y$ can be expressed as the sum of a trivial cover of a genus $g-1$ surface and a nontrivial cover on a torus.

Choose a basepoint $y_0 \in Y$.  The map $\monmap: \pi_1(Y - \{y\}, y_0) \rightarrow \Gq$ specifying the cover $Z \rightarrow Y$ induces a map 
$$H_1(Y, \mathbf{Z})  \cong H_1(Y - \{y\}, \mathbf{Z}) \rightarrow \mathbf{F}_q^*$$
on abelianizations.  The group $ \mathbf{F}_q^*$ is cyclic and $H_1(Y, \mathbf{Z})$ is free. %
If we choose a surjection $\mathbf{Z} \twoheadrightarrow \mathbf{F}_q^*$, then the map on abelianizations lifts to a map
$$H_1(Y, \mathbf{Z}) \rightarrow \mathbf{Z} \rightarrow \mathbf{F}_q^*.$$
We can choose this map so that $H_1(Y, \mathbf{Z}) \rightarrow \mathbf{Z}$ is surjective, so  
it is given by intersecting with a primitive integral homology class $\alpha_1$.  

Choose a simple closed curve, which we also call $\alpha_1$, representing this class.
(Indeed, any primitive integral homology class is represented by a simple closed curve: \cite[Proposition 6.2]{FM}.)
In fact, choose two such curves, $\alpha_1^+$ and $\alpha_1^-$, which
pass ``close by'' but on either side of the ramification point $y$, and are parallel to one another. 
Note that, since our cover is ramified at $y$,  and the monodromy at $y$ is a nontrivial element of in $\mathbf{F}_q^+$,   
$\monmap(\alpha_1^+)$ and $\monmap(\alpha_1^-)$ cannot both be trivial. 

\begin{figure}[ht]
\label{alpha1alpha2}
\centering
\includegraphics[height=6cm]{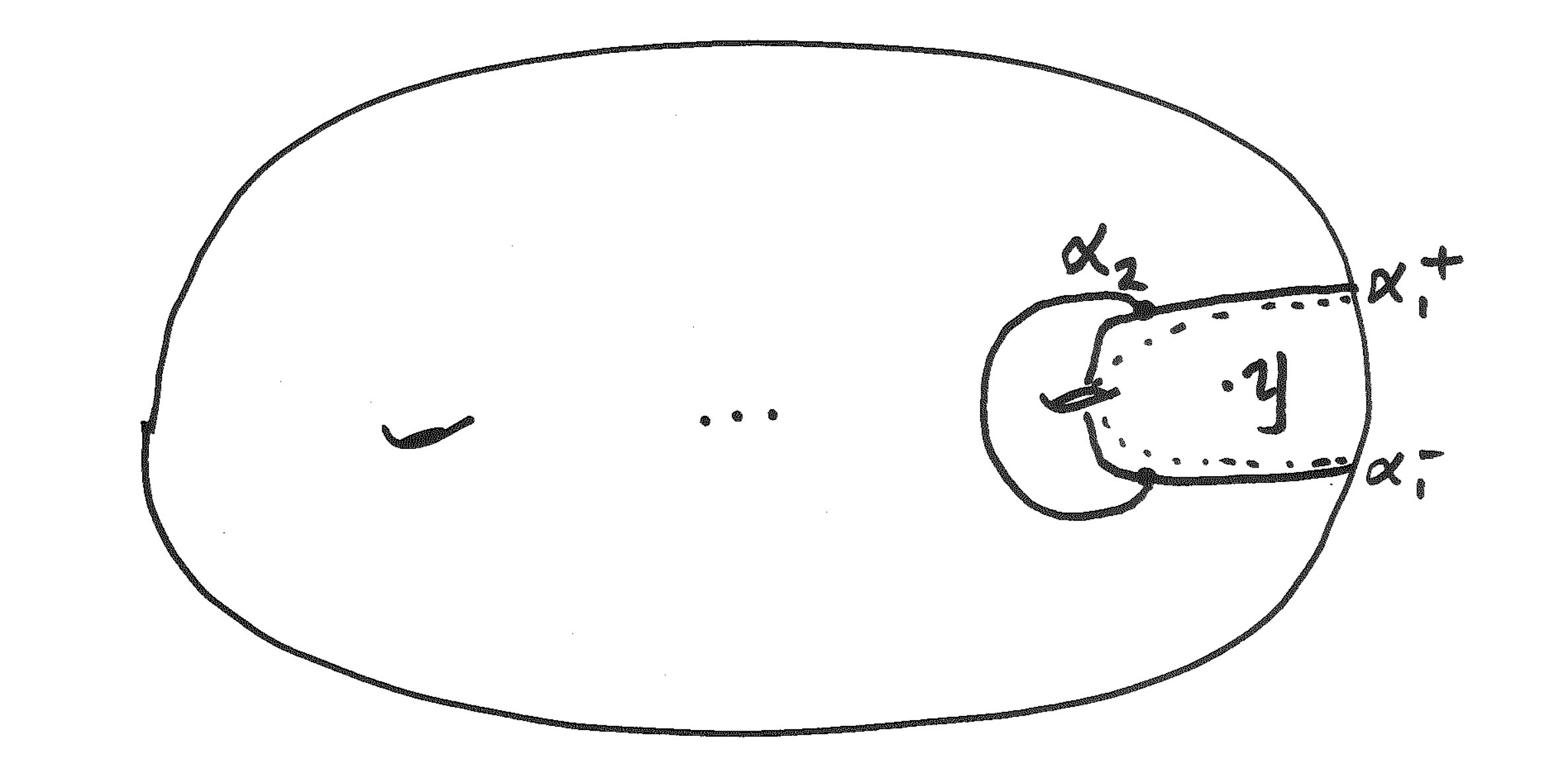}
\caption{The curves $\alpha_1^{\pm}$ and $\alpha_2$ on $Y$.}
\end{figure}

Cutting $Y$ along the curves $\alpha_1^{\pm}$ and discarding the connected component of $y$, we obtain a surface $Y^1$ with two boundary components.
Let $Z^1$ be the pullback of our covering to $Y^1$.  The map $\monmap: \pi_1(Y^1, y_0) \rightarrow \Gq$ has image contained in $\mathbf{F}_q^+ \subseteq \Gq$ by our choice of $\alpha$; so it factors through $H_1(Y^1, \mathbf{Z})$.

The boundary components (with orientations defined by an outward normal)
define classes $b_+, b_- \in H_1(Y^1,\Z)$; these classes satisfy $b_+ + b_- = 0$ because 
their sum is the boundary of $Y^1$.  We saw above that $b_+$ and $b_-$ cannot both have trivial image in $\mathbf{F}_q^+$;
so $\monmap(b_+) = - \monmap(b_-)$ must be nontrivial.
Conjugating by a suitable element of $\Gq$ as necessary, we may as well suppose that $b_+ \in H_1(Y^1,\Z)$ maps to $1 \in \mathbf{F}_q^+$.

For a surface such as $Y^1$ with boundary $\partial Y^1$,  Poincar{\'e} duality
takes the form of a perfect pairing between absolute and relative homology:
\begin{equation} \label{PD} H_1(Y^1, \partial Y^1; \Z) \times H_1(Y^1; \Z) \rightarrow \Z\end{equation}
The map $H_1(Y^1,\Z) \rightarrow \mathbf{F}_q^+$ lifts to a map $H_1(Y^1,\Z) \rightarrow \mathbf{Z}$; since $b_+$ 
is a primitive element of $H_1(Y^1,\Z)$, we can choose such a lift taking $b_+$ to $1$.  
This lift is of the form $x \mapsto \langle x, \alpha_2 \rangle$ for a relative homology class $\alpha_2 \in H_1(Y^1, \partial Y^1; \Z)$. 
Therefore $\alpha_2$ intersects the boundary components with multiplicity $+1$ and $-1$. 

The following lemma readily implies that $\alpha_2$ can be represented by a simple curve,
the image of an immersion $e: [0,1] \rightarrow Y^1$ that meets $\partial Y^1$ only at the endpoints, 
which we also call $\alpha_2$.
Indeed, it implies $MCG(Y)$ acts transitively on that subset of $H_1(Y, \partial Y) \simeq \Hom(H_1(Y,\Z), \Z)$ 
consisting of elements whose pairing with a fixed boundary circle is $1$.  It follows that we can find a mapping
class carrying  the homology class of $e$ to the homology class $\alpha_2$, as desired.

\begin{lemma}
Suppose $Y$ is a surface of genus $g$ with $2$ boundary components, so $V = H_1(Y,  \mathbf{Z})$ is a free $\mathbf{Z}$-module of rank $2g + 1$.
We regard it as equipped with a (degenerate) alternating form via $H_1(Y) \rightarrow H_1(Y, \partial Y)$ and the duality pairing \eqref{PD};  the radical  of  this form is the rank-$1$ submodule $V^0$ generated by $b$, the class of one of the two boundary components of $Y$. 
 
   Let $\Sp(V, b)$ denote the group of automorphisms of $V$ preserving the bilinear form and fixing $b$.  Then the natural map $MCG(Y) \rightarrow \Sp(V, b)$ is surjective.
\end{lemma}

\begin{proof}

The group $\Sp(V, b)$ fits into an exact sequence
$$1 \rightarrow \Hom(V/V^0, V^0)  \rightarrow \Sp(V, b) \rightarrow \Sp(V/V^0) \rightarrow 1,$$
where the left-hand map is given by $f \mapsto 1+f$.  

Now one obtains a closed surface from $Y$ by capping off both boundary components.
The mapping class group $\MCG(Y)$ surjects onto the mapping class group of this closure \cite[Prop 3.19]{FM}.
Therefore (by the surjectivity of the symplectic representation for a closed surface \cite[Theorem 6.4]{FM}) it surjects onto $\Sp(V/V^0)$.

Now let $v \in V$ be a class, not in $V^0$, which is represented by a simple closed curve in $Y$; and let $b$ be one of the two boundary components of $Y$.  We can represent $v + b$ by a simple closed curve as well (possibly after replacing $b$ with $-b$).   Thus the image of $\MCG(Y)$ contains the transvections $T_v$ and $T_{v+b}$.  The composition $T_{v+b} T_v^{-1}$ is a nontrivial element of $\Sp(V, b)$, coming from the element
$$x \mapsto \langle x, v \rangle b \in \Hom(V/V^0, V^0).$$
These generate $\Hom(V/V^0, V^0)$ so the result follows. 
\end{proof}

Cut $Y^1$ along $\alpha_2$, and let $Y^2$ be the resulting surface; it is a surface of genus $g-1$ with one boundary component. 
The pullback of the cover $Z \rightarrow Y$ to $Y^2$ {\em splits}, i.e., becomes a disjoint union of $q$ copies of $Y^2$.
We can recover $Y$ from $Y^2$ by gluing 
to $Y^2$ a torus with one boundary component. 
Thus our discussion has shown that
  it is possible to put any $\Gq$-cover $Z \rightarrow Y$ into a normal form:
a connected sum of a trivial cover of a genus $g-1$ surface and a nontrivial cover of a torus. 

\begin{figure}[ht]
\centering
\includegraphics[height=6cm]{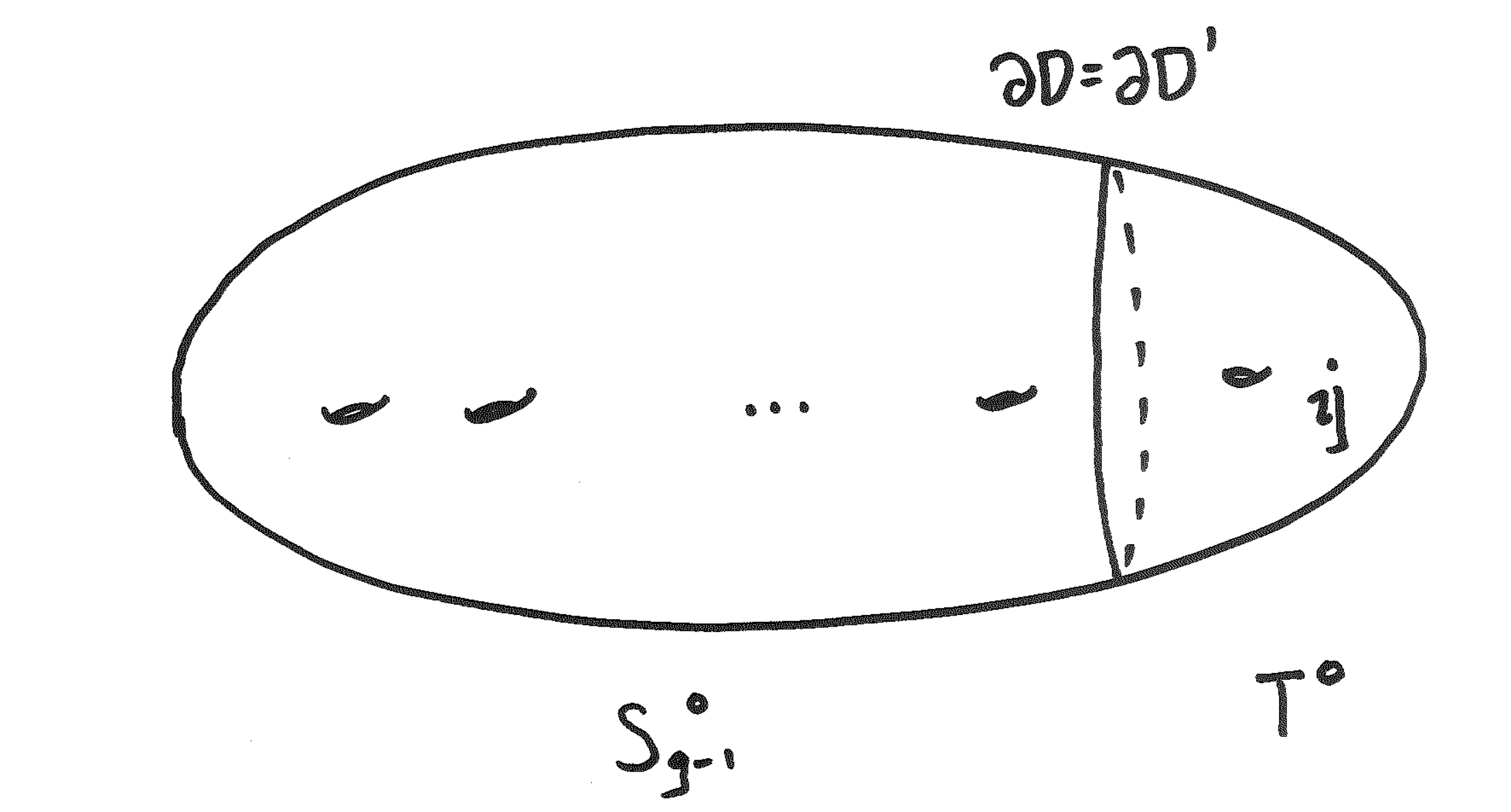}
\caption{$Y$ as a connected sum.}
\end{figure}

To summarize: Let $\mathsf{S}_{g-1}$ be a genus-$(g-1)$ surface
and let $\mathsf{T}$ be a torus.
Fix a small open disk $D$ in $\mathsf{S}_{g-1}$ and $D'$ in $\mathsf{T}$,
and set
\begin{equation} \mathsf{S}_{g-1}^{\circ} = \mathsf{S}_{g-1} - D,\ \ \label{TSdef} \mathsf{T}^{\circ} = \mathsf{T} - D', \end{equation}
so these are, respectively, a surface of genus $g-1$ with one boundary component and a torus with one boundary component. 
We identify $Y$ with the genus-$g$ surface obtained by gluing  
$\mathsf{S}_{g-1}^{\circ}$ to $\mathsf{T}^{\circ}$ along an identification $\partial D' \simeq \partial D$.
(In relation to  the discussion just given, 
 $\mathsf{S}_{g-1}^{\circ}$ is homotopy-equivalent to $Y^2$. )

\begin{prop}[Normal form for $\Gq$-covers]
Let $Z$ be a singly ramified $\Gq$-cover of $Y$.
Then we may write $Y$ as a connected sum: %
$$Y = \mathsf{S}_{g-1} \# \mathsf{T},$$
where $\mathsf{S}_{g-1}$ is a genus-$(g-1)$ surface and $\mathsf{T}$ is a genus-$1$ surface, satisfying the following properties (with notation as above).

\begin{itemize}
\item The ramification point $y$ belongs to the interior of $\mathsf{T}^{\circ}$,
\item the cover $Z \rightarrow Y$ splits over $\mathsf{S}_{g-1}^{\circ}$. 
\item the cover $Z \rightarrow Y$, when restricted to $\mathsf{T}^{\circ}$, extends over $\mathsf{T}$,
i.e.\ has trivial monodromy around the boundary circle of $\mathsf{T}^{\circ}$. 
\item With respect to a standard basis for $\pi_1(\mathsf{T}-y,*)$, a free group
on two generators $\beta_1, \beta_2$, the monodromy of 
the cover sends
\begin{itemize}
\item $\beta_1$ to an element of $\Gq$ projecting to a generator for $\F_q^*$, and
\item $\beta_2$ to an nonzero element of $\F_q^+$. 
\end{itemize}
\end{itemize}
\end{prop}

Here $\beta_1$ is a curve which crosses $\alpha_1$ once and does not cross $\alpha_2$;
and similarly for $\beta_2$.   

\begin{figure}[ht]
\centering
\includegraphics[height=6cm]{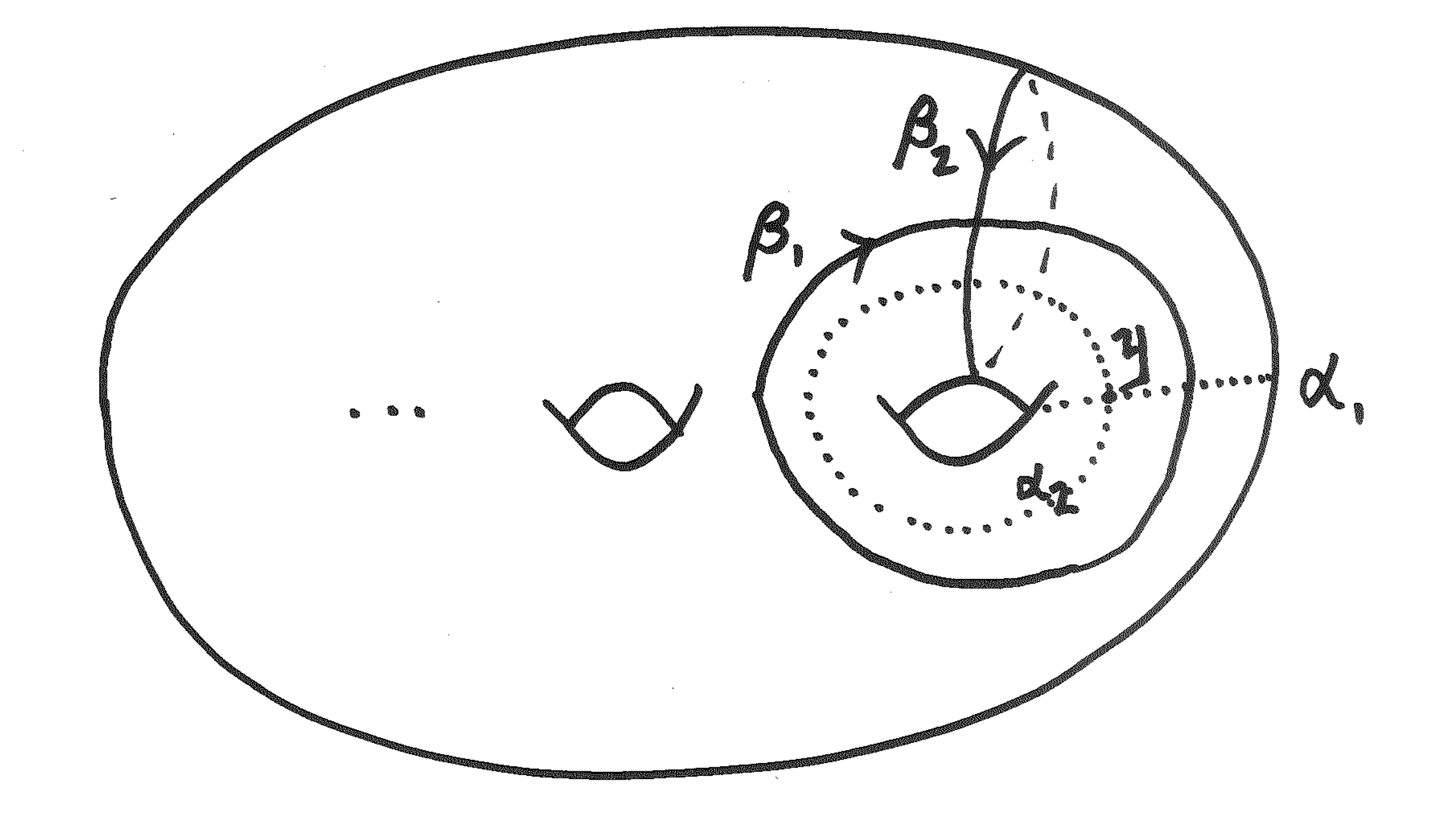}
\caption{The curves $\beta_1$ and $\beta_2$. (The basepoint is the intersection of $\beta_1$ and $\beta_2$.)}
\end{figure}

Thus $Z$ consists of $q$ copies of $\mathsf{S}_{g-1}^{\circ}$
glued to a degree-$q$ cover $\widetilde{\mathsf{T}^{\circ}}$ of $\mathsf{T}^{\circ}$
along $q$ boundary circles. In the sequel we will use
$\widetilde{\mathsf{S}_{g-1}^{\circ}}$ for the cover of $\mathsf{S}_{g-1}^{\circ}$
induced by $Z$.

\subsection{Proof of Theorem \ref{monodromy theorem}}

We must show \eqref{TBD} has Zariski-dense image. 
We will perform a series of reductions;
the main steps are Lemmas \ref{one_star}, \ref{two_stars}, and \ref{three_stars}.

\begin{lemma}
The image of $\pi_1(Y, y)_0$ (see after \eqref{birman} for definition) under the monodromy map 
$$ \pi_1(Y, y)_0 \rightarrow \Sp( \HPr(Z_i, Y) )$$
to any factor $\Sp(\HPr(Z_i, Y))$ of the right-hand side of \eqref{TBD} is not contained in the center of $\Sp$.
\end{lemma}

\begin{proof}
We leave the simple topological proof to the reader.\footnote{One can also give an algebro-geometric argument, as follows.
Suppose to the contrary. Now, as in \S \ref{rationalpoints}, there is an associated 
finite covering $Y' \rightarrow Y$ such that the various $Z_i$ fit together into a curve fibration $\mathsf{Z} \rightarrow Y'$. 
If (a) were false, the theorem of the fixed part means that
the Hodge structure of the fibers of $\mathsf{Z} \rightarrow Y'$ are constant, at least over one component of $Y'$. 
 By Torelli, this means that all the fibers are actually isomorphic. This contradicts de Franchis's theorem.}
\end{proof}

Because $\pi_1(Y, y)_0$ is normal inside $\MCG(Y-\{y\})_0$
and the symplectic groups are almost simple,
Theorem \ref{monodromy theorem} follows from the subsequent Lemma:

\begin{lemma} \label{one_star}
The monodromy map restricted to $\MCG(Y-\{y\})_0$,
$$\monsp: \MCG(Y-\{y\})_0 \rightarrow \prod_{i=1}^{\numcovers} \Sp( \HPr(Z_i, Y) ),$$
has Zariski-dense image.
\end{lemma}

In turn, using Lemma \ref{goursat}, this will follow from Lemmas \ref{item_b} and \ref{two_stars}.

\begin{lemma}[Distinct covers are distinguished by monodromy around a simple closed curve.] \label{item_b}
For two non-isomorphic $\Gq$-covers $Z_1, Z_2$
there exists a simple closed curve  $\eta$ in $Y$
such that the cycle decompositions of the monodromy around $\eta$ 
in $Z_1$ and $Z_2$ are different.
\end{lemma}

\proof  
Two coverings $Z_1, Z_2$ define two maps $\pi_1(Y-y) \rightarrow \Gq$.
Suppose, first of all, that their projections to $\mathbf{F}_q^*$
have different kernels (i.e.\ are not related by an automorphism of $\mathbf{F}_q^*$).
We may find a primitive homology class whose images under the two maps $f_1, f_2: H_1(Y, \Z) \rightarrow \mathbf{F}_q^*$
have different orders in $\mathbf{F}_q^*$. 
Indeed, there is a basis  $e_1, \dots, e_r$ for $H_1(Y, \Z)$ 
such that the kernel of $f_1$ equals $(q-1)e_1, e_2, \dots, e_r$; 
not all of $e_2, \dots, e_r$ can be in the kernel of $f_2$, and 
so at least one of these latter classes suffice.
Represent this primitive homology class
by a simple closed curve to construct $\eta$. 

Otherwise, the coverings $Z_1, Z_2$ define maps $\pi_1(Y-y) \rightarrow \F_q^*$ having the same kernel.
Accordingly, in the algorithm to convert an $\Gq$-cover into a normal form described
in \S \ref{GStd}, we can cut $Y$ along the {\em same} curve $\alpha_1$, as in \S \ref{GStd}, 
for both $Z_1$ and $Z_2$. We obtain, as before, a surface $Y^1$ with two boundary components; the covers $Z_1, Z_2$
define two maps
$$g_1, g_2: H_1(Y^1,\Z) \longrightarrow \F_q^+.$$
If $g_1$ is not proportional to $g_2$, we can find a primitive homology class for $H_1(Y^1,\Z)$
which is in the kernel of one map but not the other. Represent this primitive homology class by a simple closed curve to construct $\eta$. 

Otherwise $g_1$ and $g_2$ are proportional, so the two maps $\pi_1(Y^1) \rightarrow \F_q^+$ have the same kernel.
Therefore, we can cut $Y_1$ along the same curve $\alpha_2$ for both $Z_1$ and $Z_2$.
So we get a decomposition of $Y$ as a connected sum $Y = \mathsf{S}_{g-1} \# \mathsf{T}$
as above, such that both $Z_1$ and $Z_2$ become trivial on $\mathsf{S}_{g-1}$.

Let $\beta_1$ and $\beta_2$ be curves on $\mathsf{T}$ as in the end of \S \ref{GStd}.
Then both maps $\pi_1(Y - y) \rightarrow \Gq$ send $\beta_1$ to an element of $\Gq$
projecting to a generator for $\F_q^*$; and they both send $\beta_2$ to an element of $\F_q^+$.
Each of $\monmap_1 (\beta_1)$ and $\monmap_2 (\beta_1)$ has a unique fixed point in $\F_q$;
up to conjugation, we may suppose this fixed point is $0$.  By a further conjugation we may assume that
$\monmap_1 (\beta_2) = \monmap_2 (\beta_2) = 1 \in \F_q^+$.

So we can write 
$$\monmap_1(\beta_1) \colon x \mapsto c_1 x$$
and
$$\monmap_2(\beta_1) \colon x \mapsto c_2 x.$$
If $Z_1$ and $Z_2$ are not isomorphic covers, we must have $c_1 \neq c_2$. 

There is a map
$$ \pi_1(Y-y) \longrightarrow \pi_1(\mathsf{T}-y)$$
which is obtained (in the notation of \eqref{TSdef}) 
by collapsing $\mathsf{S}_{g-1}^{\circ}$ to a point; this gives a map from $Y-y$ 
to a surface that is homotopy equivalent to $\mathsf{T}-y$.
  
There exists a simple closed curve $\eta \in \pi_1(Y - y)$  mapping to $\beta_1 \beta_2 \beta_1^{-1} \beta_2^{q-c_1}$ 
under this map $\pi_1(Y - y) \rightarrow \pi_1(\mathsf{T} - y)$: 
see Figure \ref{proof_by_picture} and its caption.

Then $\monmap_1(\eta)$ is trivial
but $\monmap_2(\eta)$ is not trivial.  This concludes the proof. \qed

\begin{figure}[ht]
\label{proof_by_picture}
\centering
\includegraphics[height=6cm]{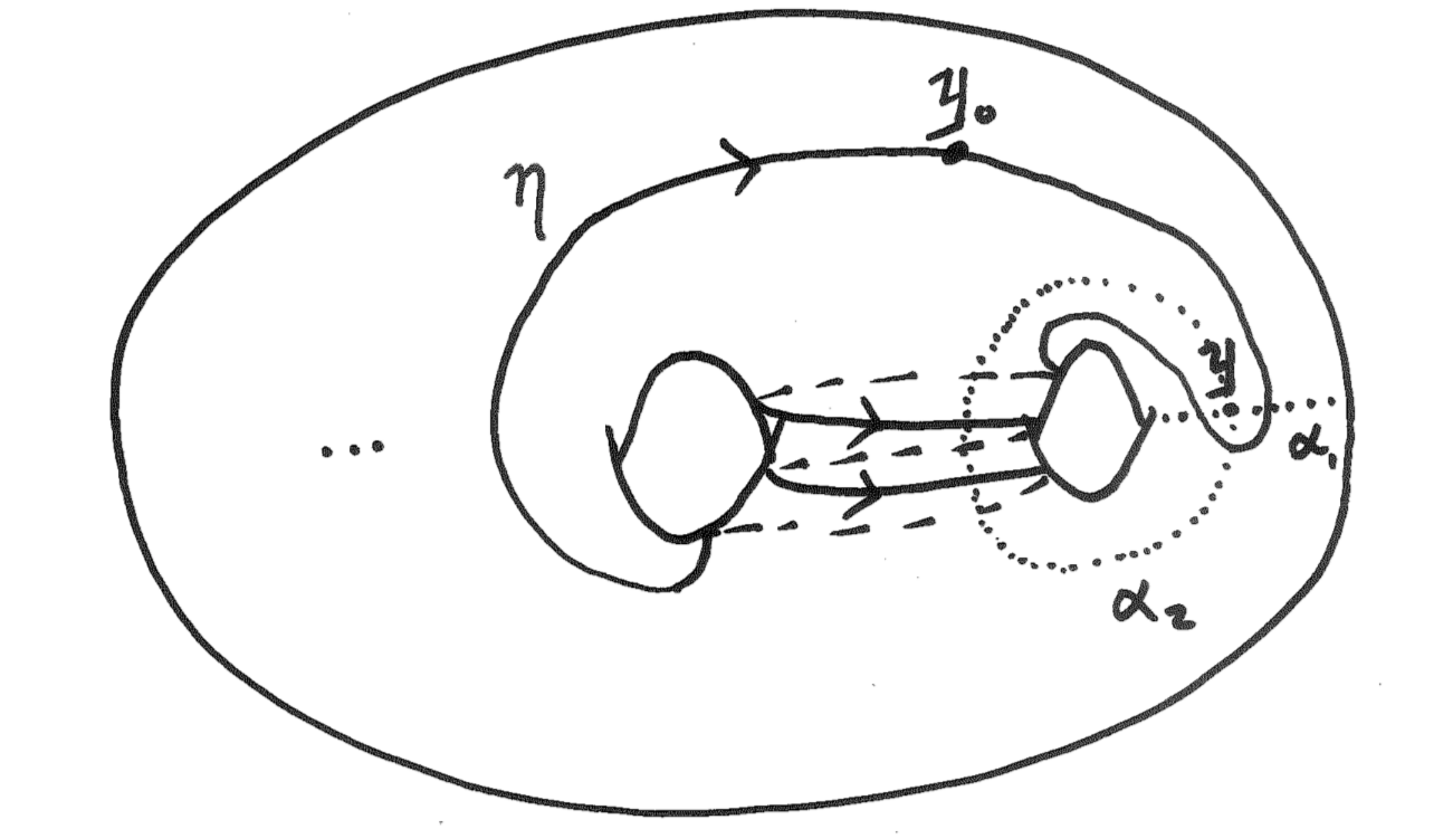}
\caption{The curve $\eta$.  \newline
How to read the picture: Follow along the path $\eta$, starting at the basepoint $y_0$.  
Write down a word in the symbols $\beta_1$ and $\beta_2$ as follows. 
Every time $\eta$ crosses $\alpha_1$, write $\beta_1$ or $\beta_1^{-1}$, depending whether the crossing
was in the positive or negative direction.
Every time $\eta$ crosses $\alpha_2$, write $\beta_2$ or $\beta_2^{-1}$.
The resulting word is the image of $\eta$ under the map $\pi_1(Y - y, y_0) \rightarrow \pi_1(\mathsf{T} - y, y_0)$,
which we readily see is $\beta_1 \beta_2 \beta_1^{-1} \beta_2^{2}$.}
\end{figure}

\begin{lemma} \label{two_stars}
The monodromy map
$\mathrm{Mon}: \MCG(Y)_{Z_i} \longrightarrow \Sp(\HPr(Z_i, Y))$
has Zariski-dense image.
\end{lemma}

We are now reduced to proving Lemma \ref{two_stars}. Let $Z=Z_i$ for some fixed $i$. 
By the construction of Dehn twists from liftable curves (see discussion at end of  
\S \ref{dehn_homology}), as well as
Lemma \ref{transvection_graph} on generation by transvections, 
it is enough to show:

\begin{lemma} \label{three_stars}
There exists 
a collection of liftable curves $A_1, \dots, A_N$  on $Y$ such that:
\begin{itemize}
\item[(a)] the $\widetilde{A}_i$ span the primitive homology $\HPr(Z, Y)$;
\item[(b)] the graph   obtained
by connecting $A_i, A_j$ when $\widetilde{A_i} \cdot \widetilde{A_j} \neq 0$ is connected.
\end{itemize}
\end{lemma}

 \subsection{Proof of Lemma \ref{three_stars}} 
 We put the singly ramified $\Gq$-cover $Z \rightarrow Y$ 
 in a normal form, as explained in 
\S \ref{GStd}.
 Recall notation ($D$, $D'$, $\mathsf{S}_{g-1}$, $\mathsf{T}$, and so forth) from the end of \S \ref{GStd}.
 We will produce the curves $A_i$ by concatenating curves on $\mathsf{T}^{\circ}$ and curves on $\mathsf{S}_{g-1}^{\circ}$. 
 
Fix a point $p \in \partial D \cong \partial D'$. 
Fix a labelling of the points of $Z$ above $p$ by $\F_q$,
compatible with the usual action of $\Gq$
for some fixed homomorphism
$$\monmap: \pi_1(Y-y, p) \longrightarrow \Gq.$$ 

Recall that the cover $Z \rightarrow Y$ splits
over $\mathsf{S}_{g-1}^{\circ}$; the labelling above $p$ therefore permits us also to label
the components of $\mathsf{S}_{g-1}^{\circ}$ by $\F_q^+$. 

\begin{lemma} \label{scctoruslemma}
There exist $q+1$ simple closed curves $\{ \gamma_j: 0 \leqslant j \leqslant q \}$ on $\mathsf{T}^{\circ}$, beginning and ending at $p$, 
not passing through $y$, and intersecting $\partial D'$ only at its endpoints,
such that: 
\begin{itemize}
\item[(i)] For each $j$, the monodromy $\mathrm{Cov}(\gamma_j)$ projects
under $\Gq \rightarrow \F_q^*$ to the same fixed generator of $\F_q^*$; 
 \item[(ii)] The monodromy of $\gamma_j$, defining a map $\F_q \rightarrow \F_q$, fixes exactly $j$ modulo $q$. 
\item[(iii)] The (unique) lifts $\gamma_j^+$ to simple closed curves on $\widetilde{\mathsf{T}^{\circ}}$ span the homology of $\widetilde{\mathsf{T}^{\circ}}$  
{\em modulo the homology of its boundary.}
\item[(iv)] Each $\gamma_j$ has the same orientation near $p$, i.e., either the outgoing branch is ``above'' the incoming branch for all $j$,
or vice versa.
\end{itemize}
\end{lemma}

\medskip

\begin{figure}[ht] \label{gammaw}
\centering
\includegraphics[height=6cm]{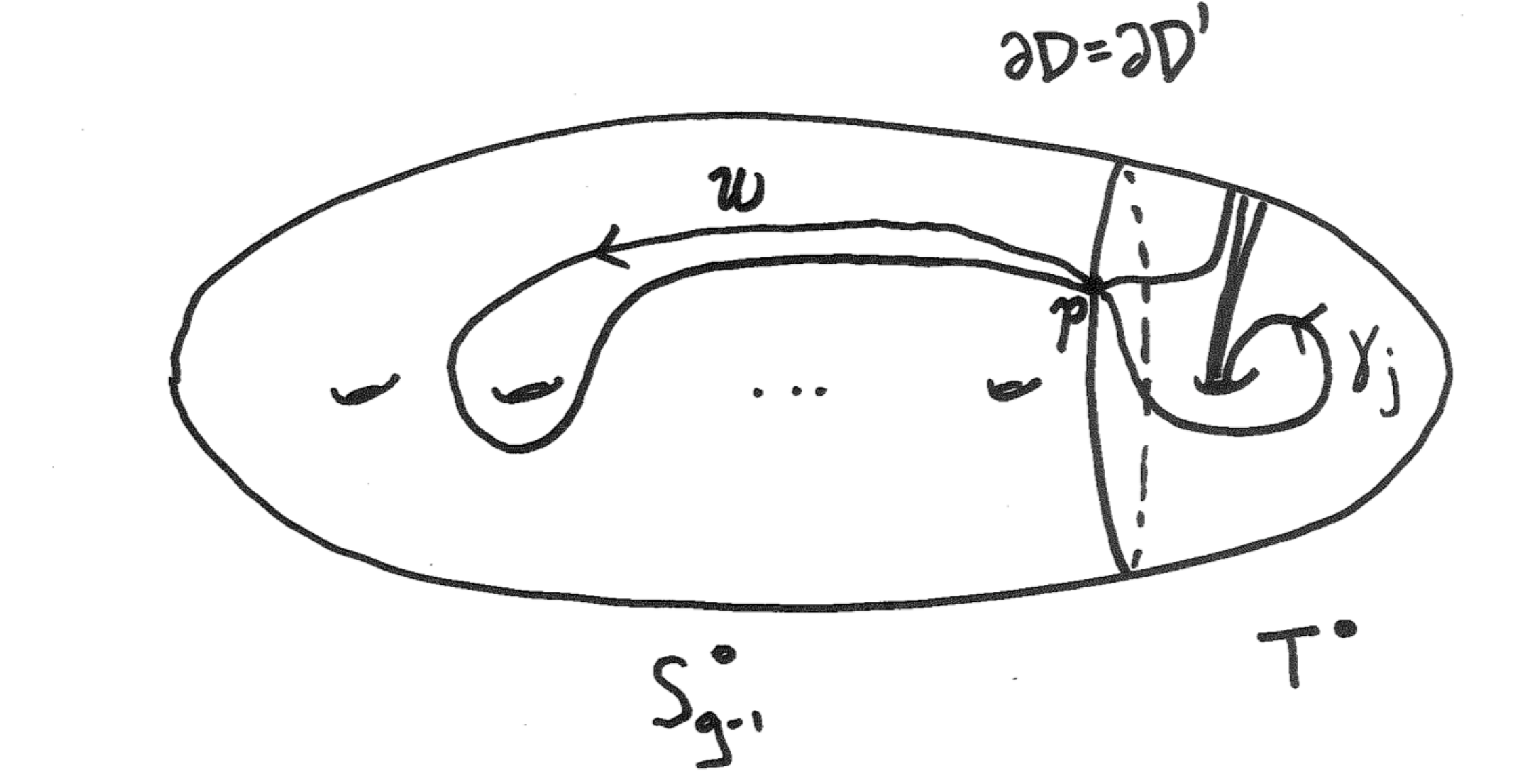}
\caption{The curves $\gamma_j$ and $w$ on $Y$.}
\end{figure}

\begin{proof}
Take an explicit basis $\beta_1, \beta_2$ of homology of $T$, such as was described
in \S \ref{GStd}; conjugating if necessary we can suppose that the monodromy of $\beta_1$
is $x \mapsto g x$ (for $g \in \F_q^*$  a generator) and the monodromy of $\beta_2$ is $x \mapsto x+1$.
 We can choose this basis in such a way that all powers $\beta_1 \beta_2^j$ with $j$ non-negative are represented by simple closed curves on $\mathsf{T}^{\circ}$, which start and end at $p$.  

 The monodromy around $\beta_1 \beta_2^j$ is given by $x \mapsto g(x+j)$, which fixes $  \frac{gj}{1-g} \in \F_q^+$.
 Write $[\ell]$ for the unique representative of $\ell \in \F_q$ that lies in $[0,q-1]$, and put
 $$ j^* =  \begin{cases}  [\frac{gj}{1-g}],    &  j  \neq q  \\  q, & j=q. \end{cases}$$
The map $j \mapsto j^*$ gives a bijection from $[0,q]$ to itself.
Now put $$\gamma_{j^*} = \beta_1 \beta_2^j \ \ (j \in [0,q]). $$
 
Conditions (i) and (ii) are clearly satisfied.  To check (iii)
we must verify that the associated homology classes span homology of $\widetilde{\mathsf{T}^{\circ}}$ modulo its boundary.
One could verify this by computing an explicit CW-complex for $\widetilde{\mathsf{T}^{\circ}}$; we present an alternative
group-theoretic proof.  
It is sufficient to show that
\begin{equation} \label{redu} \mbox{ the homology clases of the lifts of $\gamma_j$  span $H_1(\widetilde{\mathsf{T}-y})$.}\end{equation}
Here and in what follows we make use of the fact that our $\Gq$-cover extends over $\mathsf{T}$,
and thus write (e.g.) $\widetilde{\mathsf{T}}$. 
To see that \eqref{redu} indeed implies (iii),   consider the diagram
$$
\xymatrix{
H_1(\widetilde{\mathsf{T}^{\circ}-y})  \ar[d] \ar[r] & H_1(\widetilde{\mathsf{T}^{\circ}})/H_1(\partial \widetilde{\mathsf{T}^{\circ}}) \ar[d]^{\sim}  %
\\ 
H_1(\widetilde{\mathsf{T}-y}) \ar[r]^f & H_1(\widetilde{\mathsf{T}})/H_1(\widetilde{\mathsf{D}}) = H_1(\widetilde{\mathsf{T}}) %
}$$
where $f$ is surjective because in fact $H_1(\widetilde{\mathsf{T}-y}) \simeq H_1(\widetilde{\mathsf{T}})$:
the preimage of $y$ is a single point. 

Let $\tilde{p}$ be the point above $p$ corresponding to $0 \in \mathbf{F}_q$. 
Projection to $T$ identifies $\pi_1(\widetilde{\mathsf{T}-y}, \tilde{p})$
 with the subgroup
$\mathsf{H} \leqslant \langle \beta_1, \beta_2 \rangle$ defined by
$$ \mathsf{H} = \mbox{stabilizer of $0 \in \F_q$}.$$
Therefore the 
first homology of $\widetilde{\mathsf{T}-y}$  is the abelianization of $\mathsf{H}$. 
Under this correspondence, the homology class of the lift of $\gamma_j$ 
corresponds to the image in $\mathsf{H}^{\mathrm{ab}}$ of $(\beta_2^{-j^*}) \beta_1 \beta_2^j (\beta_2^{j^*}) \in \mathsf{H}$.

 We must therefore show that the elements
$ \beta_2^{-j^*} \beta_1 \beta_2^{j+j^*}$  actually generate $\mathsf{H}^{\mathrm{ab}}$.
 Note that among these elements are $\beta_1$ and (a conjugate of) $\beta_1 \beta_2^q$,
 so it is enough to show that
\begin{equation} \label{SCS} \beta_2^q \mbox{ and } \beta_2^{-j^*} \beta_1 \beta_2^{j+j^*} \  \ (0 \leqslant j \leqslant q-1)\end{equation}
 generate $\mathsf{H}^{\mathrm{ab}}$.  
However, a  set of left coset representatives for $\mathsf{H}$ are given by
$1, \beta_2, \dots, \beta_2^{q-1}$; according to Schreier's algorithm 
a generating set for $\mathsf{H}$ is given by
$$ \beta_2^q,  \medskip \beta_2^{-[g j]} \beta_1 \beta_2^j, j \in [0,q-1] $$
Now considered modulo $q$  the set of pairs $(-[gj], j) \equiv (-gj, j)$ appearing here coincide with the pairs
 $(-j^*, j+j^*) \equiv (-\frac{gj}{1-g}, \frac{j}{1-g})$ appearing in \eqref{SCS}. So the elements of \eqref{SCS} even generate $\mathsf{H}$
 not just its abelianization.
 \end{proof}

We return to the proof of Lemma \ref{three_stars}.
For each primitive homology class in $\mathsf{S}_{g-1}^{\circ}$ we fix a representative
which is a simple closed curve on $\mathsf{S}^{\circ}_{g-1}$  beginning and ending at $p$. 
Let $W$ be the resulting collection of simple closed curves.  
 For each $w \in W$ at least one of the two homotopy classes
\begin{equation} \label{signchoice} A(w,j) = \gamma_j \cdot w^{\pm 1} \in \pi_1(Y, p)\end{equation}
is representable by a simple closed curve on $Y$.   
The choice of sign depends only on $w$ and does not depend on $j$, in view of property (iv) of the curves $\gamma_j$. For a picture of the curve $A(w,j)$,
see Figure \ref{gammaw}.

The image of this curve in $\Gq$ projects to a generator of $\F_q^*$;
therefore it is ``liftable'' in the sense of \S \ref{dehn_homology}.  Recall also from
\S \ref{dehn_homology} the notation $e^+$ for the degree-$1$ lift of a liftable curve $e$.
 The lift of $A(w, j)$ has homology class given by
$$[A(w,j)^+] = [ \gamma_j^+] \pm [w_j],$$
where $w_j$ means that we lift  $w$ to a closed loop on the $j$th preimage
of $\mathsf{S}^{\circ}_{g-1}$ inside $Z$; the sign above is the same as in \eqref{signchoice}.

We have  
$$ [A(w,j)^+] -  [A(w', j)^+] =  \epsilon [w_j] + \epsilon' [(w')_j] \ (\epsilon, \epsilon' \in \pm 1)$$
and we see readily that these classes -- as both $w, w'$ vary through $W$ -- 
span the homology of the $j$th  preimage
of $\mathsf{S}_{g-1}^{\circ}$ in the cover $Z$. 

 The boundary of $\widetilde{\mathsf{S}_{g-1}^{\circ}}$
is a union of $q$ circles. 
Considering the Mayer--Vietoris sequence 
$$H_1(S^1)^q \rightarrow H_1(\widetilde{\mathsf{S}_{g-1}^{\circ}}) \oplus  H_1(\widetilde{\mathsf{T}^{\circ}}) \twoheadrightarrow H_1(Z).$$
and using the fact that the $[\gamma_j]$ span $H_1(\widetilde{\mathsf{T}^{\circ}})$ modulo its boundary (Lemma \ref{scctoruslemma} (iii)),
we see that the $[A(w,j)^+]$ span $H_1(Z)$. 

As before, we define
$$ \widetilde{A(w,j)} = \mbox{projection of } [A(w,j)^+] \mbox{ to primitive homology.}$$
so that the homology classes $\widetilde{A(w,j)}$
span $\HPr(Z, Y)$. This completes the proof of part (a) of Lemma \ref{three_stars}.

To prove part (b), we need to compute some intersection numbers.  We note that
 the intersection number between any $\gamma_j^+$ and any $w_k$ is trivial. Thus 
$$ [A(w_1,j)] \cdot [A(w_2,k)] = [\gamma_j ] \cdot [\gamma_k] \pm [w_1] \cdot [w_2]$$
$$ [ A(w_1, j)^+] \cdot  [A(w_2,k)^+] =   [\gamma_j^+ \cdot \gamma_k^+]
\pm \delta_{jk} [w_1] \cdot [w_2],
$$
where $\delta_{jk}$ is the Kronecker $\delta$ symbol, and, in both instances,
the sign that appears ithe product of the sign for $w_1$ and the sign for $w_2$. 
Upon projecting to primitive homology, \eqref{phf} gives
$$  \widetilde{A(w_1, j)} \cdot \widetilde{A(w_2, k)}=   \pm (  \delta_{jk}  -q^{-1})[w_1] \cdot [w_2] 
+ (   [\gamma_j^+] \cdot [\gamma_k^+] -  q^{-1} [\gamma_j \cdot \gamma_k]). $$

The connectedness of the ``intersection graph''  follows from this.  It is enough to show that
given $(w_1, j)$ and $(w_2, \ell)$ there exists $(w_3, k)$ with both intersection numbers nonzero. 
For this, we note that the factor $( \delta_{jk} - q^{-1})$ is never zero, so we simply choose $w_3$ so that $[w_1] \cdot [w_3]$ and $[w_2] \cdot [w_3]$ 
are sufficiently large: this is possible because $[w_1], [w_2] \neq 0$, the intersection pairing on $H_1(S_{g-1}^{\circ})$ is perfect, and we can choose $w_3$ 
such that $[w_3]$ is any given primitive homology class. 
This proves Lemma \ref{three_stars} and Theorem \ref{monodromy theorem}. \qed

\newcommand{\pmax}{p_{\max}}
 \section{Transcendence of period mappings; the Bakker--Tsimerman theorem} \label{hypersurface}
 
   It is desirable to extend the method to settings where the base $Y$ is higher-dimensional,
 thus feasibly leading to finiteness results for integral points on $Y$. 
  We will study the example when   $X \rightarrow Y$ 
is  the moduli space of smooth hypersurfaces in $\P^m$;
 then integral points on $Y$ correspond to integral homogeneous polynomials
 $ P(x_0, \dots, x_{m})$ of degree $d$ whose discriminant $(\mathrm{disc} \ P) \in \mathcal{O}^*$.

 (A natural family of generalizations of this example is given by considering
 the integral points on $\mathbf{P}^m - Z^{\vee}$, where $Z \subset \mathbf{P}^m$
 is a smooth subvariety, and $Z^{\vee}$ is the dual projective variety to $Z$: 
there is a natural smooth projective family over $\mathbf{P}^m-Z^{\vee}$,  namely, the 
 family of smooth hyperplane sections of $Z$.)

\subsection{The Ax-Schanuel theorem of Bakker and Tsimerman}

Suppose that we are given a smooth proper map $X \rightarrow Y$ of relative dimension $d$ over the complex numbers
(we identify complex algebraic varieties with their complex points). 
The primitive cohomology of each fiber  $H^d(X_y, \C)^{\mathrm{prim}}$ carries a polarized Hodge structure.   
Let $\mathfrak{H}$ be the associated period domain which classifies
polarized Hodge structures with the same numerical data 
as this primitive cohomology, 
so we have an analytic period map 
$$ \Phi : \widetilde{Y} \longrightarrow \mathfrak{H}$$
where $\widetilde{Y}$ is the universal cover of $Y(\C)$.  This $\mathfrak{H}$
is open (for the analytic topology) in a certain complex flag variety $\mathfrak{H}^*$, which
parameterizes isotropic flags with a given dimensional data inside a certain orthogonal or symplectic complex vector space.

Bakker and Tsimerman \cite{BT} have proven the following analogue of the Ax--Schanuel theorem.
It is a very strong statement about the transcendence of $\Phi$.

To simplify the statement, we assume that the monodromy
mapping 
$$\pi_1(Y) \rightarrow \mathrm{Aut}(H^d(X_y, \C)^{\prim})$$
has image whose Zariski closure contains the full special orthogonal
or symplectic group, stabilizing the intersection form.  (This restriction,
which guarantees that the image $\Phi(\widetilde{Y})$ is Zariski-dense in $\mathfrak{H}^*$, 
is not important, and in \cite{BT} the theorem is formulated
for an arbitrary Mumford-Tate domain as target.)
  
\begin{thm} \label{BTtheorem} (Theorem of Bakker and Tsimerman.)
 Suppose that $V \subset  Y \times \mathfrak{H}^*$
is algebraic. Write $W$ for the image of $\tilde{Y}$ in $Y \times \mathfrak{H}$. 
Suppose that $U \subset V \cap W$ is irreducible analytic such that
$$ \codim_{Y \times \mathfrak{H}^*}(U) < \codim_{Y \times \mathfrak{H}^*}(V) + \codim_{Y \times \mathfrak{H}^*}(W),$$
where all the codimensions are taken inside $Y \times \mathfrak{H}^*$.
Then the projection of $U$ to $Y$ is contained in a proper (``weak Mumford-Tate'') subvariety.
\end{thm}

In particular this has the following corollary:

\begin{cor}[Transcendence property of period mappings] 
\label{trans_prop}

With notation as above, suppose that  $Z \subset \mathfrak{H}^*$ is an algebraic subvariety, and
\begin{equation} \label{dim cond} \mathrm{codim}_{\mathfrak{H}^*}(Z)  \geqslant  \dim(Y) .\end{equation}

Then any irreducible component of $\Phi^{-1}(Z)$ is contained inside the preimage, in $\widetilde{Y}$, of  the complex points of a proper subvariety of $Y$.
\end{cor}
 
\proof
Let $Q$ be an irreducible component as in the statement of the corollary. 

Let $V = Y \times Z$.
The intersection $W^Z$ of $W$ with $Y \times (Z \cap \mathfrak{H})$,
intersction taken in $Y \times \mathfrak{H}$, 
is  an analytic set. 
Moreover, the image of $Q$ under the analytic map
$\tilde{Y} \rightarrow Y \times \mathfrak{H}$
is contained in $W^Z$. Therefore, the image of $Q$ is contained
in some irreducible component of $W^Z$, call it $U$:
$$U = \mbox{an irreducible component of $W \cap \left( Y \times (Z \cap \mathfrak{H})\right)$.}$$

We apply Theorem \ref{BTtheorem} with this chioce of $U,V,W$. 
Then
$$\codim_{Y \times \mathfrak{H}^*} V = \codim_{\mathfrak{H}^*} Z
\mbox{ and }\codim_{Y \times \mathfrak{H}^*} W = \dim \mathfrak{H}^*.$$
so 
$$ \codim_{Y \times \mathfrak{H}^*} W + \codim_{Y \times \mathfrak{H}^*} V  = \dim \mathfrak{H}^* +\codim_{\mathfrak{H}^*} Z  \stackrel{\eqref{dim cond}}{\geqslant} \dim \mathfrak{H}^* + \dim(Y).$$
 This shows $\dim U  = 0$, unless the projection of $U$ to $Y$ is contained in a proper weak Mumford--Tate subvariety. This implies the same property for $Q$, as desired. \qed

 \subsection{Transferring transcendence to a $p$-adic setting} \label{padic conj}

Theorem \ref{BTtheorem} can be transferred to the $p$-adic setting, which is where we use it:
 
With notation as above, suppose additionally 
that $X \rightarrow Y$ is defined over $\Z[S^{-1}]$. 
Fix $p \notin S$ and $y_0 \in Y(\Z_p)$.
As before we can form the $p$-adic period map
\begin{equation} \label{phipdef} \Phi_p:  \underbrace{\mbox{ residue disk around $y_0$ in $Y(\Q_p)$}}_{U_p} \longrightarrow \mathfrak{H}^*_{\Q_p},\end{equation}
where $\mathfrak{H}^*_{\Q_p}$ is the base-change of $\mathfrak{H}^*$ to $\Q_p$;
the map above is $p$-adic analytic, i.e., it is given in suitable coordinate charts
by power series absolutely convergent on the residue disk.

Now suppose that
we give ourselves a $\Q_p$-algebraic subvariety $Z \subset \mathfrak{H}^*_{\Q_p}$
satisfying the dimensional condition \eqref{dim cond}, i.e.\ the codimension of $Z$
is greater than or equal to the dimension of $Y$.

\begin{lemma} \label{padic BT}
Let $U_p$ be as in \eqref{phipdef}, i.e.\ $\{y \in Y(\Z_p): y \equiv y_0 \mbox{ modulo $p$}\}$. 
The set  $$\Phi_p^{-1}(Z)$$
is not Zariski dense in $Y$.  \end{lemma} 

Note that $\Phi_p$ is defined only on the residue disk $U_p$.

\proof

  It will be convenient to have the freedom to vary $y_0$ later in the argument. To that end, 
 note that the statement above depends only on $U_p$;
after all, at the level of points, $\Phi_p$ is the map
 sending $y \in U_p$ to the induced Hodge filtration on the primitive
 crystalline cohomology of the special fiber of $X_y$.

By \cite[Thm.\ 7.6]{Mazur_estimates} (or by the discussion
of \S \ref{GMsubsec})
 the image of $\Phi_p(U_p)$ is contained in a residue disk on $\mathfrak{H}_p^*$ containing $\Phi_p(y_0)$, in particular, in some affine open set $\Spec A_p$ of $\mathfrak{H}_p^*$ containing $\Phi_p(y_0)$.
We may  suppose that $Z \subset \mathfrak{H}_{\Q_p}^*$ is defined locally by equations
$F_i = 0$, where we suppose $F_i \in A_p$, i.e.\ the $F_i$ are regular functions on this affine open set.

Consider now $$G_i = F_i \circ \Phi_p.$$
These are defined by power series converging absolutely in $U_p$, i.e.
in a suitable choice of local coordinates,  $G_i$ lies in a Tate algebra
$$ R= \Q_p \left \langle \frac{x_1}{p}, \dots, \frac{x_N}{p} \right \rangle$$ of formal power series convergent on a disk of $p$-adic radius $|p|$.  In these coordinates $U_p$ corresponds to $(x_1, \dots, x_N) \in (p \Z_p)^N$. 
We want to show that the common zero-locus, inside $U_p$,  of the $G_i$ is contained in 
(the $\Q_p$-points of) an algebraic set.
As a preliminary reduction, we will reduce to considering a single ``irreducible
component'' of this common zero locus.

Fix a suitable open affine set $\Spec \ B_p \subset Y_{\Q_p}$
``containing the residue disc of $y_0$.'' (More precisely, we may fix an open 
 affine neighbourhood in $Y$, considered now as $\Z_p$-scheme, of the image of the $\Z_p$-valued point $y_0: \Spec\ \Z_p \rightarrow Y$,
  and take its generic fiber.)
Then there is a morphism from $B_p$ to $R$. 
Our result  will now follow from the
\begin{quote}
{\em Claim:} Let $\mathfrak{p}$ be a prime ideal of $R$, vanishing
at some point of $U_p$. Suppose that
$\mathfrak{p}$ is minimal among prime ideals
containing 
$\langle G_1, \dots, G_n \rangle$. 
Then $\mathfrak{p}$ contains (the image in $R$ of) a regular function $H$,
i.e.\ a function $H$ belonging to $B_p$ as above. 
\end{quote}

To see why this implies the statement of the lemma, assume the {\em Claim.} There are only 
finitely many such minimal primes as in the statement.
Call them $\mathfrak{p}_1, \dots, \mathfrak{p}_t$.
Let $H_j \in \mathfrak{p}_j$ the function constructed according to the claim above.
Then the vanishing locus of $\prod_j H_j$
contains the common vanishing locus of the $G_i$: if $y$ lies in this common zero-locus,
it lies in the vanishing locus of some $\mathfrak{p}_j$, and then $H_j(y) = 0$. 

We now prove the {\em Claim}. The ideal $\mathfrak{p}$ vanishes at some point at $U_p$ by assumption;
choose such a point $y_0$.  

We now transfer the question to the complex numbers. 
We fix an isomorphism $\sigma: \overline{\Q_p} \simeq \C$,
which gives in particular an
 embedding $\sigma: \Q_p \hookrightarrow \C$.
Then $y_0$ gives rise to a complex point $y_0^{\sigma} \in Y(\C)$,
and the de Rham cohomology of $X_{y_0}^{\sigma}$
is obtained from that of $X_{y_0}$ via $\sigma$:
\begin{equation} \label{pciso} H^*_{\dR}(X_{y_0}) \otimes_{(\Q_p, \sigma)} \C = H^*_{\dR}(X_{y_0^{\sigma}}/\C).\end{equation}

We may regard the period map $\Phi_p$
as taking values in the Grassmannian $\mathfrak{H}^*_{\Q_p}$
for the left-hand de Rham cohomology. 
Also let $U_{\C}$
be a small complex neighbourhood of $y_0^{\sigma}$ and 
let $\Phi_{\C}: U_{\C} \rightarrow \mathfrak{H}_{\C}^*$
be the complex period mapping, which 
we regard as taking values in the associated complex variety $\mathfrak{H}^*_{\C} := \left(\mathfrak{H}^*_{\Q_p}\right)^{\sigma}$.
This complex variety parameterizes certain flags inside the right-hand space of \eqref{pciso}. 
Note the identification $\Phi_{\C}(y_0^{\sigma}) = \Phi_p(y_0)^{\sigma}$. 

Now $Z$ gives rise to an algebraic subvariety $Z^{\sigma} \subset \mathfrak{H}^*_{\C}$
and this subvariety again satisfies condition \eqref{dim cond}.
The functions $F_i$ are regular on an open affine containing $\Phi_p(y_0)$;
correspondingly we obtain
$F_i^{\sigma}$ on
an affine open in $\mathfrak{H}_{\C}^*$ containing $\Phi_{\C}(y_0^{\sigma})$, 
which locally cut out $Z^{\sigma}$.

Ignoring convergence for a moment, regard the $G_i$ in the completed local ring of $Y_{\Q_p}$ at $y_0$. This is a formal power series ring over $\Q_p$,
and $\sigma$ induces an injection from this completed local  
ring to the corresponding completed local ring of $Y_{\C}$ at $y_0^{\sigma}$;
call this map $G \mapsto G^{\sigma}$.  Then we have in fact  
\begin{equation} \label{KIKI} G_i^{\sigma} = \mbox{ power series expansion of } F_i^{\sigma} \circ \Phi_{\C} \mbox{ at $y_0^{\sigma}$}.\end{equation}
This follows from just the same analysis of \S \ref{GMsubsec}, or phrased informally, 
from the fact that the complex and $p$-adic period map satisfy the same differential equation.

 It follows from \eqref{KIKI} that the $G_i^{\sigma}$, {\em a priori} complex formal power series,
 are  in fact convergent in a small complex neighbourhood of $y_0^{\sigma}$;
their  common vanishing locus for a sufficiently small such neighbourhood  $V$ coincides with 
$\Phi_{\C}^{-1}(Z^{\sigma}) \cap V$.

Corollary \ref{trans_prop}, applied to $Z^{\sigma} \subset \mathfrak{H}^*_{\C}$, shows that 
$\Phi_{\C}^{-1}(Z^{\sigma}) \cap V \subset Y_{\C}$ is not Zariski dense in $Y_{\C}$.
Indeed, after analytically continuing $\Phi_{\C}$ from $V$ to a universal cover of $Y_{\C}$,  there are only finitely many irreducible components
of $\Phi_{\C}^{-1}(Z^{\sigma})$ which intersect $V$ (by local finiteness of irreducible components of an analytic set). We can apply
Corollary \ref{trans_prop} to each of them to conclude that the common zero-locus of $G_i^{\sigma}$ on $V$ is contained in the zero locus of some algebraic function $G$ (i.e., $G$ 
arises from a regular function on a Zariski-open subset of $Y_{\C}$ containing
 $V$).  
 
Consider the ring  $R_{\C} = \C\{x_1, \dots, x_n\}$ of formal power series that are convergent in some neighbourhood of $0$. 
 Given an ideal $I$ of this ring, we can associate a germ $V(I)$ of an analytic set at the origin. 
The locally analytic Nullstellensatz 
\cite[\S 3.4]{PdJ} asserts that
the ideal of functions vanishing along this germ is precisely the radical $\sqrt{I}$ of $I$. 

We apply this with $R_{\C}$ the ring of  germs of holomorphic functions near $y_0^{\sigma} \in Y_{\C}$,
taking $I$ to be the ideal 
 generated by the $G_i^{\sigma}$. Then $\sqrt{I}$
is the ideal of functions vanishing on $V(I)$ and in particular contains $G$. 
Thus  $G^m \in I$ for some $m \geqslant 1$. 
   
Therefore the ideal spanned by $G_i^{\sigma}$ 
  inside the ring of locally convergent power series contains the image of an algebraic function, i.e.\ a regular function on some
 Zariski-open subset of $Y_{\C}$ containing $y_0^{\sigma}$.  The same is then {\em a fortiori} true if we replace
 ``locally convergent'' by ``formal,'' 
 and this latter assertion can be carried back, via $\sigma^{-1}$, to $Y_{\overline{\Q_p}}$.
Thus, there is a regular function $H$,  in a neighbourhood of $y_0$ on $Y_{\overline{\Q_p}}$, belonging to the ideal
\begin{equation} \label{guruguha}  H \in \langle G_1, \dots, G_k \rangle\end{equation}
generated by the $G_i$ in the completed local ring $\widehat{\mathcal{O}}$ of $Y_{\overline{\Q_p}}$ at $y_0$.

By taking a norm we may suppose that $H$ in fact arises from a regular function 
in a neighbourhood of $y_0$ on $Y_{\Q_p}$. 
Without loss of generality (multiplying by a suitable denominator if necessary),
we may suppose that $H$ is regular on the chosen open affine around $y_0$, i.e., $H \in B_p$. 
Note that $B_p \otimes \overline{\Q_p}$ surjects on to each quotient $\widehat{\mathcal{O}}/\mathfrak{m}_{\widehat{\mathcal{O}}}^t$
(where $\mathfrak{m}_{\widehat{\mathcal{O}}}$ is the maximal ideal).
Therefore,   for each $t \geqslant 1$, there are $Z_1, \dots, Z_k \in B_p \otimes \overline{\Q_p}$ such that
\begin{equation} \label{einy} H  \in \sum Z_i G_i + \mathfrak{m}_{\widehat{\mathcal{O}}}^t.\end{equation}
By linear algebra   we see that we can even choose $Z_i \in B_p$.

The function $H$ then defines a rigid-analytic function on the residue disk of $y_0$. 
Thus $H$ and $G_i$ both lie inside the Tate algebra $R$ previously defined.
Recall that we have fixed a prime ideal $\mathfrak{p}$
of $R$,  contained in the maximal ideal $\mathfrak{m}$
associated to $y_0$, and containing the ideal $J$ generated
by the $G_i$ inside $R$. 

  Now \eqref{einy}  implies that
$$ H \in J + \mathfrak{m}^t$$
for every $t \geqslant 1$.
 Then 
the image of $H$ in $R/\mathfrak{p}$
lies in the
intersection $\bigcap_{t \geqslant 1} \mathfrak{m}^t$.
Krull's intersection theorem, applied to the Noetherian integral domain $R/\mathfrak{p}$, 
implies that the intersection of powers of $\mathfrak{m}$
is trivial. Therefore $H \in \mathfrak{p}$, as desired. 
\qed

\section{Bounds on points with good reduction}
\label{Hodge_numbers_funny}
Let $\pi: X \rightarrow Y$ be a smooth proper morphism over $\Z[S^{-1}]$, whose fibers are geometrically connected of relative dimension $d$. 
The goal of this section is to bound $Y(\Z[S^{-1}])$ by means of the same general techniques we have used elsewhere in the paper, 
i.e., by studying the variation of $p$-adic Galois representations of the fibers. We refer the reader to the Introduction (\S \ref{intro})
for a discussion of the methods and how they compare with the curve case; the main difference in this general setting
is that the linear algebra arguments required to avoid semisimplicity are much more elaborate, and are
discussed in \S \ref{GG combinatorics}.

\subsection{} \label{hodge setup subsec}
  
Fix $y_0 \in Y(\C)$, with fiber $X_0$ and set $\mathsf{V}_0 = H^d(X_{0,\C}, \Q)^{\prim}$.
This is equipped with an intersection form $\langle -, - \rangle$. Assume that 
 the image of \begin{equation} \label{large mon image} \pi_1(Y_{\C},y_0) \rightarrow \Aut\left( \mathsf{V}_0 \otimes \C,  \langle- , -\rangle \right)\end{equation}
has   Zariski closure containing the identity component of the right-hand group. 

The Hodge structure on $\mathsf{V}_0$ induces a  weight-zero Hodge structure on  
\begin{equation} \label{ZeroweightHS} \mathrm{Lie} \ \ \mathrm{GAut}\left( \mathsf{V}_0 \otimes \C, \langle-, - \rangle\right) \simeq  \C \oplus \mathrm{Sym}^2  \mathsf{V}_0 \mbox{ or } \C \oplus \wedge^2 \mathsf{V}_0, \end{equation} according to the parity of $\langle -, - \rangle$.  
 We will refer to this as the {\em adjoint Hodge structure} to distinguish it from the Hodge structure on $\mathsf{V}_0 \otimes \C$.  
 
Let  $h^p$ be the dimension of the Hodge component $(p,-p)$ in the adjoint Hodge structure.\footnote{As in Section \ref{red_galois}, we are abusing the symbol $p$ to refer to the indexing on a Hodge filtration.}  For any $E \in \mathbf{Z}_{\geqslant 0}$ that is at most the dimension of the adjoint Hodge structure,  let 
$$T(E) = \mbox{sum of the topmost $E$ Hodge numbers.}$$
Here the {\em Hodge numbers} are the list of $p$s for which $h^{p} \neq 0$,
each written with multiplicity $h^p$; 
thus, for example, if $p_{\max}$ is the largest $p$ for which $h^{p} \neq 0$,
then $T(1) = \pmax$, and if $h^{\pmax}  > 1$ then $T(2)=2\pmax$. 
 
We can extend  $T$ to be a continuous piecewise linear function $[0, \sum_{j \in \mathbf{Z}} h^j] \rightarrow \mathbf{R}_{\geqslant 0}$
such that $T(0) = 0$, and with derivative specified as
\begin{equation}
\label{Tdef}
T'(x) = 
\begin{cases}
\pmax & \mbox{ for $x \in (0, h^{\pmax})$},\\
\pmax-1 & \mbox{ for $x \in (h^{\pmax}, h^{\pmax}+h^{\pmax-1})$}, \\
\mbox{and so forth.} & \\
\end{cases}
\end{equation}

The transcendence property of period mappings is an essential ingredient in the following theorem. 
It says that integral points on the base are not Zariski dense whenever the adjoint Hodge structure is quite ``spread out,'' that is to say, whenever  the contribution of large $|p|$ to  the total dimension $\sum h^{p}$
is large.

\begin{thm} \label{transprop}
Let $\pi: X \rightarrow Y$ be a smooth proper morphism over $\Z[S^{-1}]$, whose fibers are geometrically connected of relative dimension $d$.   
With notation as above, suppose that the monodromy representation has large image, i.e.\ that \eqref{large mon image}  is satisfied, and moreover that 
\begin{equation} \label{weak condition} \sum_{p >0} h^p \geqslant h^0 + \dim(Y) \end{equation}
and
\begin{equation}  \label{condition statement}  \sum_{p>0} p h^{p} > T\left (h^{0}+\dim(Y)\right ) + T\left (\frac{3}{2} h^{0}+\dim(Y)\right ).\end{equation}
Then $Y(\Z[S^{-1}])$ is not Zariski dense in $Y$. 

If we assume, moreover, that the monodromy representation for any subvariety $Y' \subset Y$
continues to have large image\footnote{This is an unrealistically strong assumption. We include this statement simply 
to make clear the importance of this problem -- controlling monodromy drop along subvarieties -- for our method.} (see \eqref{large mon image}), then in fact $Y(\Z[S^{-1}])$ is {\em finite}. 
\end{thm}

Roughly speaking, a condition of type \eqref{weak condition} is easily seen to be necessary for our method:
with reference to the discussion of \S   \ref{subsec:outline} we want $Y$ to be transverse to all orbits
of a certain group $\mathrm{Z}(\phi)$ on a flag variety; the dimension of the flag variety is $\sum_{p > 0} h^p$,
and in our argument we shall bound the dimension of $\mathrm{Z}(\phi)$ above by $h^0$. 
Equation \ref{condition statement} is in practice a much more restrictive condition and is needed to control semisimplification.

The combinatorial machinations that
give rise to inequality \ref{condition statement} could probably be greatly optimized. We aimed to give a treatment that was fairly short,
at some cost to the sharpness of the results. 
Informally speaking, the condition says that the Hodge diamond of $Y$ is not very concentrated near the middle.  

  \subsection{Application to hypersurfaces} \label{app_hyp}
    
We will now outline the proof of the following statement:\footnote{In Section \ref{app_hyp} \emph{only}, the symbol $d$ represents the degree of a hypersurface,
and $n-1$ its dimension.}

\begin{proposition} \label{hyp prop}
There exists $n_0$ and a function $D_0(n)$ such that both \eqref{weak condition} and
\eqref{condition statement} apply to $X \rightarrow Y$ the universal family of hypersurfaces in $\mathbf{P}^n$ of degree $d$, 
so long as $n \geqslant n_0$ and $d \geqslant D_0(n)$.
\end{proposition}

Numerical experiments suggest that $n_0 \approx 60$ will do.
Note that this family indeed has large monodromy image by \cite{Beauville}.

We must emphasize that, in this case, the dimension of $Y$ is {\em very large}, and so the statement
that $Y(\Z[S^{-1}])$ is not Zariski dense is very modest indeed; but it seems to us an interesting first step,
and potentially one can then iterate the argument by replacing $Y$ by the Zariski closure of integral points. 
As suggested by the last line of the Theorem, it becomes relevant to analyze the following question: 
   
  \begin{quote}
  What is the smallest possible codimension of a subvariety $Y' \subset Y$
  along which the monodromy drops?
  \end{quote}

In outline, the proof of Proposition \ref{hyp prop} is as follows. It can be verified  (we will omit the proof) that the middle Hodge numbers $h^{pq}$ of a degree-$d$ hypersurface inside
$\mathbf{P}^n$ satisfy
\begin{equation} \label{BLestimate} h^{pq}(d) \sim \frac{d^n}{n!} A(n,p)\end{equation}
where $p+q=n-1$, and $A(n,p)$ is the \emph{Eulerian number}:  the number of permutations  $\sigma: \{1, \dots, n\} \rightarrow \{1, \dots, n\}$
with the property that $\sigma(i+1) > \sigma(i)$ for precisely $p$ values of $i$.  
(Here we fix  the dimension $n$ of the ambient projective space,  and the meaning of $\sim$ is that  the ratio approaches $1$ as $d \rightarrow \infty$.)
Now consider  $\alpha_p := \frac{1}{n!} A(n, p)$,
which defines a probability distribution on $p \in \{0, \dots, n-1\}$. 
The conclusion will be deduced, in essence, from the fact that $\alpha_p$ 
 is well approximated by a binomial distribution with mean $n/2$ and variance $n/12$.
  We now describe the details.

First, consider the Hodge numbers $h^p$ for the adjoint Hodge structure. 
Since the dimension of the symmetric or adjoint square of a $k$-dimensional
vector space equals $\frac{k^2 \pm k}{2}$, 
we have   
$$ 2 h^{p} = \sum_{p_1  + p_2 = p + (n-1)} h^{p_1, q_1} h^{p_2, q_2}  \pm h^{(p+n-1)/2,(-p+n-1)/2} $$
where in all cases $p_1+q_1=p_2+q_2=n-1$.  In particular, we deduce that 
\begin{equation} \label{adjho} h^p \sim d^{2n} \frac{1}{2}  \underbrace{ \sum_{p_1 -p_2 = p}  \alpha_{p_1} \alpha_{p_2}}_{\beta_p}.\end{equation}

Next,  note that the dimension of the moduli space of degree $d$ hypersurfaces in $\mathbf{P}^n$ is given by   $$ {n+d \choose d-1}   -1 = \frac{d (d+1) \dots (d+n)}{(n+1)!}  -1 \sim  \frac{d^{n+1}}{(n+1)!} $$
 where the meaning of $\sim$ is as before.  In particular,  for any fixed $n \geqslant 2$, 
\begin{equation}  \label{limit}  \lim_{d \rightarrow \infty} \frac{\dim Y}{h^0} = 0, \end{equation}
 where $h^0$ is the dimension of the zeroth Hodge number for the adjoint structure.  
  
Let $X(n)$ (or just $X$ for short) be the random variable which sends a uniformly distributed random permutation $\sigma$ of $\{1, \dots, n\}$ to the number
of $i$ for which $\sigma(i+1) > \sigma(i)$,  subtract $\frac{n-1}{2}$. 
Write $y_i (1 \leq i \leq n-1)$ for the random variable, on the same space, 
with value $1/2$ if $\sigma(i+1)> \sigma(i)$, and $-1/2$  if $\sigma(i+1) < \sigma(i)$. 
 Thus $X = \sum y_i$ and the expectation $\mathbb{E}(X)$ is zero. 
The variance of $X$ is then given by
\begin{equation} \label{VarX}  \mbox{Var}(X)  = \sum_{i,j} \mathbb{E}(y_i y_j) = \underbrace{ \frac{n-1}{4} }_{i=j} - \underbrace{ 2 \frac{n-2}{12} }_{|i-j|=1} = \frac{n+1}{12}.\end{equation}

Now let $X'(n)$ be the random variable obtained
by convolving $X(n)$ with itself, i.e. with adding together two copies of $X(n)$.  
Then
\begin{equation} \label{Xpvar} \mbox{Var}(X') =2  \mbox{Var}(X) = \frac{n+1}{6}, \mbox{ and  the probability that $(X' =p)$} = \beta_p,\end{equation}
where $\beta_p$ is as in \eqref{adjho}. 
Moreover,
it is also known (see \cite{CD} for discussion and references to the literature) that as $n \rightarrow \infty$, \begin{equation} \label{limit equation}
\mbox{$X(n)/\sqrt{n}$ converges in distribution to a normal distribution with variance $1/12$.}
\end{equation}
and it follows then that $X'(n)/\sqrt{n}$ converges in distribution to a normal distribution with variance $1/6$. 
It follows in particular that 
   \begin{equation} \label{slb} \sum_{p > 0} p \beta_p > A \sqrt{n}\end{equation} for some absolute $A >0$. 
 We also need:

\begin{lemma} \label{large n}
For sufficiently large $n$, we have $\beta_0 < \frac{\Hodgeestimatebound}{\sqrt{n}}$. 
\end{lemma} 

\begin{proof}
The sequence $\beta_p$ is symmetric and log-concave. The symmetry follows readily from the definition, whereas
the second statement follows from the classical fact 
that the Eulerian numbers are log-concave.  (See, for example, \cite[Thms 1.4, 3.3]{JG}.)

Let $c = \frac{1}{\Hodgeestimatebound}$. This number is chosen  to be
less than the density of the normal distribution with mean zero and variance $1/6$, 
at the point $1.1$.  From the convergence
in distribution of $X'$, it follows that for all large enough $n$ there exists $P > \sqrt{n}$ 
with the property that $\beta_{P} > \frac{c}{\sqrt{n}}$.
 
We show that  $\beta_0\leqslant \frac{c^{-1}}{\sqrt{n}}$.
Suppose not; then log-concavity means
$$ \beta_p > \frac{ (1/c)^{1-p/P} c^{p/P} }{\sqrt{n}}$$
for all $p \in [0, P]$.  In particular, this implies that $\beta_p > \frac{1}{\sqrt{n}}$
whenever $|p| \leq P/2$. This contradicts  $\sum \beta_p = 1$ for large enough $n$. 
\end{proof}

 \proof (of Proposition \ref{hyp prop}):
In what follows, write ``for big enough $n$ and $d$'' as an abbreviation 
for ``for $n \geqslant n_0$ and $d \geqslant D_0(n)$, for some function $D_0$ of $n$.''

There are two conditions to be checked, \eqref{weak condition} and \eqref{condition statement}.
 That the former  condition holds for big enough $n$ and $d$ follows from 
 \eqref{limit}, \eqref{adjho} and the convergence in distribution of $X'(n)/\sqrt{n}$. 
It remains then to verify that \eqref{condition statement} holds for  big enough $n$ and $d$.

  Write $T(y)$ for the sum of the topmost $y$ adjoint Hodge numbers and $H$ for the total dimension of the adjoint Hodge structure.  We claim that 
 $$2 T(2 h^0) < \sum_{p >0} p h^p,$$
for big enough $n$ and $d$.  That  statement readily implies the desired conclusion, in view of \eqref{limit}. 
 
By Lemma \ref{large n},  for   sufficiently large  $n$  we have $\beta_0 < \frac{\Hodgeestimatebound}{\sqrt{n}}$. Therefore,  for $d$ sufficiently large (depending on $n$) we have $h^0 < \Hodgeestimatebound H/\sqrt{n}$.
On the other hand, by \eqref{slb}, the right-hand side $\sum_{p > 0} p h^p$ is bounded below by a constant multiple of $ H \sqrt{n}$, for big enough $d$ and $n$. 
 Therefore,  it is enough to  
verify that, for fixed positive constants $c, \delta$, we have the inequality
\begin{equation} \label{md3} T(\frac{c H}{\sqrt{n}}) \leqslant \delta H \sqrt{n} \end{equation} for big enough $n$ and $d$. 
 
   Let $\epsilon = \frac{\delta}{2c}$. Separate the contribution of Hodge numbers above and below $\epsilon n$ to $T$; we get:
  $$ T(\frac{cH}{\sqrt{n}}) \leqslant   (\epsilon n) \frac{cH}{\sqrt{n}} +  \sum_{p > \epsilon n} p h^p$$
  Now the first quantity is bounded by $\frac{1}{2} \delta  H \sqrt{n}$.
The second quantity equals
  $$ \sum_{p > \epsilon n} p h^p = H   \sum_{p > \epsilon n}  p  \beta_p +   H \sum_{p > \epsilon n}  p \left( \frac{h_p}{H} -\beta_p \right).$$
There is a function $D_1$ such that, for $d \geqslant D_1(n)$, the second term is at most $  H$. Also, using the variance bound
 $\sum p^2 \beta_p = \frac{n+1}{6}$, the first term is at most $H \epsilon^{-1}$.  Thus,
 $$ T(\frac{cH}{\sqrt{n}})  \leqslant \frac{1}{2} \delta H \sqrt{n} +  H(1 + \epsilon^{-1}),$$
 and the latter term is certainly bounded above by $\frac{1}{2} \delta H \sqrt{n}$ for $n \geqslant n_0$. 
  This concludes the proof of \eqref{md3}, so also of our Proposition.

  \qed

\subsection{Setup for the proof of Theorem \ref{transprop}} 
In what follows, $\ell$ denotes an arbitrary prime number not belonging to the fixed set $S$. 

Working in the complex analytic category, let $\mathsf{V} = \mathbf{R}^d \pi_* \Q$. It is a local system of $\Q$-vector spaces
on $Y(\C)$ (and it admits an integral structure); the $\mathsf{V}_0$ defined in \S \ref{hodge setup subsec} is its fiber above $y_0$. 

Let $\mathbf{G}$ be the connected automorphism group of the intersection form on $\mathsf{V}_0$,
a semisimple $\Q$-group; 
also let 
$$ \mathbf{G}' =  \mathrm{GAut}(\mathsf{V}_0, \langle -, - \rangle),$$
the corresponding generalized automorphism group, where we permit to scale the form $\langle -, - \rangle$.

Passing to $\ell$-adic \'{e}tale cohomology,   there is a monodromy mapping
$\pi_1^{\arith}(Y, y_0) \longrightarrow \mathbf{G}'(\Q_{\ell})$
and the section associated to an integral point  $y \in Y(\Z[S^{-1}])$
gives a representation
$$ \rho_{y,\ell}: G_{\Q} \longrightarrow \mathbf{G}'(\Q_{\ell}).$$
This describes the Galois action on the primitive geometric {\'e}tale cohomology of the fiber $X_y$ in degree $d$
(after using an isomorphism $\mathsf{V}_0 \otimes \Q_{\ell} \simeq \mathsf{V}_y \otimes \Q_{\ell}$).

In what follows we will freely use certain results
about Galois representations into $\mathbf{G}'$ which are parallel to certain known results
about $\GL_n$-valued representations; we refer to \S \ref{SS} for further discussion of these points. 

We denote by $\rho_{y,\ell}^{\mathrm{ss}}$ the semisimplification of  $\rho_{y,\ell}$ relative to $\mathbf{G}'$ (see \S \ref{SS}). 
 By Faltings' finiteness theorem (Lemma \ref{faltings_finiteness_reductive})
 there are only finitely many possibilities
for  the $\mathbf{G}'(\Q_{\ell})$-conjugacy class of $\rho_{y,\ell}^{\mathrm{ss}}$.

  We must understand the variation of the representation $\rho_y$ with $y$;
as usual, we will study this using the period mapping.  We begin with the complex Hodge structures.
  
The Hodge structure on $\mathsf{V}_y$, the fiber of $\mathsf{V}$ at $y$, 
 is given by a self-dual filtration
\begin{equation} \label{FV data} \mathsf{V}_y = F^0 \mathsf{V}_y \supset \dots \supset F^i \mathsf{V}_y \supset \dots \end{equation}
 and in this way  we can regard the period mapping as  \begin{equation} \label{MT domain}  \mbox{ universal cover of }Y_{\C} \longrightarrow   \mbox{Mumford-Tate domain for $\mathbf{G}'$,}\end{equation}
 where the Mumford-Tate domain in question is understood to be the space of self-dual filtrations
 on $\mathsf{V}_0$ with the same dimensional data as the Hodge filtration on $\mathsf{V}_0$.

Also, the Hodge structure on $\mathsf{V}_0$ gives rise to a morphism
$$\varphi_{0}: S^1 \longrightarrow \mathbf{G}'(\C).$$
For each $y \in Y(\mathbf{Z}[S^{-1}])$, we may reduce modulo $\ell$ and
consider the crystalline Frobenius  of the reduction $X_{y,\mathbf{F}_{\ell}} := X_y \times_{\Z[S^{-1}]} \mathbf{F}_{\ell}$. 
This determines a transformation of the (primitive) crystalline cohomology
$$\mathrm{F}_{y}^{\crys,\ell} \in \Aut \ H^d_{\mathrm{crys}}(X_{y,\mathbf{F}_{\ell}})^{\prim}.$$
The characteristic polynomial of this endomorphism is determined by the $\zeta$-function  
of   $X_{y,\mathbf{F}_{\ell}}$, and it can be deduced (see \cite{KM})  that its eigenvalues coincide with the eigenvalues
of $\ell$-Frobenius on {\em $p$-adic} absolute {\'e}tale cohomology for any prime $p \neq \ell$.

In the coming subsections we will prove the following two Lemmas:  

\begin{lemma}[Frobenius centralizer small, for some $\ell$ below an absolute bound] \label{Goodprimelemma}
There exists an integer $L$ with the following property: 
\begin{quote} For any $y \in Y(\mathbf{Z}[S^{-1}])$, 
there exists a prime $\ell \leqslant  L, \ell \notin S$ such that
the semisimplification of $\mathrm{F}_y^{\crys,\ell}$ (and so also the crystalline Frobenius itself)
satisfies
\begin{equation} \label{crys Frob ineq} \dim Z(\left[ \mathrm{F}_y^{\crys,\ell}\right]^{\mathrm{ss}}) \leqslant  \dim Z_{\mathbf{G}'(\C)}(\varphi_0).\end{equation}
\end{quote}
On the left hand side, we take the centralizer
inside $\mathrm{GAut}(H^{d,\prim}_{\crys})$, to which the crystalline Frobenius -- and so also its semisimplification --  belongs. 
\end{lemma} 
 
\begin{lemma}[Not Zariski dense.  This is where semisimplicity gets taken care of.] 
\label{NZD}
Given a prime $\ell \notin S$ and $y_0 \in Y(\Z[\frac{1}{S}])$ with the property that the centralizer
of crystalline Frobenius $\Frob_{y_0}^{\crys, \ell}$ is at most the dimension of $Z_{G'(\C)}(\varphi_0)$, 
the set 
\begin{equation} \label{onco2} \{y \in Y(\Z[\frac{1}{S}]): y \equiv y_0 \mbox{ modulo $\ell$}, \rho_{y,\ell}^{\mathrm{ss}} \simeq \rho_{y_0, \ell}^{\mathrm{ss}}\}\end{equation}
is not Zariski dense.  (Here $\simeq$ means that the representations are $\mathbf{G}'$-conjugate). 
\end{lemma}

 Assuming these Lemmas, let us conclude the proof of Theorem \ref{transprop}.   With $L$ as in Lemma \ref{Goodprimelemma} let $N= \prod_{\ell \leqslant  L, \ell \notin S} \ell$. 
Now each $y \in Y(\Z[\frac{1}{S}])$ gives a collection of representations $\rho_y^{\mathrm{ss}}: G_{\Q} \rightarrow \mathbf{G}'(\Q_{\ell})$, one for each $\ell$ dividing $N$. 
For each $\ell$ dividing $N$, let $\mathcal{G}_{\ell}$ be the  set of  representations $G_{\Q} \rightarrow \mathbf{G}'(\Q_{\ell})$
that arises  as some $\rho_{y}^{\mathrm{ss}}$. This is a finite set (modulo conjugacy) by Lemma 
\ref{faltings_finiteness_reductive}  applied to $\mathbf{G}' \subset \mathrm{GL}(\mathsf{V}_0)$;  note that it is straightforward
to verify that the integrality of characteristic polynomial of Frobenius passes from the whole cohomology to the primitive cohomology.  

Call a pair $(y, \ell)$ as in Lemma \ref{Goodprimelemma} \emph{good} if it satisfies
\eqref{crys Frob ineq}.  For each $\ell$, Lemma \ref{NZD} and the finiteness of $\mathcal{G}_{\ell}$
guarantee that the set of $y$ for which $(y, \ell)$ is good is not Zariski dense.
Taking the union over $\ell \leq L$ and applying Lemma \ref{Goodprimelemma}, we see that
$ Y(\mathbf{Z}[S^{-1}])$ is itself not Zariski dense.

 \subsection{Proof of Lemma \ref{Goodprimelemma}}

Fix a prime $p \notin S$
and let $\rho_{y,p}: G_{\Q} \rightarrow \mathbf{G}'(\Q_p)$ be the $p$-adic Galois representation at $y$, as above.   We have observed that there are only finitely many possibilities for $\rho_{y,p}^{\mathrm{ss}}$
(here, and below, the semisimplification is taken inside $\mathbf{G}'$). 

Let $\mathbf{H}$ be the Zariski closure of $\rho_{y,p}^{\mathrm{ss}}(G_{\Q})$, with identity component $\mathbf{H}^{\circ}$. 
It is a reductive group (because we took the semisimplification, see \S \ref{SS} and references therein).    Call an element in $\mathbf{H}^{\circ}(\overline{\Q_p})$ {\em very regular} if  it is semisimple and:
\begin{quote} (*) its centralizer  inside $\Aut(\mathsf{V}_0 \otimes \overline{ \Q_p})$ has minimal dimension amongst all semisimple elements of $\mathbf{H}^{\circ}(\overline{\Q_p})$. 
\end{quote}

Choose a maximal torus $\mathbf{T}_0 \subset \mathbf{H}^{\circ}$,
and let $\Phi$ be the set of nontrivial characters $\mathbf{T}_0 \rightarrow \mathbf{G}_m$
arising from the conjugation action of $\mathbf{T}_0$ on the Lie algebra 
of $\Aut(\mathsf{V}_0 \otimes \overline{\Q_p})$. 
For $t \in \mathbf{T}_0(\overline{\Q_p})$ the dimension of the centralizer
of $t$, in $\Aut(\mathsf{V}_0 \otimes \overline{\Q_p})$, is the dimension of the centralizer of $\mathbf{T}_0$ in $\Aut(\mathsf{V}_0 \otimes \overline{\Q_p})$,  plus the number
of roots $\alpha \in \Phi$ with $\alpha(t) =1$ (counted with multiplicity). 
 The condition (*) for an element $t \in \mathbf{T}_0(\overline{\Q_p})$
amounts to asking that $\alpha(t) \neq 1$ for all $\alpha \in \Phi$. In particular:
\begin{itemize}
\item Any very regular element is regular inside $\mathbf{H}^{\circ}$, and
\item Condition (*) implies the same condition with $\Aut(\mathsf{V}_0 \otimes \overline{ \Q_p})$ replaced by $\mathbf{G}'$.
\end{itemize}

The set of very regular elements is a nonempty Zariski-open subset of $\mathbf{H}^{\circ}$ (so also of $\mathbf{H}$).  Indeed, the function
$f =\prod_{\alpha \in \Phi}  (\alpha(t)-1)$ defines a regular function on $\mathbf{T}_0$ which is invariant under the Weyl group.
Therefore $f$ extends to a regular function on $\mathbf{H}^{\circ}$, and the set of very regular elements is the locus where $f \neq 0$
(this forces semisimplicity).

It follows, then, that   the 
set of very regular elements in $\mathbf{H}(\Q_p)$ is 
the complement of a proper Zariski-closed set. 
The preimage of the very regular set  under $\rho_{y,p}^{\mathrm{ss}}: G_{\Q} \rightarrow \mathbf{H}(\Q_p)$
is nonempty,  because $\rho_{y,p}^{\mathrm{ss}}(G_{\Q})$ is Zariski-dense in $\mathbf{H}$. 
This preimage is also topologically open, since the very regular set is open.
By the Chebotarev density theorem, then, we may choose some $\ell$  
such that \begin{equation}\label{VR}
\mbox{$\rho_{y,p}^{\mathrm{ss}}(\Frob_{\ell})$ is a very regular element
of $\mathbf{H}^{\circ}$.}\end{equation}
Because  there are only finitely many possibilities for $\rho_{y,p}^{\mathrm{ss}}$, this
 $\ell$  can be taken to be bounded above by $L$ that depends only on $S, p, \dim(\mathsf{V})$.

 On the other hand,  it is known that:
 \begin{quote} the Zariski closure of $\rho_{y,p}(G_{\Q_p})$ (this is an algebraic subgroup of $\mathbf{G'}$)
 contains a  group $\mathbf{S}$  defined over $\overline{\Q_p}$, with the following property: 
 with respect to a suitable isomorphism $\overline{\Q_{\p}} \simeq \mathbf{C}$,  the group $\mathbf{S}$
 becomes isomorphic
to the Hodge torus, i.e.\ to the  Zariski closure of the image of $\varphi_0$ in $\mathbf{G}'(\C)$. 
\end{quote}

A result of rather similar nature  to the quoted statement was proved by Sen \cite{Sen} using Hodge--Tate decomposition (Sen's result 
pertains to the target group $\GL_n$).  It can be deduced using 
a   remarkable result of Wintenberger \cite{Wintenberger} about functorially splitting the Hodge filtration for Fontaine--Laffaille modules. 
This is carried out by Pink \cite[\S 2]{Pink}; this latter method also readily adapts to $\mathbf{G}'$ 
target.\footnote{We outline how this is done. 
We may describe the Zariski closure $\mathbf{Z}$ in question as the Tannakian group associated
to the neutral Tannakian category of $G_{\Q_p}$-modules generated by $\mathsf{V}_y \otimes \Q_p$ (i.e., the automorphisms
of the natural fiber functor).  By the theory of Fontaine--Laffaile, 
there is another fiber functor on this category, arising from passing to filtered $\phi$-modules;
in particular, this gives rise to another Tannakian group $\mathbf{Z}' $, which acts on the  
(primitive part of the) de Rham cohomology of $X_y \times_{\Q} \Q_p$.
These two fiber functors become isomorphic over $\overline{\Q_p}$ (cf. \cite[\S 3]{Deligne--Milne});
in particular there is an isomorphism of $\mathsf{V}_y \otimes \overline{\Q_p}$
with the de Rham cohomology of $X_y \otimes_{\Q} \overline{\Q_p}$, 
which can be taken to preserve the respective intersection forms, and
which carries $\mathbf{Z}_{\overline{\Q_p}}$ to $\mathbf{Z}'_{\overline{\Q_p}}$.

The Hodge filtration gives this fiber functor
the structure of a filtered fiber functor; it 
gives a parabolic subgroup $\mathbf{P}' \subset \mathbf{Z}'$.
Now Wintenberger's canonical splitting of the Hodge filtration provides
a character $\varphi_W: \mathbf{G}_m \rightarrow \mathbf{P}'$. 

Now pass to $\C$ by means of an isomorphism $\overline{\Q_p} \simeq \C$;
then $\mathbf{Z}'_{\C}$ acts on the cohomology of $X_y \otimes_{\Q} \C$,
as does $\varphi_0$.  We claim that $\varphi_0$ and $\varphi_W|_{S^1}$
are conjugate inside $\mathrm{GAut}(H^d(X_y \otimes_{\Q} \C)^{\prim})$;
but they both preserve the Hodge filtration and induce the same scalar on the successive quotients;
the conjugacy then follows by  Lemma \ref{wks2}.}

Thus \begin{equation} \label{zz}  \dim Z_{\mathbf{G}'}(\mathbf{S}) = \dim Z_{\mathbf{G}'(\C)}(\varphi_0). \end{equation}
 Moreover,  a $\mathbf{G}(\overline{\Q_p})$-conjugate of $\mathbf{S}$ -- call it $\mathbf{S}'$ --  is also contained in the Zariski closure of the image of $\rho_{y,p}^{\ss}(G_{\Q})$.
Indeed, 
 choose a parabolic $\mathbf{Q} \leqslant \mathbf{G}'$ containing the image of $\rho_{y,p}$ and minimal for that property;
then $\rho_{y,p}^{\ss}$ is obtained by projecting $\rho_{y,p}$ to a Levi factor of $\mathbf{Q}$,
and in particular the Zariski closure of the image of $\rho_{y,p}^{\ss}$
certainly contains the projection of the Zariski closure of the image of $\rho_{y,p}$. 
Now apply Lemma \ref{wks2}.

Now we have
 $$\left[ \rho_{y,p}(\mathrm{Frob}_{\ell})\right]^{\mathrm{ss}}  \stackrel{\text{Lemma \ref{Artss}}}{\sim} \left[\rho_{y,p}^{\mathrm{ss}}(\mathrm{Frob}_{\ell})\right]^{\mathrm{ss}}
\stackrel{\eqref{VR}}{=}   \rho_{y,p}^{\mathrm{ss}}(\mathrm{Frob}_{\ell}) $$
 where $\sim$ denotes $\mathbf{G}'(\overline{\Q_p})$-conjugacy.  
 By \eqref{VR},  the definition of ``very regular'' element, and the discussion that follows it,  the centralizer of this element
inside $\mathbf{G}'$ is as small as possible,  amongst semisimple elements in $\mathbf{H}^{\circ}(\overline{\Q_p})$.  In particular, this centralizer is at most as large as the centralizer of $\mathbf{S}'$ on $\mathbf{G}'$, and so 
 \begin{equation} \label{one} \dim Z_{\mathbf{G}'} \left[ \rho_{y,p}(\Frob_{\ell})\right]^{\mathrm{ss}} \leqslant  
\dim Z_{\mathbf{G}'}(\mathbf{S}') = \dim Z_{\mathbf{G}'}(\mathbf{S}) = \dim Z_{\mathbf{G}'(\C)}(\varphi_0).\end{equation}
 
 We  now transfer this to the corresponding assertion for the crystalline Frobenius $\Frob_{\ell}^{\crys}$. 
We know that 
the crystalline  $\ell$-Frobenius on the $\ell$-adic vector space $H^d_{\crys}(X_{y,\mathbf{F}_{\ell}})$ and the usual $\ell$-Frobenius on
the $p$-adic geometric {\'e}tale cohomology of $X_y$
have the same characteristic polynomial.   The same is true for primitive parts. 
Thus 
$\rho_{y,p}(\Frob_{\ell})^{\ss}$
and $(\Frob_{\ell}^{\crys})^{\ss}$ both have the same characteristic polynomial;
also they both scale the 
bilinear forms by $\ell$.

Split $\mathsf{V}_0 \otimes \Q_p = \bigoplus \mathsf{V}_{\lambda}$ into eigenspaces
for $\rho_{y,p}(\Frob_{\ell})^{\ss}$.  The biinear form gives a perfect pairing  between each $\mathsf{V}_{\lambda}$ and $\mathsf{V}_{\ell \lambda^{-1}}$ 
(interpreted as a self-pairing when $\lambda^2=\ell$); 
the centralizer of $\rho_{y,p}(\Frob_{\ell})^{\ss}$ in $\mathbf{G}'$ is the set of $g$
stabilizing each $\mathsf{V}_\lambda$ and respecting these pairings. 
In particular the centralizer dimension is determined by the function $\lambda \mapsto \dim(\mathsf{V}_{\lambda})$;  the same analysis applies for $(\Frob_{\ell}^{\crys})^{\ss}$.   We deduce that 
\begin{equation} \label{two} \dim Z_{\mathrm{GAut}}(\left[ \Frob_{\ell}^{\crys}\right]^{\mathrm{ss}}) = \dim Z_{\mathbf{G}'}(\rho_{y,p}(\Frob_{\ell})^{\mathrm{ss}} )
\stackrel{\eqref{one}}{\leqslant}  \dim Z_{\mathbf{G}'(\C)}(\varphi_0),\end{equation}
concluding the proof of the Lemma. \qed
 
\subsection{Proof of Lemma \ref{NZD}}
We must analyze the set
\begin{equation} \label{onco23} \{y \in Y(\Z[\frac{1}{S}]): y \equiv y_0 \mbox{ modulo $p$}, \rho_{y,p}^{\mathrm{ss}} \simeq \rho_{y_0, p}^{\mathrm{ss}}\}\end{equation}
(we have  switched from $\ell$ to $p$ for typographical simplicity).  Here
we are assuming that the centralizer
of crystalline Frobenius $\Frob_{y_0}^{\crys, p}$, inside the group $\mathrm{GAut}$ of generalized automorphisms of the intersection pairing, has dimension at most the dimension of $Z_{G'(\C)}(\varphi_0)$,

Now let us unwind the condition in \eqref{onco23}, namely, that the semisimplified $p$-adic Galois representations  for $y$ and for $y_0$ are isomorphic. 
Recall that semisimplification is taken relative to the ambient group $\mathbf{G}'(\Q_p)$. The representation $\rho_{y,p}$ is realized on $H^d(X_y, \Q_p)^{\prim}$, and similarly for $y_0$. 
The semisimplification of $\rho_{y_0,p}$ (in the ambient group $\mathbf{G}'$) is obtained by taking a maximal  self-dual flag of $\rho_{y_0,p}$-stable subspaces
$$0 \subset \mathfrak{f}^1 \subset \mathfrak{f}^2 \subset \dots \subset  \mathfrak{f}^{\Jmid} \subset \underbrace{ (\mathfrak{f}^\Jmid)^{\perp} }_{\mathfrak{f}^{\Jmid+1}}\subset \dots  \subset H^d(X_{y_0}, \Q_p)^{\prim}$$ with the property that the representation on each graded piece is irreducible. (For the middle graded piece,
i.e.\ the piece $\mathfrak{f}^{\Jmid+1}/\mathfrak{f}^{\Jmid}$, 
we interpret ``irreducible'' to mean that there is no {\em isotropic} invariant subspace, see \S \ref{SS} for explanation. We also permit
the possibility that $\mathfrak{f}^{\Jmid} = \mathfrak{f}^{\Jmid+1}$.)

Since $\rho_y^{\mathrm{ss}}$ and $\rho_{y_0}^{\mathrm{ss}}$ are isomorphic, it means that 
there exist such flags $\mathfrak{f}_{y}$ and $\mathfrak{f}_0$ for both $y$ and $y_0$ such that
the $G_{\Q}$-representations on $\bigoplus_j \mathrm{gr}^{\mathfrak{f}}_j $ are isomorphic.
In fact, we can arrange even that this is true for every $j$ individually, and that the isomorphism preserves the intersection form for $j = \Jmid$:
this follows by using the last sentence of Lemma \ref{faltings_finiteness_reductive}. \footnote{
In more detail:  In our reasoning to date, 
instead of using the finiteness of conjugacy classes 
of possible $\rho_y^{\mathrm{ss}}: G_{\Q} \rightarrow \mathbf{G}'$, we could instead use the stronger finiteness provided by the last sentence of Lemma \ref{faltings_finiteness_reductive} . Namely, we fix 
for each $y$ a parabolic subgroup $\mathbf{Q}_y$ containing
the image of $\rho_y$, such that the projection of $\rho_y$ to its Levi
gives the semisimplification, and then use the finiteness up to conjugacy
of possible pairs $(\mathbf{Q}_y, \rho_y^{\mathrm{ss}})$. 
}

Now the functors of $p$-adic Hodge theory 
carry $H^d(X_{y_0},\Q_p)$ to $H^d_{\dR}(X_{y_0}, \Q_p)$
and similarly for $y$. Moreover, the intersection form
$$H^d(X_{y_0}, \Q_p) \otimes H^d(X_{y_0}, \Q_p) \longrightarrow H^{2d}(X_{y_0},\Q_p) \left( \simeq \Q_p(-d) \right)$$
is  carried to the intersection form
$H^d_{\dR}(X_{y_0}, \Q_p) \otimes H^d_{\dR}(X_{y_0}, \Q_p)
\longrightarrow  H^{2d}_{\dR}(X_{y_0},\Q_p).$ These assertions remain valid for the primitive parts of cohomology.

The flags $\mathfrak{f}_y$ and $\mathfrak{f}_0$ are in particular $G_{\Q_p}$-invariant, and, under the correspondence of $p$-adic Hodge theory, these flags $\mathfrak{f}_y, \mathfrak{f}_0$ correspond
to self-dual flags $\mathfrak{f}^{\dR}_y$ and $\mathfrak{f}^{\dR}_{0}$ inside the associated ``de Rham'' vector spaces:
$$ \mathfrak{f}_y^{\dR} \mbox{ in } H^d_{\dR}(X_y)^{\prim} \mbox{ and } \mathfrak{f}_0^{\dR} \mbox{ in } H^d_{\dR}(X_{y_0})^{\prim}.$$ 

Moreover, under the correspondence of $p$-adic Hodge theory,
the filtered $\phi$-modules
$$ (\mathfrak{f}^{\dR}_y)^{m+1}/(\mathfrak{f}^{\dR}_y)^{m} \mbox{ and }
(\mathfrak{f}^{\dR}_0)^{m+1}/(\mathfrak{f}^{\dR}_0)^{m}$$
correspond, respectively, to the Galois representations of $G_{\Q_p}$
on $\mathfrak{f}_y^{m+1}/\mathfrak{f}_y^{m}$ and $\mathfrak{f}_0^{m+1}/\mathfrak{f}_0^m$.
These Galois representations are isomorphic, so the filtered $\phi$-modules  just mentioned above are also isomorphic.  For $m=\Jmid$, the middle degree, the isomorphism of Galois representations can be taken
to preserve the bilinear form, and so the same is true for the isomorphism of filtered $\phi$-modules.

The map sending $y$ to the Hodge filtration on $X_y$ defines a period map
 $$  \Phi_p: \mbox{residue disk at $y_0$, modulo $p$} \longrightarrow \mbox{$p$-adic period domain $\mathfrak{H}_p$}$$
where $\mathfrak{H}_p$ is now the set of self-dual flags inside $V:= H^d_{\dR}(X_{y_0},  \Q_p)^{\prim}$ with the same
dimensional data as the Hodge filtration on $H^d_{\dR}(X_{y_0})^{\prim}$. 
Write $\phi$ for the Frobenius map on $V$.  
Our analysis above shows that the
 the set $ \{y \in Y(\Z[\frac{1}{S}]): y \equiv y_0 \mbox{ modulo $p$}, \rho_{y,p}^{\mathrm{ss}} \simeq \rho_{y_0, p}^{\mathrm{ss}}\}$
is contained in a finite union of sets of the following type:
$$ \Phi_p^{-1}(\mathfrak{S}),$$
where $\mathfrak{S} \subset \mathfrak{H}_p$ is the space of filtrations $F$ on $V = H^d_{\dR}(X_{y_0}, \Q_p)^{\prim}$ with the property that
there exists {\em another} self-dual filtration $\mathfrak{f}$, the ``semisimplification filtration'':
$$0  = \mathfrak{f}^0 \subset \mathfrak{f}^1 \subset \mathfrak{f}^2 \subset \dots \subset  \mathfrak{f}^{\Jmid} \subset \underbrace{ (\mathfrak{f}^\Jmid)^{\perp} }_{\mathfrak{f}^{\Jmid+1}}\subset \dots  \subset \mathfrak{f}^{2 \Jmid +1} =  V$$
with the following properties:

\begin{itemize}
\item $\mathfrak{f}$ is $\phi$-stable.
\item The filtration induced by $F$ on each graded piece $\mathrm{gr}^{\mathfrak{f}}_j$
has weight equal to $d/2$. (This follows because it arises from applying $p$-adic Hodge theory to the restriction of a global Galois representation that is pure of weight $d$, using  Lemma \ref{globalsimple}.)
\item  We have an isomorphism of filtered $\phi$-modules
$$ (\mathrm{gr}^{\mathfrak{f}}_j, \mbox{filtration induced by $F$}) \simeq (\mathrm{gr}^{\mathfrak{f}_0}_j, \mbox{filtration induced by $F_0$})$$
(i.e., an isomorphism of vector spaces respecting Hodge filtration and Frobenius).
In particular, the left-hand side of the above equation lies in a fixed isomorphism class. 

On the right hand side $F_0$ is the filtration at $y_0$. 
 In the case of the middle graded piece $j =\Jmid$, the isomorphism above may be taken, moreover, to 
preserve the bilinear forms on both sides.   \end{itemize}

The following Proposition \ref{linalg} implies
that the codimension of the set $\mathfrak{S}$ above
is at least equal to the dimension of $Y$.  
Given this Proposition, Lemma \ref{NZD} now follows from 
Lemma \ref{padic BT} (the $p$-adic transcendence of period mappings).

\begin{proposition} \label{linalg}
Suppose $V$  is a vector space over the field $K$ 
 equipped with a bilinear form $\langle -, - \rangle$
and a linear automorphism $\phi \in \mathrm{GAut}(V)$.  

Suppose $A_1, \dots, A_{\Jmid}$ is a collection of $K$-vector spaces,
each equipped with a decreasing filtration and a 
linear automorphism $\phi_i: A_i \rightarrow A_i$. 
We suppose the final space $A_{\Jmid}$ is equipped with a bilinear form $\langle -, - \rangle$. 

Consider all self-dual filtrations 
$$V = F^0 V \supset F^1  V\supset \dots \supset F^d V  \supset F^{d+1}V = \{0\}$$
on $V$, where we fix the dimensions of each $F^i$. 

Call such a filtration $F$ ``bad'' if there  exists another self-dual filtration $\mathfrak{f}$ on $V$   
$$ 0 = \mathfrak{f}^0 \subset \mathfrak{f}^1 \subset \dots \subset \mathfrak{f}^{\Jmid} \subset \mathfrak{f}^{\Jmid+1} \subset \dots \subset \mathfrak{f}^{2\Jmid+1} = V.$$
such that the following conditions hold.
 \begin{itemize}
\item[(a)] $\mathfrak{f}$ is $\phi$-stable.
\item[(b)] The weight of the filtration induced by $F$ on each graded piece $\mathrm{gr}^k_{\mathfrak{f}}$
equals $d/2$, i.e.\ the weight of the filtration $F$ on $V$. 
\item[(c)]   There exists an isomorphism of filtered $\phi$-modules:
$$ \left( \mathrm{gr}_{\mathfrak{f}}^j V, \mbox{ filtration induced by $F$} \right) \simeq A_j$$
for each $j \leq \Jmid$, and in the middle dimension $j=\Jmid$ this also preserves bilinear forms. 
\end{itemize}

Define the Hodge numbers $h^p$ as the dimension of $\mathrm{gr}^p_F \ \mathrm{Lie} \mathrm{GAut}(V)$; 
let $T(y)$ be the sum of the topmost $y$ Hodge numbers, extended by linearity as in \eqref{Tdef}. 

Put $z = \dim Z(\phi^{\mathrm{ss}})$, the dimension of the centralizer of the semisimple part of $\phi$ in $\mathrm{GAut}(V)$. 

 If $e$ is a positive integer such that 
\begin{equation} \mbox{number of positive Hodge numbers} \geqslant z+e \end{equation} 
and
\begin{equation} \label{equat}  \mbox{sum of all positive Hodge numbers} >  T(z+e)   + \ T \left( \frac{h^0}{2}+z+e \right),\end{equation}  then the codimension of the space of bad filtrations 
is greater than or equal to $e$.  
\end{proposition}

To be clear, we apply this with:
\begin{itemize}
\item  $K = \Q_p$
and $V = H^d_{\dR}(X_{y_0}, \Q_p)^{\prim}$
for some fiber of the family of  Theorem \ref{transprop};
\item 
The filtration $F$ comes from the Hodge filtration on $X_y$, where $y$ lies in the residue disk of $y_0$.
\item 
$\mathfrak{f}$ is another filtration which comes from a potential failure of the global
Galois representation at $y$ to be semisimple; 
the passage to the graded $\mathrm{gr}_{\mathfrak{f}}$ affects semisimplification of the Galois representation.
\item Condition (b) comes eventually from global purity.
\item  We have $z \leqslant h^0$ by  assumption (this came
from Lemma \ref{Goodprimelemma}) and we take $e=\dim(Y)$. 
\end{itemize}

The statement of Proposition \ref{linalg} is complicated -- it is an analogue, in our current setting, of Lemma \ref{generalsimple2}.  We'll offer some vague motivation here. 
Proposition \ref{linalg} asks: for which filtrations $F$ on $V$ do $A_1, \ldots, A_{\Jmid}$ form a composition series for $V$ as filtered $\phi$-modules?
Of course, filtered $\phi$-modules are in general far from simple, and the choice of $F$ often amounts to a choice of extension class.
Based on this, one might expect that the space of bad filtrations is large, perhaps even Zariski dense in the flag variety.
This does not happen here because of the condition that the weight of the filtration on each $A_i$ equal $d/2$.
This equal-weight condition generalizes Equation (\ref{crit}).
Requiring the subobjects $\mathfrak{f}^j V$ giving the composition series to have large intersection with pieces of the filtration
turns out to impose strong conditions on the filtration $F$.
This is the content of Proposition \ref{linalg}.
     
\section{Combinatorics related to reductive groups} \label{GG combinatorics}

It remains to prove Proposition 
\ref{linalg}  from the prior section. This is ``just'' a problem in linear algebra but it is a notational mess. 
We analyze it using some  simple ideas about root systems. Although we work in the generality of an arbitrary reductive group, to help the exposition
we will often explicate the discussion in the case of $\GL_n$.  One other reason we chose to work in this generality is that
analysis of this type is likely necessary when  carrying out a similar analysis for more general monodromy groups. 

Since Proposition \ref{linalg} is geometric, concerning the dimensions of certain algebraic sets,
we can and will suppose that the base field $K$ is algebraically closed. We will therefore permit
ourselves to identify algebraic groups with their $K$-points; they will be correspondingly denoted
by usual letters $P, G$ etc., rather than boldface letters
as we have done previously. 

There is a correspondence between filtrations and parabolic subgroups.  We have a question about
the interaction of two filtrations $\mathfrak{f}$ and $F$; we're going to convert it to a question about
the interaction of two parabolic subgroups $P$ and $Q$.

One important warning: As defined  $\mathfrak{f}$ is an increasing filtration, whereas $F$ is decreasing.  %
However, in actual fact, the indexing of $\mathfrak{f}$ is irrelevant. All that
will matter throughout is the stabilizer of $\mathfrak{f}$; we could re-index it to be a decreasing filtration
and nothing at all would change.   On the other hand,
the indexing of $F$ {\em does} matter, and thus we will need to keep track of extra data beyond its stabilizer.

Tracing back the origins of Proposition \ref{linalg}, $\mathfrak{f}$
comes from the semisimplification filtration on a global Galois representation, and $F$ from the Hodge filtration.  
The following informal dictionary may be helpful, at least in interpreting the material 
 from \S \ref{indfil2} onward:
\begin{itemize}
  \item The parabolic denoted  $P$  should be thought of as the stabilizer of the semisimplification filtration $\mathfrak{f}$. 
 
\item The Levi quotient $M$ of $P$ corresponds to the associated graded for $\mathfrak{f}$;
 globally, 
 the semisimplification of the Galois representation takes values in $M$.
 \item The parabolic   $Q$ should be thought of as the stabilizer of the Hodge filtration $F$. 
 \end{itemize}
  
The argument  can be informally summarized like this: 
\begin{itemize}
\item First of all, we bound the number of possibilities for $\mathfrak{f}$, using the fact that it is $\phi$-stable.
This uses the fact that the centralizer of $\phi^{\mathrm{ss}}$ is not too large and happens
in \eqref{1db}. 
 After this point, it is enough to work with a given $\mathfrak{f}$ and $P$. 
\item   Having fixed $\mathfrak{f}$ and $P$, we break up the space of possible $F$ into $P$-orbits. 
The set of $F$ satisfying the weight condition (b) of   Proposition \ref{linalg},
is a union of $P$ orbits. We need to show that no $P$ orbit of small codimension
occurs in this set. 
 
\item    
To illustrate the idea, we will just explain why the {\em open} $P$ orbit doesn't occur.
Suppose $F$ satisfies the weight condition (b) of Proposition \ref{linalg}.  We show then
that $PQ/Q$ is not open in $G/Q$. 

We find a maximal torus $T \subset P \cap Q$ and a character $\nu: \mathbb{G}_m \rightarrow T$ which
defines the filtration $F$. In particular, $Q$ consists of non-negative root spaces for $\nu$. 
The weight condition  will imply that 
\begin{equation} \label{weight}  \sum_{\gamma \in \Sigma-\Sigma_P} \langle \nu, \gamma \rangle = 0,\end{equation}
the sum being taken over roots $\Sigma$ for $T$ that correspond to root spaces outside $P$.

By using the assumed numerology of Hodge numbers, not too many of these $\langle \nu, \gamma \rangle$ can be zero. 
In particular, \eqref{weight} implies that
$\langle \nu, \gamma \rangle  < 0$ for at least one $\gamma \in \Sigma-\Sigma_P$. 

That means there is at least one such root $\gamma \in \Sigma-\Sigma_P$ that doesn't belong to the Lie algebra of $Q$;
equivalently, 
$$ \mathrm{Lie}(Q) + \mathrm{Lie}(P) \neq \mathrm{Lie}(G),$$
which implies the desired conclusion.
\end{itemize}

\subsection{Filtrations on reductive groups}
 
Let $G$ be a reductive group over an algebraically closed field $K$. 
 
A (rational) cocharacter $\lambda : \G_m \dashrightarrow G$
is simply a co-character that is allowed to be defined on a finite cover of $\G_m$. 
It determines a parabolic $P_{\lambda}$, whose Lie algebra is the sum of non-negative weight spaces
for $\lambda$; the centralizer of $\lambda$ is  therefore a Levi factor for this parabolic. 
A ``filtration'' for $G$ will be, by definition, 
an equivalence class  of  such rational cocharacters $\lambda$, where $\lambda \sim \lambda'$ 
if $\lambda'$ is conjugate to $\lambda$ under $P_{\lambda}$ (or equivalently 
under the unipotent radical of $P_{\lambda}$). 

\begin{example}
Filtrations. %
\begin{itemize}
\item
A filtration on $G = \GL(V)$  is the same as a (decreasing) filtration $F^{\bullet}V$ on $V$,
where the indices are indexed by rational numbers. 
Specifically we set
\begin{equation} \label{fp formula}  F^p V = \mbox{ sum of all  weight spaces for $\lambda$ on $V$ with weights $\geqslant p$}\end{equation}
The associated parabolic $P_{\lambda}$ is precisely the stabilizer of this filtration. 
 
Note that $F^{\bullet} V$ determines $\lambda$ up to the equivalence described above:
any two rational characters $\mathbf{G}_m \rightarrow P$
with the same projection to a Levi quotient are actually $P$-conjugate by Lemma \ref{wks2}. 
  
\item
If $V$ is equipped with a bilinear form $\langle -, - \rangle$, then a filtration
on $\mathrm{GAut}(V, \langle -, - \rangle)$ is the 
same as a {\em self-dual} filtration on $V$, again via the formula \eqref{fp formula};
more precisely, if the filtration $F$ corresponds to a character $\chi: \mathbf{G}_m \dashrightarrow \mathrm{GAut}(V)$  
for which $\chi(x)$ scales the form by $x^{r}$, then 
\begin{equation}\label{FFF} \mbox{ $F^p V$ and $F^{r-p+\epsilon} V$ are orthogonal complements of  one another 
(for sufficiently small $\epsilon$). }
\end{equation} 
\end{itemize}
\end{example}
 
A map $G_1 \rightarrow G_2$ of reductive groups induces, obviously, a map from
filtrations for $G_1$ to filtrations for $G_2$. Thus a filtration on $G$ determines a filtration on the underlying space of any $G$-representation. 
If $G = \GL_n$, this corresponds to the usual way in which a filtration on $V$
induces (e.g.) a filtration on $V \otimes V, V^*$, etc. 

Indeed, for a general group $G$, to give a filtration of $G$ is the same
as giving a filtration functorially on every representation of $G$: this is part of the theory
of filtered fibered functors, see \cite[Section IV.2.1]{NSR}.

For any reductive group $S$ write $$\mathfrak{a}_{\mathbf{S}} := X_*( Z_{S}) \otimes \Q$$ where $Z_{S}$ is the center.  (As usual,
we write $X_*$ for cocharacters and $X^*$ for characters.) This space is canonically in duality 
with $X^*(S) \otimes \Q$.   If $F$ is a filtration on $S$
the projection of the associated cocharacter to the torus quotient of $S$ 
defines a class in $\mathfrak{a}_S$. We call this the weight of $F$:
$$ \wt(F) \in \mathfrak{a}_S.$$

\begin{example}
Weights of filtrations. %
\begin{itemize}
\item
For $\GL(V)$, $\mathfrak{a}_{\GL(V)}$is a one-dimensional $\Q$-vector space.
We identify it with $\Q$ by identifying the character
$t \in \mathbf{G}_m \mapsto t \mathrm{Id}_V$ with $1 \in \Q$. 
With this identification, 
the weight of the filtration  on $\GL(V)$, corresponding to $F^p V$ as in (a) above,  is 
$\frac{\sum_{p} p \dim(F^p/F^{p+1})}{\dim V}$; thus this definition coincides
with our previous definition (\ref{weight definition}). 

\item  For $\mathrm{GAut}(V)$, we can make the same identification of $\mathfrak{a}$ with $\Q$
as for $\GL$. With this identification, the weight of the filtration described before \eqref{FFF} is necessarily equal to $r/2$,
one-half of the integer by which the associated character scales the form. 
\end{itemize}
\end{example}

We can alternately describe filtrations using parabolics: 
For $\lambda: \mathbf{G}_m \rightarrow G$
the projection of $\lambda$ to the Levi quotient $M_{\lambda}$ of the parabolic $P_{\lambda}$
is central in $M_{\lambda}$; thus we get a class  $\bar{\lambda}$ in $\mathfrak{a}_{M_{\lambda}}$. 
The pair $(P_{\lambda}, \bar{\lambda} \in \mathfrak{a}_{M_{\lambda}})$
depends only on the filtration associated to $\lambda$, and moreover
completely determines that filtration, because of Lemma \ref{wks2}. 
In fact, any pair $(P, e \in \mathfrak{a}_M)$
of a parabolic and a ``strictly positive'' element of $\mathfrak{a}_M$,
i.e.\ positive on all roots in the unipotent radical of $P$, arises from
a filtration. 
 
\subsection{Levi subgroups}  \label{sec:Levi}
 Now suppose that $N$ is a Levi subgroup of $G$. 
 The center of $N$ then contains the center of $G$. In this way we obtain a map
\begin{equation} \label{aGN} \mathfrak{a}_G \longrightarrow \mathfrak{a}_N\end{equation}
which is naturally split: A character of $G$, i.e.\ a homomorphism $G \rightarrow \mathbf{G}_m$,
 can be pulled back to a character of $N$. The resulting map
 $$ \underbrace{ X^*(G) \otimes \Q}_{\simeq \mathfrak{a}_G^*} \longrightarrow \underbrace{ X^*(N) \otimes \Q}_{\simeq \mathfrak{a}_N^*}$$ 
gives rise to a splitting of \eqref{aGN}. 

\begin{example}
If $\dim V_i = n_i$ then $\GL(V_1) \times \GL(V_2)$ is a Levi subgroup of $\GL(V_1 \oplus V_2)$.
We identify $\mathfrak{a}_N  = \Q^2$ as in the previous example;
then $\mathfrak{a}_G$ is embedded as the subspace  $(1, 1)$
and the complementary subspace is spanned by $(-\dim(V_2), \dim(V_1))$. 
\end{example}

\subsection{The induced filtration on a Levi subgroup} \label{indfil2}

If $V$ is a vector space equipped with
filtrations $F^{\bullet}$ and $\mathfrak{f}^{\bullet}$, then $F^{\bullet}$
induces a filtration on $\mathrm{gr}^{\mathfrak{f}}_* V$. 
We need to analyze this induced filtration carefully when $F^{\bullet}$ the Hodge filtration and 
$\mathfrak{f}^{\bullet}$ the semisimplification filtration.    

It is convenient to again express this abstractly:
For any reductive group $G$ and any parabolic $P$,
a filtration  $F$ on $G$ induces a filtration  $F_M$ on the Levi quotient $M$ of $P$.
(With reference to the example above, $P$ corresponds to the filtration $\mathfrak{f}^{\bullet}$,
and $M$ to the associated graded). To explain this we require the following Lemma:
 
\begin{lemma}
Let  $\chi: \mathbf{G}_m \rightarrow G$ be a character defining the parabolic subgroup $Q$. 
Let $Q$ act transitively on an algebraic variety $Y$. Then all fixed points
of $\chi$ on $Y$ are conjugate under the centralizer $N$ of $\chi$. 
 \end{lemma}
 \proof
We have a Levi decomposition
 $Q=NV$, with $V$ the unipotent radical.
Suppose that $y_0 \in Y$ is $\chi$-fixed.   It is enough to verify that $y_0$ is the unique point in $Vy_0$ that is $\chi$-fixed. 
Let $V_0$ be the stabilizer of $y_0$ inside $V$. For $x \in \mathbf{G}_m$, 
$$\chi(x) \cdot (vy_0) = (\chi(x) v \chi(x)^{-1}) y_0,$$
and thus the $\chi$-fixed points on $Vy_0$ correspond to the
fixed points for $\chi(x)$-conjugation on $V/V_0$. 
 
 But all the weights of this $\mathbf{G}_m$-action on $V$ are positive,
 i.e.\ the limit of $\chi(x) v \chi(x)^{-1}$ as $x \rightarrow 0$ is equal to the identity. 
  Therefore the only  fixed point on $V/V_0$
 for  conjugation  by $\chi(\mathbf{G}_m)$ is the identity coset. 
\qed

Before we formulate the induced filtration in terms of parabolics, we recall some linear algebra associated to two parabolics.
Suppose that $P,Q$ are parabolic subgroups of $G$, where $Q$ is the stabilizer of a filtration $F$. 
It is known that $P$ and $Q$ contain a common maximal torus $T$ and that $P \cap Q$ is connected;
this, together with everything else we will use is contained in \cite[Chapter 2]{DigneMichel}.   We will briefly summarize what we need. 
 
Fixing $T$ as above, we get Levi decompositions of $P$ and $Q$ such that both Levi factors contain $T$: 
\begin{equation} \label{pmu}  P = M U \mbox{ and }Q = N V\end{equation}
 We have a factorization
\begin{equation} \label{factorization} P \cap Q = (M \cap Q) \cdot (U \cap Q).\end{equation}
In particular, this implies that the projection of $P \cap Q$ to $M$ along $P \twoheadrightarrow M$
is just $M \cap Q$.

 To verify this factorization, we note that  $(M \cap Q)$ normalizes $(U \cap Q)$,
 and also that it is easy to verify the corresponding splitting at the level of Lie algebras;
 since $P \cap Q$ is connected, this factorization also follows.
  
  \begin{lemma}\label{indfil} (Induced filtration on the Levi factor of a parabolic.)
  Let $F$ be a filtration for the group $G$. 
  There exists a   
   representative $\chi_P: \mathbf{G}_m \rightarrow G$
for the filtration $F$ with the property that $\chi$ is valued in $P$.  Moreover, any two such representatives
are conjugate under $P \cap Q$.

For each such representative $\chi_P$, the projection of $\chi_P$ to the Levi quotient $M$ of $P$ defines a filtration 
on $M$ which is independent of the choice of $\chi_P$. 
 \end{lemma}

\proof 
Let $\chi: \G_m \rightarrow G$ represent the filtration, and let $Q$ be the associated parabolic. 
The intersection $P \cap Q$ contains a maximal torus $T$ of $G$ and we may certainly conjugate
$\chi$ so it is valued in $T$, so in $P$; this proves the existence statement.

For uniqueness fix $\chi_P$, which we may now suppose to be valued in $P \cap Q$. 
Now the image of $\Ad(q)^{-1} \chi_P$ is in $P$ if and only if 
$$\chi_P(\mathbf{G}_m) \subset \Ad(q) P,$$
i.e.\ $qP$ lies in the set of fixed points of $\chi_P(\mathbf{G}_m)$ on $QP/P$.
These fixed points are all conjugate under $N$, the centralizer of $\chi_P$,  as we have seen above;
thus $qP \in NP$, so that $q \in N(Q \cap P)$. 
Thus the characters $\Ad(q)^{-1} \chi_P$ lie in a single $(P \cap Q)$-orbit.  

It remains to prove the final statement. Choose $\chi_P, \chi_P': \mathbf{G}_m \rightarrow P$, as above, both   $P$-valued representatives for the filtration $F$.
We have $\chi_P' = \Ad(g) \chi_P$ for some $g \in P \cap Q$,
so $\overline{\chi_P'} = \Ad(\bar{g}) \overline{\chi_P}$, where bars denote projection to the Levi quotient of $P$. 
To see that these two characters define the same filtration we need to verify that
$$  \overline{g} \in Q_{\overline{\chi_P}}.$$

This follows from the remark after \eqref{factorization}: extend the image of $\chi_P$ to a maximal torus
inside $P \cap Q$;  then, with the corresponding choice $M$ of Levi subgroup for $P$, we have
$\overline{\chi_P} = \chi_P$ and $Q_{\overline{\chi_P}} = Q \cap M$. 

\qed

\begin{example}
Induced filtration on the associated graded. 
\begin{itemize}
\item
Consider the case of $G=\GL(V)$.  Suppose given a decreasing filtration $F^{\bullet} V$
(with associated parabolic $Q$) and another parabolic $P$; we fix an increasing filtration $\mathfrak{f}^{\bullet} V$  with stabilizer $P$.
We show that the construction above gives precisely the filtration induced by $F$ on 
the associated graded to $\mathfrak{f}$. 
  
As above, we can represent the character for $F$ by a character $\chi$ {\em preserving the filtration $\mathfrak{f}$}.
Then, writing $\bar{F}$ for the induced filtration:
$$ \bar{F}^j(\mathfrak{f}^k/\mathfrak{f}^{k-1}) $$
is the sum of all eigenspaces with weights $\geqslant j$; this is the image of the corresponding space in $\mathfrak{f}^k$, that is to say, 
$$ \bar{F}^j (\mathfrak{f}^k/\mathfrak{f}^{k-1}) = \mbox{image of $F^j \cap \mathfrak{f}^k$ in $\mathfrak{f}^k/\mathfrak{f}^{k+1}$}.$$
  
\item We now modify the example above by taking $G$ to be $\mathrm{GAut}(V, \langle -, - \rangle)$
for some symmetric or skew-symmetric nondegenerate bilinear pairing $\langle -, - \rangle$.
Now suppose that $P$ is a parabolic subgroup of $G$, stabilizing the self-dual increasing
filtration
$$ 0 = \mathfrak{f}^0 \subset \mathfrak{f}^1 \subset \dots \subset \mathfrak{f}^{\Jmid} \subset \mathfrak{f}^{\Jmid+1} \subset \dots  \subset \mathfrak{f}^{2\Jmid+1} = V.$$
Just as before, $F$ induces a filtration $\bar{F}$ on each graded piece $\mathfrak{g}^j = \mathfrak{f}^j/\mathfrak{f}^{j-1}$. 
The associated Levi subgroup is isomorphic to
$$ \GL(\mathfrak{g}^1) \times \cdots \times \GL(\mathfrak{g}^{\Jmid}) \times \mathrm{GAut}(\mathfrak{g}^{\Jmid+1}),$$
where we regard the last factor as $\mathbf{G}_m$ even if $\mathfrak{f}^{\Jmid} = \mathfrak{f}^{\Jmid+1}$,
and the corresponding filtration on each factor is the one
induced by $\bar{F}$.
\end{itemize}
\end{example}

\subsection{Balanced filtrations and parabolic subgroups} \label{balfilsec}
As above, let $G$ be a reductive group over a field $K$, and 
let:
\begin{itemize}
\item  $F$ be a filtration of $G$ associated with the parabolic subgroup $Q$, 
\item $P$ a parabolic subgroup of $G$, with Levi quotient $M$. 
\end{itemize}

We say that $F$ is {\em balanced} with respect to $P$
if $\wt(F) \in \mathfrak{a}_G$
is carried, under the embedding $\mathfrak{a}_G \rightarrow \mathfrak{a}_M$,
to the weight $\wt(F_M)$ of the filtration induced on the Levi quotient. 
Here $\mathfrak{a}_G \hookrightarrow \mathfrak{a}_M$ is as in \S \ref{sec:Levi}. 

\begin{example} Balanced filtrations.
\begin{itemize}
\item
If $G=\GL(V)$, and $P$ is associated to the increasing filtration
$\mathfrak{f}^q V$, then ``balanced'' says that, for every $q$, 
the filtration that $F$ induces on $\mathfrak{f}^q/\mathfrak{f}^{q+1}$
has the same weight as the filtration $F$ on $V$. 
\item 
The same assertion holds for $\mathrm{GAut}(V)$, 
where now $F$ and $\mathfrak{f}$ are self-dual filtrations.
\end{itemize}
\end{example}

Note that if we choose a cocharacter $\chi_P: \mathbf{G}_m \dashrightarrow P$
representing $F$, the condition of being ``balanced'' implies that, for any character $\psi$ of $P$ trivial on the center of $G$,
\begin{equation} \label{consequence} \langle \psi, \chi_P \rangle  = 0.\end{equation}

Now define
$$X(F) = \left \{ \mbox{$G$-conjugates of $F$ that are balanced with respect to $P$}\right \},$$
so that $X(F)$ is a $P$-stable subvariety of $G/Q$ and is equipped with a map 
\begin{equation} \label{fibered filtration}  X(F) \longrightarrow  \left \{ \mbox{filtrations of $M$} \right \}\end{equation}
via the rule $F \mapsto F_M$.  We may regard this, in an evident way,
as a ``constructible'' map between algebraic varieties (i.e.\ its graph is a constructible set)
and thus we can reasonably speak of dimension of fibers. 

We will analyze \eqref{fibered filtration} by breaking $X(F)$ into $P$-orbits.
Consider for a moment  $F \mapsto F_M$ as a map
\begin{equation} \label{NNN} \mbox{filtrations $P$-conjugate to $F$} \longrightarrow \mbox{filtrations on Levi quotient of $P$}\end{equation}
where both sides are   $P$-varieties. The left hand side  is identified with $P/(P \cap Q)$, and -- if we choose a maximal torus of $P \cap Q$
containing the image of a character defining $F$, and take the corresponding Levi decomposition $P=MU$ --  the image  is identified with  $M/(M \cap Q)$. 
From this and \eqref{factorization}  we find that each
fiber  of \eqref{NNN} has dimension    
\begin{equation} \label{Dimb} \dim(U) - \dim(Q \cap U)\end{equation}
 where $U$ is the unipotent radical of $P$, and $Q$ the stabilizer of $F$.

 \subsection{Double cosets of parabolic subgroups}
 
Fix, as before, $F$ a filtration of $G$ associated with the parabolic subgroup $Q$, and
$P$ a parabolic subgroup of $G$, with Levi quotient $M$. 
Continue with notation $X(F)$ as above.  We will be concerned with estimating
the size of the fibers of \eqref{fibered filtration}.

Fix a Borel $B$ contained in $P$  and a maximal torus $T \subset B$.
Since the variety $X(F)$ depends only on the $G$-orbit of $F$, we may harmlessly replace $F$ by a $G$-conjugate; 
in particular we may suppose that $F$ is defined by a co-character $\mu: \mathbf{G}_m \rightarrow T$
that is positive with respect to $B$, i.e.\ $B \subset Q$. 

Let $\Sigma \supset \Sigma^+$ be the set of roots of $T$ on $G$ and on $B$, respectively;
one therefore gets notions of simple and positive roots. 
Let $\Sigma_P, \Sigma_Q$ be the set of roots of $T$ on $P$ and $Q$. 
Therefore, $\Sigma^+ \subset \Sigma_P$ and $\Sigma^+ \subset \Sigma_Q$. 
Let $\Delta_P$ be the subset of simple roots $\alpha$ for which $-\alpha \in \Sigma_P$,
and similarly define $\Delta_Q$; thus $P$ and $Q$ correspond to 
the subsets $\Delta_P, \Delta_Q$ of the set of simple roots.
Note that, since  $\mu$ defines the parabolic subgroup $Q$, 
$\Sigma_Q$ is the set of roots having nonnegative pairing with $\mu$,
and in particular $\mu$ is orthogonal to all roots for $\Delta_Q$:
$$ \langle \mu, \beta \rangle =0,  \ \ \beta \in \Delta_Q.$$

Recall the ``adjoint'' Hodge numbers associated to $\mathrm{Lie} \ \ \mathrm{GAut} \left( \mathsf{V}_0 \otimes \C, \langle-, - \rangle\right)$, introduced in Section \ref{Hodge_numbers_funny}.   The following
proposition uses an abstraction of that notion: \begin{proposition} \label{lw2}
Let the ``Hodge numbers'' be the multi-set of integers of the form
 $\langle \mu, \gamma\rangle$  with $\gamma \in \Sigma$, adding multiplicity $\dim(T)$ to the multiplicity of zero. 
For $i \neq 0$ let $a_i$ be the number of roots $\gamma \in \Sigma$ with $\langle \mu,\gamma\rangle = i$,
so that $a_i$ is the multiplicity of $i$ as a Hodge number and
$\sum_{i > 0} a_i = \dim(G/Q)$;   %
we take $a_0$ the dimension of the Levi factor of $Q$. 

Suppose $e \leq \dim(G/Q)$ is a positive integer such that  
\begin{multline} \mbox{sum of all positive Hodge numbers} > \mbox{sum of top $e$ Hodge numbers} \\  + \mbox{sum of top $\left (\frac{a_0}{2}+e\right )$ Hodge numbers},\end{multline}
Then the codimension inside $G/Q$ of any fiber of the mapping 
\eqref{fibered filtration}$$X(F) \rightarrow \mbox{filtrations of $M$}$$
is greater than $e$. 
\end{proposition}

\proof  We are going to analyze this $P$-orbit by $P$-orbit.
Note that we have
$G = P W_{PQ} Q$, where  $W_{PQ}$ is the subset of the Weyl group $W$ defined 
via \begin{equation} \label{wpqdef} W_{PQ} = \{w \in W: \ w^{-1} \Delta_P  > 0, w\Delta_Q > 0\}. \end{equation}
Indeed, it is enough to see (by the Bruhat decomposition) that $W = W_P \cdot W_{PQ} \cdot W_Q$,
where $W_P$ and $W_Q$ are generated by simple reflections corresponding to $\Delta_P$ and $\Delta_Q$. 
Writing as usual
$ \ell(w) =  \# \{ \alpha > 0: w \alpha < 0\}$ for the length of a Weyl element, 
any minimal-length representative in a fixed double coset $W_P \cdot w \cdot W_Q$
belongs to $W_{PQ}$: 
for  $\alpha \in \Delta_P$, the element $s_{\alpha} w$ has shorter length than
$w$ if $w^{-1} \alpha < 0$. Similarly,  for $\beta \in \Delta_Q$, we know that $w s_{\beta}$ has shorter length than $w$
if $w \beta < 0$.

For each $w \in W_{PQ}$ we have either $PwQ/Q \subset X(F)$,
or $PwQ/Q \cap X(F) = \emptyset$. Call $w$ {\em bad} in the former case. 
For each bad $w \in W_{PQ}$ let $X(F)_w$ be the corresponding locally closed subvariety of $X(F)$,
i.e.\ $$X(F)_w = X(F) \cap  \left( (P wQ)/Q \right).$$  
 
Thus $X(F) = \coprod X(F)_w$, the union taken over bad $w$.  Assume, by way of contradiction, that there exists some bad $w$ such that
a fiber of
\begin{equation} \label{oink} X(F)_w \rightarrow \mbox{filtrations of $M$}\end{equation}
has codimension inside $G/Q$ that is $\leqslant e$.
 
That  $w$ is bad means that the filtration defined by the co-character $w\mu$ is balanced
with reference to $P$. 
This means in particular that 
\begin{equation} \label{balanced condition} \sum_{\gamma \in \Sigma - \Sigma_P}  \langle w  \mu, \gamma \rangle = 0.\end{equation}
In fact $\sum_{\gamma \in \Sigma - \Sigma_P} \gamma$
computes the modular character of the parabolic subgroup $P$: it is
the  negative of the character by which $P$ acts on the determinant of its unipotent radical,
and then use \eqref{consequence}.

For this (bad) $w$, write 
$$X =   \{\beta \in \Sigma - \Sigma_P: w^{-1} \beta > 0\} = \{\beta \in \Sigma-\Sigma_P: -w^{-1} \beta \in \Sigma-\Sigma_Q\}$$ 
(using Lemma \ref{root lemma}, see below) and let $X'$ be the complement of $X$
 inside $\Sigma-\Sigma_P$.
 
 Each fiber of \eqref{oink} 
has, by \eqref{Dimb},  dimension $$\dim(U) - \dim(\Ad(w) Q \cap U)
=     \# \{ \alpha \in \Sigma-\Sigma_P: - w^{-1} \alpha \in \Sigma - \Sigma_Q \} = \# X.$$
(see Lemma \ref{root lemma}).  This is equal to the length $\ell(w)$, although we won't make explicit use of it. 
Therefore our assumption means
$\#X \geqslant \dim(G/Q)-e$.  Then, since $\# X' = \dim(G/P) - \#X$, we have
 \begin{multline}  \#  X'
 \leqslant   \dim(G/P) - \dim(G/Q) + e \\ = \dim(Q)-\dim(P) + e \leqslant  \dim(Q/B) + e \leqslant   \frac{a_0}{2} +e.\end{multline}
 Also, by \eqref{balanced condition}, 
\begin{equation} \label{badbad} \sum_{\beta \in X} \langle \mu, w^{-1} \beta\rangle = \sum_{X'} - \langle   \mu, w^{-1} \beta \rangle.\end{equation}
All entries on the left hand side are {\em strictly} positive because $w^{-1} \beta$ is the negative
of an element of $\Sigma-\Sigma_Q$.  All entries on the right-hand side are non-negative
(because $B \subset Q$ the cocharacter $\mu$ is non-negative on positive roots.)
Now $X$ has size $\geqslant \dim(G/Q) -e$, so the image $-w^{-1}(X)$
omits at most $e$ roots inside $\Sigma-\Sigma_Q$. 
Therefore, the left-hand side of \eqref{badbad} 
is {\em at least}
$$ \mbox{sum of all positive Hodge numbers} - \mbox{sum of the topmost $e$ Hodge numbers.}$$
On the other hand, the right-hand side of \eqref{badbad} is {\em at most}
the sum of the top $(a_0/2+e)$ Hodge numbers.
(Here we have used that, since $e \leq \dim(G/Q)$, the top $e$ Hodge numbers are all positive.)   So we get a contradiction to \eqref{badbad} under the stated hypothesis.  \qed
  
We used the following Lemma:
  
\begin{lemma} \label{root lemma}
Let $\Sigma$ be the set of all roots, and take $w \in W_{PQ}$ (see \eqref{wpqdef}).   
\begin{itemize}
\item[(i)]
For $\beta \in \Sigma- \Sigma_Q$,
we have $w \beta > 0 \iff  -w\beta \in \Sigma-\Sigma_P$.
\item[(ii)] For $\alpha \in \Sigma-\Sigma_P$,
we have $w^{-1} \alpha > 0 \iff -w^{-1} \alpha \in \Sigma-\Sigma_Q$.
 
\item[(iii)] The map $x \mapsto -w(x)$ induces a bijection of these sets:
\begin{equation} \label{L1Lem} \{ \beta \in \Sigma - \Sigma_Q: w \beta > 0\} \longrightarrow \{\alpha \in \Sigma-\Sigma_P: w^{-1} \alpha >0\}\end{equation}
The size of this set is precisely the length $\ell(w)$. 

\end{itemize}

\end{lemma}
\proof 
Take $\beta \in \Sigma-\Sigma_Q$ with $w \beta > 0$. 
If $-w(\beta)$ were in $\Sigma_P$, 
then $\beta$ is a  positive linear combination of roots in $w^{-1} \Delta_P$,
contradicting the negativity of $\beta$. 

This shows the $\implies$ direction of (i) and the $\implies$ direction of (ii) is similar.
The reverse directions for (i) and (ii) are clear.
For example, if $-w\beta$ is in $\Sigma-\Sigma_P$, 
then $w \beta > 0$ because all roots in $\Sigma-\Sigma_P$ are negative.
Now it is clear that the maps $w$ and $w^{-1}$ give inverse bijections in \eqref{L1Lem}.
\qed

\subsection{Conclusion of the argument}
We now return to the situation  of Proposition \ref{linalg}. 
Let $G= \mathrm{GAut}(V, \langle -, - \rangle)$. 
 
We translate the problem into reductive group language.
Let $F_0$ be a fixed self-dual filtration on $V$; we will consider those filtrations $F$ that are conjugate to $F_0$
under $G$. 
Let $Q$ be the stabilizer of $F_0$ in $G$,
with Levi quotient $N$.   Reformulating Proposition \ref{linalg} (replacing $\mathfrak{f}$ from the Proposition with the parabolic subgroup
which is its stabilizer): 
we must estimate the codimension of $g \in G/Q$ 
such that, writing $F = gF_0$, there   exists another parabolic subgroup $P \leqslant G$ such that:

\begin{itemize}
\item[(a)'] (from property (a) of Proposition \ref{linalg}):   $\phi \in P$;

\item[(b)'] (from property (b) of Proposition \ref{linalg}): 
$F$ is balanced with respect to $P$, cf.\ the example of \S \ref{balfilsec}. 
\item[(c)']   (from property (c)  of Proposition \ref{linalg}):  The $G$-conjugacy class of $(P, \phi_M, F_M)$ is fixed, where $\phi_M$  is the projection of $\phi$ to the Levi quotient $M$ %
of the parabolic $P$.\footnote{We say here that $(P, \phi_M, F_M)$ is conjugate to $(P', \phi_{M'}, F_{M'})$ when there is $g \in G$ such that $\Ad(g) P = P'$, and
the induced isomorphism of Levi quotients carries $(\phi_M, F_M)$ to $(\phi_{M'}, F_{M'})$.}
 \end{itemize} 
 
First of all, we reduce to the case when $\phi$ is semisimple. 
Firstly,  $\phi \in P \implies \phi^{ss} \in P$ and, supposing that $\phi \in P$, then also
$(\phi^{ss})_M = (\phi_M)^{ss}$ (the subscript $M$ denotes projection to $M$).
Now if  $(P, \phi_M, F_M)$ and $(P', \phi_{M'}, F_{M'})$ are conjugate,
so that there is $g \in G$ with $\Ad(g) P=P'$ and $\Ad(g):M \rightarrow M'$
carries $\phi_M$ to $\phi_{M'}$, then $\Ad(g): M \rightarrow M'$ 
also carries $(\phi_M)^{\mathrm{ss}} = (\phi^{\mathrm{ss}})_M$
to $(\phi_{M'})^{\mathrm{ss}} = (\phi^{\mathrm{ss}})_{M'}$.
In other words, if we replace $\phi$ by $\phi^{\mathrm{ss}}$  then  the codimension of the set described above will only decrease. We do this, and can therefore assume that $\phi$ is semisimple. 

We will first show that
\begin{equation} \label{1db} (\mbox{dimension of possible pairs $(P, F_M)$}) \leq  z= \dim Z(\phi),\end{equation} 
the dimension of the centralizer of $\phi$ in $G$.  (Note that, because of our reduction above, 
$z$ corresponds to the dimension of the centralizer of $\phi^{ss}$,  for the original choice of $\phi$.)

The set of $P$ containing a given semisimple $\phi$
is a finite union of orbits of $Z(\phi)$, as we see by infinitesimal computations. 
It suffices, therefore, to examine a single $Z(\phi)$-orbit on the space of $P$.  
Fix $P_1$ in this orbit. The dimension of $Z(\phi)\cdot P_1$ 
equals
\begin{equation} \label{TT}\dim Z(\phi) - \dim Z_{P_1}(\phi)\end{equation}

Next, if we fix $P \in Z(\phi) \cdot P_1$, 
the collection of filtrations $\mathcal{F}$ on its Levi factor $M$
for which $(P, \phi_M,  \mathcal{F})$ belongs to 
a fixed $G$-isomorphism class corresponds to a finite
collection of orbits of $Z_M(\phi_M)$ on the space of filtrations on $M$. 
Now $\phi$ is $P$-conjugate to $\phi_M$ by \eqref{well known statement}  so that  $\dim Z_M(\phi_M) \leqslant  \dim Z_P(\phi)$. 
It follows that the dimension of the space of possible filtrations on $M$,
for $P$ fixed, is at most $\dim Z_P(\phi) = \dim Z_{P_1}(\phi)$. 
Adding this to \eqref{TT} we  deduce \eqref{1db}.

 We may now conclude the proof. Suppose $e$ is as in \eqref{equat}, so that both conditions are satisfied:
$$  \mbox{number of positive Hodge numbers} \geqslant z+e $$  
$$  \mbox{sum of all positive  Hodge numbers} >  T(z+e)   + \ T(\frac{h^0}{2}+z+e). $$
Recall that $X(F) \subset G/Q$ is the set of filtrations that are $G$-conjugate to $F$ and are balanced with respect to $P$;
 We may apply  Proposition \ref{lw2}, but taking the $e$ of that Proposition to be $z+e$ in the discussion above. (Note that the first displayed equation above guarantees, in the notation of Proposition \ref{lw2},
that $z+e \leqslant \dim(G/Q)$, as needed to apply it.) Thus,  if we fix $P$, the codimension inside $G/Q$ of any fiber of
$$   X(F) \rightarrow \mbox{filtrations on $M$}$$
is at least $z+e$. 

However, we saw above that the dimension of possibilities for $(P, \mbox{filtration on $M$})$
is at most $z$. 
Therefore, the total codimension of the set of $g \in G/Q$ satisfying (a)', (b)', (c)' 
is at least $e$, concluding the proof. 
\qed

\section{Bounding Frobenius via point counts}
\label{bound_frob}

We remark on an alternative approach to bounding the size of the Frobenius centralizer, 
i.e.\ the step that was achieved in the previous argument by Lemma \ref{Goodprimelemma}. 
It is likely that in some ranges this gives rise to better numerical bounds: %
 
\begin{lemma} \label{end lemma}
Let $Y$ be a smooth hypersurface of degree $d$ and dimension $n \geqslant 2$, defined over the finite field $k$
with $q$ elements;
let $b=  \dim H^n_{\mathrm{prim}}(Y_{\bar{k}}, \Q_{\ell}).$  Then the centralizer $Z$
of the semisimplified Frobenius, acting on $H^n_{\mathrm{prim}}(Y_{\bar{k}}, \Q_{\ell})$, has dimension at most $3b^2/N$,
where $N$ is the largest integer for which $q^{(n/2+1)N} < b/3$.
\end{lemma}
 
\proof
To avoid confusion between $i = \sqrt{-1}$ and as an index we write $e(\alpha) := \exp(2 \pi i \alpha)$. 

Let the Frobenius eigenvalues on  $H^n_{\mathrm{prim}}(Y_{\bar{k}}, \Q_{\ell})$ be given by
 $$\lambda_1=q^{n/2} e(\theta_1),  \dots, \lambda_b= q^{n/2} e (\theta_b),$$
 and  let $\mu$ be the measure on $S^1$
given by $\sum_{i=1}^b \delta_{\theta_i}$.  If the multiplicities of the $\theta_i$ are $m_1, \dots, m_r$, with $\sum m_i = b$, 
then $\dim Z = \sum m_i^2$. 

If $g$ is any  non-negative real-valued function on $S^1$ we have
$\int g(t-\theta) d\mu(\theta) = \sum_{s} g(t- \theta_s)$, and so 
\begin{equation} \label{gbound} \int_{t} dt \left| \int g(t-\theta) d\mu(\theta) \right|^2  \geqslant \mathrm{dim} Z \cdot \|g\|_{L^2}^2 \end{equation}
which bounds from above the dimension of the centralizer; this estimate is most effective if the support of $g$ is concentrated near $0$. 
Here, and in what follows, the measure is the Haar probability measure on $S^1$. 

If $k'$ is the field extension of $k$
of degree $j$, the number of points of $Y(k')$ is given by
$$ |Y(k')| =  \sum_{\ell=0}^{n} q^{\ell j}  + (-1)^n q^{nj/2} \sum_{s=1}^{b} e(j \theta_s).$$
 Since this lies between $0$ and the size of $\mathbf{P}^{n+1}(k')$,  i.e.\ between $0$ and $\sum_{\ell=0}^{n+1} q^{\ell j}$, 
we see that
\begin{equation} \label{pointbound} \left| \sum_{s} e(j  \theta_s) \right| \leqslant   q^{(n/2+1) j}.\end{equation}    Let 
$$g_N(t) = \left (\sum_{r = -N}^{N} e(rt) \right )^2 = \sum_{r = -2N}^{2N} (2N+1 - \left | r \right |) e(rt),$$ 
a function on $S^1$.  Note that $\|g_N\|_{L^2}^2  = (2N+1)^2 + 2 \sum_{i=1}^{2N} i^2$.  %
We have
$$  \int g_N(t - \theta) d\mu(\theta) = \sum_{r=-N}^N (2N+1 - \left | r \right |)  \sum_{s=1}^b e(r(t-\theta_s) ).$$
Using \eqref{pointbound}, we see that this is bounded in absolute value by
 $$  (2N+1) \left[ b  + 2   \sum_{r=1}^N q^{(n/2+1)|r|}  \right]\leqslant  (2N+1) \left( b + 3  q^{(n/2+1) N}\right)$$
since $q^{n/2+1} \geqslant 4$.   Therefore, by \eqref{gbound}, 
$$ \dim(Z) \leqslant   b^2 (1 + 3 b^{-1} q^{(n/2+1)N})^2 \cdot  \underbrace{
\left( \frac{(2N+1)^2}
{(2N+1)^2 + 2 \sum_{i=1}^{2N}i^2}
\right)}_{\leq \frac{3}{4N}}.$$
   Choose   $N$  the largest integer with $q^{(n/2+1)N} < b/3$; we get
$$ \dim(Z) \leqslant 3b^2/N.$$
\qed

\appendix
 
 \bibliography{padicTorelli-sept2019-arxiv}
\bibliographystyle{plain}

 \end{document}